\documentclass[a4paper]{article}
\usepackage[a4paper, total={6in, 8in}]{geometry}

\usepackage{lipsum}
\usepackage{amsfonts}
\usepackage{graphicx}
\usepackage{epstopdf}
\usepackage{algorithmic}
\usepackage{amsmath}
\usepackage{amsthm}
\usepackage{cleveref}
\ifpdf
\DeclareGraphicsExtensions{.eps,.pdf,.png,.jpg}
\else
\DeclareGraphicsExtensions{.eps}
\fi

\theoremstyle{plain}
\newtheorem{lemma}{Lemma}
\newtheorem{theorem}{Theorem}
\newtheorem{proposition}{Proposition}
\newtheorem{corollary}{Corollary}

\theoremstyle{definition}
\newtheorem{definition}{Definition}

\theoremstyle{remark}
\newtheorem*{remark}{Remark.}

\title{A Multiresolution approach to solve\\large-scale optimization problems}

\author{Rosa Donat\thanks{Universitat de Val\`encia, 50th Doctor Moliner street, Burjassot (Spain)
		(donat@uv.es, sergio.lopez-urena@uv.es). The authors have been supported by grant MTM2017-83942 funded by Spanish MINEC and by grant PID2020-117211GB-I00 funded by MCIN/AEI/10.13039/501100011033.}, Sergio L\'opez-Ure\~na$^*$}

\usepackage{amsopn}

\def\R{\mathbb{R}}
\newcommand{\mR}{\mathbb{R}}
\newcommand{\cD}{\mathcal{D}}
\DeclareMathOperator*{\argmin}{arg\ min}

\newcommand{\cC}{\mathcal{C}}

\newcommand{\z}{z}

\usepackage{graphicx,epstopdf}
\usepackage[caption=false]{subfig}

\usepackage[utf8]{inputenc} 

\usepackage{amsopn}
\usepackage{amssymb}

\begin{document}

\maketitle

\begin{abstract} 
General purpose optimization techniques can be used to solve many problems in engineering computations, although their cost is often prohibitive when the number of degrees of freedom is very large. We describe a multilevel approach to speed up the computation of the solution of a large-scale optimization problem by a given optimization technique. By embedding the problem within Harten's Multiresolution Framework (MRF), we set up a procedure that leads to the desired solution, after the computation of a finite sequence of {\em sub-optimal solutions}, which solve auxiliary optimization problems involving a smaller number of variables. For convex optimization problems having {\em smooth} solutions, we prove that the distance between the optimal solution and each sub-optimal approximation is related to the accuracy of the interpolation technique used within the MRF and analyze its relation with the performance of the proposed algorithm. Several numerical experiments confirm that our technique provides a computationally efficient strategy that allows the end user to treat both the optimizer and the objective function as black boxes throughout the optimization process.

\end{abstract}

%

\section{Introduction}
 Optimization problems  of the type
\begin{equation} \label{eq:problema}
\textrm{Find }\z_{\min} \in \mR^{N} \textrm{ such that } F(\z_{\min})=
\min_{\z \in \mR^{N}} F(\z), \qquad N>>1,
\end{equation}
 arise frequently in applications, often  in connection with the
 context of calculus of variations,  
PDE constrained optimization, optimal control \cite{Brandt77,Briggs87,Hackbusch80} or
image
processing \cite{CC06,CC10}.

In many occasions,  specially in  engineering computations,  it is
known, a priori, that the problem has a unique solution 
which  can  be computed (from a given initial data)  with an appropriate
`off-the-shelf' (black-box type) 
optimization technique which, however,  tends to be very time consuming.
Assuming that  problem \cref{eq:problema}  can be solved by using a
given  optimization technique, ${\cal D}$,  we describe a strategy to 
reduce the
computational time required by the direct application  of $\cD$ on a
space with $N$ degrees of freedom. Our technique is  based on
embedding problem \cref{eq:problema}  in Harten's  
Multiresolution Framework  (MRF henceforth) \cite{Harten96},  which  can be  loosely described as a  set of mathematical tools to
obtain multilevel representations of discrete data sets, similar to
those obtained in the well known wavelet framework. 

Harten's MRF  has been
successfully applied for different purposes and in several
scenarios (see e.g. 
\cite{DL-U17,ACG14,AMR13,CKMP03,CD01,BS99})  but, to
the best of our knowledge,  the applications  carried out in
\cite{DL-UM16,L-UT-LDG-A-C18} constitute the first attempt to  use
Harten's MRF in connection  with constrained/unconstrained
optimization.  The aim of this paper is to provide a complete mathematical
description of the technique used in \cite{DL-UM16,L-UT-LDG-A-C18} to
solve large-scale optimization 
problems of the type \cref{eq:problema}, together with some
theoretical results on its properties and  performance, including several  numerical
experiments that confirm our theoretical analysis. Our technique will be
referred throughout the paper as 
MR/OPT (for {\em Multiresolution for Optimization}).

\subsection{Related Work}

The use of multilevel techniques to solve large-scale optimization problems has
been a common trend in the literature since the mid 80's. 
Largely fueled by  the success of multigrid methods as  efficient solvers for
discretized partial differential 
equations (PDE) \cite{Brandt77,Briggs87,Hackbusch80})
and the fact that
many large-scale optimization 
problems arise from the discretization of a PDE
\cite{Nash00,CPT13,CWY18,FP14,FP15,GST08}, 
there are nowadays many multilevel
algorithms which exploit the fact that, if \cref{eq:problema} comes from an underlying  infinite dimensional problem,  it  may be described at several
discretization levels.

The  idea behind  many of these multilevel optimization techniques
is to combine an iterative (or descent) method with a correction which
requires the 
solution of a reduced, lower dimensional, auxiliary problem which is
cheaper to solve. For example, in the 
MG/OPT framework introduced by Nash in \cite{Nash00}, the solutions of
low-dimensional problems ({\em coarse resolution} problems) are used
to compute search directions for higher dimensional problems in a
multigrid fashion, an idea that has been considered by Frandi and
Papini in a series of papers \cite{FP14,FP15} where derivative-free multilevel
optimization techniques are considered.  Other   techniques, such as the {\em trust-region} multilevel optimization also make use of   auxiliary, lower
dimensional, optimization problems
(or {\em surrogate models} \cite{FP14,FP15})
which make an explicit use of the family of discrete
versions, at different resolution levels, associated to  the {infinite-dimensional/continuous} problem.

The present paper describes  a  different strategy to solve
\cref{eq:problema}, which is close, in spirit, to the 
so-called {\em `cost-effective'} multiscale strategy used in
\cite{CD01}, whose objective is to reduce the cost of  using a
time-consuming numerical scheme on a  fine-mesh
simulation involving hyperbolic PDEs.  Our aim is not to design a new optimization
algorithm, but rather to speed up the computation carried out 
 by a convenient (but possibly very time-consuming)  optimization technique, 
starting at a  given initial guess provided by the
end-user. 

In the  MR/OPT algorithm,  the function $F$ in
\cref{eq:problema} and the optimization tool are treated as black boxes, with the only requirement that ${\cal D}$ is appropriate  to solve \cref{eq:problema}  (except for  the possibly very large computational expense).
  Its design   involves the  
computation of  a sequence of iterates, which solve  {\em auxiliary optimization problems}
in $\R^m$ ($m<N$)  which do not make  use of  any 
discretization of the underlying infinite-dimensional problem at
lower  resolution levels (in contrast with
\cite{CPT13,FP14,FP15,GST08}). Somewhat related multilevel optimization strategies can be found in e.g.  \cite{DL-UM16,HZ14,MTRA17} (and references therein)  in the context of aerodynamic shape optimization, 
where it is now   recognized that a  shape optimization problem can be solved as a sequence of optimization steps that evolve from a basic parametrization of the shape to more refined parametrizations. These techniques are capable of  
 implementing  geometry changes while maintaining the required constraints and obtaining an improvement in the efficiency, accuracy, and robustness of widely used optimization procedures. As in the MR/OPT strategy, they are not based on discretizing the underlying infinite dimensional problem at several accuracy levels.

The paper is  
organized as follows: In \cref{sec:dmf} we briefly recall the main
ingredients of Harten's MRF required for the description of the MR/OPT
strategy, which is carefully described  in \cref{sec:MSO}. In \cref{sec:1D-int}, we
describe the
interpolatory MR setting, including the specific prediction schemes used in the  numerical
examples of \cref{sec:numex}, chosen to illustrate the performance and
limitations of our technique. We finish with some conclusions and perspectives in \cref{sec:conclusions}.

\section{Harten's Framework For Multiresolution} \label{sec:dmf}

In this section we shall recall  the essential ingredients of
Harten's MRF which are required for the description of our
algorithm (for a complete description of Harten's MRF see  e.g.
\cite{Harten96,AD00,ADH98a}).

A distinctive characteristic of Harten's approach  is that the discrete data
at different resolution levels in a specific MRF are
assumed to come from a particular {\em discretization} procedure, 
acting on an appropriate functional space, which  is
`naturally'  linked to a  hierarchy of {\em nested
  meshes}\footnote{$\{{\cal G}^k\}_k$ is nested if 
	${\cal G}^k \subset {\cal G}^{k+1}$} $({\cal
        G}^k)_{k=0}^L$ covering  a fixed   spatial
domain.

Assuming that $N_k$ is the dimension of  the discrete data attached to
${\cal G}^k$, where  $N_L>N_{L-1}> \cdots > N_0 $,
a multiresolution (MR) representation of a
discrete data set, $\z^L\in \R^{N_L}$, is an {\em  equivalent
	representation} (in $\R^{N_L}$)
which consists in 
a {\em coarse} representation of this data set,
$\z^0 \in \R^{N_0}$, together 
with a sequence of {\em details}, $d^\ell \in \R^{N_{\ell+1}- N_\ell}$, $\ell=0, \cdots, L-1$, that can be
understood as `essential' (non-redundant) difference in information between con\-se\-cu\-tive
resolution levels.

In a nutshell, Harten's MRF can be described as follows:
For  each pair of consecutive resolution levels, we are given a {\em
  decimation operator} $D_k^{k-1}$, which extracts  $k-1$ level
information from the 
discrete data at level $k$, and a {\em prediction operator} $P_{k-1}^k$
which produces $k$ level data from $k-1$ level
information. $D_k^{k-1}$ is a  linear operator, fully  determined by the underlying
discretization procedure that defines the framework.  $P_{k-1}^k$ may
be linear or nonlinear\footnote{This is a major difference between Harten's
  MRF and the classical wavelet theory.}. These two
operators are required to satisfy a {\em consistency relation}   at all resolution levels (I$_\ell$ is the identity operator on $\R^{N_\ell}$):
\begin{equation} \label{eq:consist-g}
D_k^{k-1} P_{k-1}^k  = \textrm{I}_{k-1},
\end{equation}
(but $ P_{k-1}^k D_k^{k-1} \neq \textrm{I}_{k}$). Notice that, for
$z^k \in \R^{N_k}$ we can write  
\[ z^k = P_{k-1}^k D_k^{k-1} z^k + (\textrm{I}_{k} - P_{k-1}^k
D_k^{k-1}) z^k =  P_{k-1}^k z^{k-1} + e^{k},\]
 hence each  $z^k \in \R^{N_k}$ can be {\em represented}
in terms of a `decimated' version (at the next, coarser, resolution
level)
together with a `prediction error', i.e.
\begin{equation} \label{eq:zk-ek}
\z^k \quad \equiv \quad \begin{cases}
z^{k-1}&\!\!\!:= D_{k}^{k-1} z^k \in \R^{N_{k-1}}, \\
e^{k}&\!\!\!:= z^k - P_{k-1}^k z^{k-1} \in \R^{N_k}. \end{cases}\end{equation}
The consistency relation \cref{eq:consist-g} implies that
$D_{k}^{k-1} e^{k}=0$, so that the
representation of $z^k$ in terms of  $( z^{k-1}, e^{k}) $ specified in
\cref{eq:zk-ek}  (with $N_{k-1}+N_k$ components)  is
redundant.  The {\em details}  $d^{k-1}$ represent  exactly the
non-redundant 
information contained in $e^k$, so that a
one-to-one correspondence between $z^k$ and  $(z^{k-1}, d^{k-1})$ can
be established (hence  $d^{k-1} \in \R^{N_k-N_{k-1}}$). By iterating this 
procedure from level $L$ to level $1$, we obtain the full MR
representation of $z^L\in \R^{N_L}$:
\begin{equation} \label{eq:full-MR-g}
z^L \leftrightarrow (z^{L-1},d^{L-1}) \leftrightarrow
(z^{L-2},d^{L-2},d^{L-1}) \leftrightarrow \cdots \leftrightarrow
(z^0, d^0, \cdots, d^{L-1}).
\end{equation}
In \cref{eq:full-MR-g}, $z^{k}=D_{k+1}^{k} \cdots D^{L-1}_L z^L \in \R^{N_k}$,   $L > k\geq 0$, i.e. $z^k$ are  coarse  versions
of $z^L$  at the different resolution  levels,  obtained by successive
decimation. Moreover, the details  $d^{k-1}$ are related to the prediction error $e^k=
(I- P_{k-1}^k D_{k}^{k-1} ) z^k$,
and  provide information on the local smoothness of the function which is
associated to the discrete data $z^L$ by the discretization operator
which defines the framework (see  also \cref{sec:1D-int}).

If we denote by  $M_{m,j}$ ($m<j$)  the
multiresolution transform 
o\-pe\-ra\-tor that acts between the resolution levels 
$m$ and $j$, then, for $0 \leq \ell \leq k \leq L$
and  $\forall z^L \in \R^{N_L}$,
\begin{equation} \label{eq:M0L-cases-0} 
M_k^Lz^L =(z^{k}, d^{k}, \cdots d^{L-1})  \quad \rightarrow \quad 
M_{\ell}^L z^L = (M_{\ell}^k z^k, d^{k}, \cdots d^{L-1}),
\end{equation}
or, equivalently, 
\begin{equation} \label{eq:M0L-cases}
M_{\ell,L} z^L = (z^\ell, d^\ell, \cdots,d^{k-1},d^{k}, \cdots ,
d^{L-1}) \,  \equiv \, 
\begin{cases}
M_{k,L} \, z^L= (z^{k}, d^{k}, \cdots d^{L-1}), \\
M_{\ell,k} z^{k} =(z^\ell,d^\ell, \cdots, d^{k-1}).
\end{cases}
\end{equation}
We shall use these relations in various places throughout the paper.

An important property for the practical use of   MR  transformations is their  {\em
  stability} with respect to {\em perturbations}.  In this paper, we shall assume that the prediction operators    $P_{k-1}^k:\R^{N_{k-1}}
        \rightarrow \R^{N_k}$  are linear (they can be 
	represented as $\R^{N_k \times N_{k-1}}$ matrices) which implies in turn that the
      multiscale transformations
         $M_{m,j} : \R^{N_j} \rightarrow
        \R^{N_j}$ are also a linear  operators ($\R^{N_j \times N_j}$
        matrices). In this case, the stability of the   associated     MRF  can be stated as
        follows 
\begin{equation} \label{eq:M-stab}
||M_{m,j}||_\infty \leq C, \qquad ||M_{m,j}^{-1}||_\infty \leq
\tilde C,
\end{equation} with $C, \tilde C$ independent of $ m,j$ or  $L$, for  $0\leq m \leq
j \leq L$, and $||\cdot||_\infty$ the usual operator norm.

Within Harten's framework,  stability is  related to the  {\em convergence} of the process that
results from the recursive application of the  prediction
operators, which involves a refinement process  akin to those found
in the theory of subdivision refinement\footnote{The convergence theory of
these processes is outside the scope of this paper. The reader is referred to
\cite{Dyn92} for the concept of 
convergence in subdivision refinement.}. 
The prediction operators considered in this paper satisfy a  property which  guarantees the stability of the associated MR transformations. This property, which we shall denote as {\em Property S}, is defined below. 
\begin{definition}[Property S] Let {${\cal P}=\{P_k^{k+1}\}_{k=0}^\infty$} be a
  sequence of linear prediction operators. We will say that ${\cal P}$
  satisfies {\em property S} if 
  there exist  $d_1, d_2>0$, such that for any $0\leq \ell<k$
\begin{equation} \label{eq:d1-d2}
 d_1|| z^\ell ||_\infty \leq || P_\ell^k z^\ell ||_\infty \leq d_2
 ||z^\ell||_\infty, \qquad \forall z^\ell\in\R^{N_\ell},
\end{equation}
where $P_\ell^k := P_{k-1}^k \cdots P_{\ell+1}^{\ell + 2}  P_\ell^{\ell + 1} $ is the operator that computes the recursive prediction process between levels $l < k$.
\end{definition}

\begin{remark} We shall assume that $P_{k}^k = I_k$. Notice that  
\begin{equation} \label{eq:Sk-M0k}
z^L= P_k^L z^k \quad \Leftrightarrow  \quad M_{k,L} z^L= (z^k,0, \ldots, 0).
\end{equation}
\end{remark}

\subsection{Harten's  Interpolatory MRF} \label{sec:1D-int}

In this paper, we shall use Harten's  interpolatory MR framework, which
is characterized by the assumption that the discrete data correspond
to the point-evaluations  of a function, defined on the   spatial
domain underlying the nested mesh structure. 

We briefly describe next the one-dimensional case:
Let us consider the sequence of nested meshes on the interval $[0,1]$
obtained by the dyadic refinement of a uniform mesh ${\cal G}^0$, with grid spacing
$h_0=1/J_0$. Thus, ${\cal G}^k=\{ i
h_k\}_{i=0}^{J_k}$ with $h_k=h_0/2^k=1/J_k$ and  the $k$-th level information consists in
$N_k=J_k+1$ discrete values, for each $k$, $0 \leq k \leq L$.

The operators that define 
the transfer of information between two consecutive  resolution levels
($k$ fine, $k-1$ coarse) in the 1D interpolatory MRF are  as follows: \\[2pt]
\noindent $\bullet$  {\em Decimation} by {\em downsampling}:  For $ z^k  \in
\R^{N_k}$,
\begin{equation} \label{eq:deci-def} z^{k-1}_i = (D_{k}^{k-1}
z^k)_{i}=z^k_{2i}, \quad 0 \leq i \leq J_{k-1}. 
\end{equation}
Notice that this decimation operator corresponds to the interpretation
of the discrete data as the  
point-values of an underlying function defined on  the unit interval
$[0,1]$ at each resolution level. 
Obviously $z^{k-1} \in
\R^{N_{k-1}}$ is a {\em	coarse} version of the  data $z^k$, which  corresponds to the
point-values of the function on the  grid ${\cal
	G}^{k-1}$. \\[2pt]
\noindent $\bullet$  {\em Prediction} by {\em interpolation}.
Because of \cref{eq:deci-def}, the 
consistency relation \cref{eq:consist-g} becomes 
\[  z^{k-1}_i= (D_k^{k-1} P_{k-1}^k z^{k-1})_i = (P_{k-1}^k z^{k-1})_{2i}.\]
This  {\em  interpolation
condition} implies that the prediction errors at even points are zero,
thus, only
the predicted values at odd grid points need to be 
computed and stored. 
Then, the   one-to-one correspondence between the discrete
sets $z^k$ and $(z^{k-1},d^{k-1})$, i.e. the two-level multiresolution
transformation $M_{k-1,k}$ and its inverse, are
\begin{equation} \label{eq:M-k-1-k}
M_{k-1,k} z^k = (z^{k-1}, d^{k-1}) \ \leftrightarrow 
\begin{cases} z^{k-1}_{i}\!\!\!\!\!\!&= z^k_{2i}, 
\, \, 0 \leq i \leq J_{k-1},\\
d_i^{k-1}\!\!\!\!\!\!&=z^k_{2i-1} - (P_{k-1}^k z^{k-1})_{2i-1},
\, \, 1 \leq i \leq J_{k-1}, \end{cases}
\end{equation}
\begin{equation} \label{eq:M-k-1-k-inv}
M_{k-1,k}^{-1} (z^{k-1}, d^k) = z^k \  \leftrightarrow 
\begin{cases}
z^k_{2i}\!\!\!\!\!\!&= z^{k-1}_{i},   \, \, 0 \leq i \leq J_{k-1},\\
z^k_{2i-1}\!\!\!\!\!\!&= d_i^{k-1}+ (P_{k-1}^k z^{k-1})_{2i-1}, \, \, 1 \leq i \leq J_{k-1}.
\end{cases}
\end{equation}
\begin{remark} \label{rem:d-zL}
From \cref{eq:M-k-1-k},  since $z^{k-1}= D_{k}^{k-1} z^k$, $z^k= D_{k+1}^k
\cdots D_L^{L-1} z^L$, we can also write $d_i^{k-1}=d_i^{k-1}(z^L)$, which
clearly express that the detail coefficients are determined by the
properties of the discrete data at the finest resolution level.
 \end{remark}

In this paper, we shall consider prediction operators based on  certain piecewise
polynomial interpolatory techniques of fixed polynomial degree, $n$,
odd. We refer the reader to \cite{ADH98a, DGH03} for specific details,
and simply provide the   prediction rules for $n=1,3,5$
\begin{equation} \label{eq:Pkdef}
\begin{cases}
(P_{k-1}^k z^{k-1})_{2i-1}\!\!\!&=\sum_{\ell=1}^{(n+1)/2} \beta_\ell
(z^{k-1}_{i+\ell}+ z^{k-1}_{i+1-\ell}), \quad \frac{n+1}2 \leq i \leq J_{k-1}-\frac{n-1}2, \\
(P_{k-1}^k z^{k-1})_{2i}\!\!\!&=z^{k-1}_i, \quad 0 \leq i
\leq J_{k-1},
\end{cases}
\end{equation}
\begin{equation} \label{eq:n-coeff}
\begin{cases} n=1 \Rightarrow & \beta_1=1/2, \\
n=3 \Rightarrow & \beta_1=9/16, \, \beta_2=-1/16, \\
n=5 \Rightarrow & \beta_1=150/256, \, \beta_2=-25/256, \, \beta_3=
3/256, \end{cases} 
\end{equation}
with  the
following  {\em special boundary  rules} at the left boundary of $[0,1]$ for $n=3$,
\begin{equation}  \label{eq:b-predn3}
(P_{k-1}^k z^{k-1})_{1}= \frac{5}{16} z^{k-1}_0 + \frac{15}{16} z^{k-1}_1
-\frac{5}{16}z^{k-1}_2+ \frac1{16} z^{k-1}_3,
\end{equation}
and  $n=5$,
\begin{equation}  \label{eq:b-predn5} 
\! \begin{cases}
\!\!(P_{k-1}^k z^{k-1})_{1}\!\!\!\!\!\!&=
\frac{63}{256}z^{k-1}_0\!+\frac{315}{256}z^{k-1}_1\!-\!\frac{105}{128}z^{k-1}_2\!+\!\frac{63}{128}z^{k-1}_3\!-\!\frac{45}{256}z^{k-1}_4\!+\!\frac{7}{256}z^{k-1}_5,
\\[6pt]  
\!\!(P_{k-1}^k
z^{k-1})_{3}\!\!\!\!\!\!&=-\frac{7}{256}z^{k-1}_0\!+\!\frac{105}{256}z^{k-1}_1\!+\!\frac{105}{128}z^{k-1}_2\!-\!\frac{35}{128}z^{k-1}_3\!+\!\frac{21}{256}z^{k-1}_4\!-\!\frac{3}{256}z^{k-1}_5.
\end{cases}
\end{equation}
 The  corresponding rules at the right end of $[0,1]$  can be obtained
 by symmetry. We prove in the Appendix that these prediction operators
 satisfy {\em property S} with $d_1=1$, hence, the associated MR transformations are stable, i.e. satisfy \cref{eq:M-stab}.

\begin{remark} The  prediction operators defined by \cref{eq:Pkdef} are  related to the  well known (stationary\footnote{Stationary subdivision rules are
	independent of the refinement level.})  subdivision rules of the interpolatory Deslauries-Dubuc
subdivision schemes \cite{Dyn92,DGH03} for infinite sequences. To deal with bounded domains, we have to take care of the, possibly special, prediction rules close to the boundaries, where the required centered stencils might leave the domain.  For $n=1$, i.e. linear
interpolation, only   the boundary values  $z^{k-1}_0, z^{k-1}_{N_{k-1}}$ are required to
apply the prediction operator  specified in \cref{eq:Pkdef} at each
resolution level, hence no special rules are required close to the boundaries of the domain. However, for  $n=3$ (or $n=5$) interpolation polynomials of degree 3 (or 5) are
involved, hence special rules are required to predict
$k$-data from  data from $(k-1)$-th grid values at the positions
$i=1$ and $i=J_{k-1}-1$  (or $i=1,2$, $i=J_{k-1}-1, J_{k-1}-2
	$ for $n=5$). To obtain \cref{eq:b-predn3,eq:b-predn5}, we
        have used the stencil of the $n+1$ points closest to the
boundary, in order to obtain the required predicted value
\cite{ADH98a}.  Hence, all polynomial pieces in the interpolatory
reconstruction process have the same degree (and the same
approximation order).
\end{remark}

The decay of the scale coefficients with respect to the resolution
 level under smoothness assumptions, is a well known fact that is
 very simple to explain in the interpolatory MRF.  Since the details
are interpolation errors, their size (and behavior across
scales) is determined by the approximation order of the interpolatory technique
used to define the prediction operators.
  The rules  in   \cref{eq:Pkdef,eq:n-coeff,eq:b-predn3,eq:b-predn5},
  involve polynomials of degree $n$, hence
\begin{equation} \label{eq:dk-interp}
 d^{k-1}_i = {\cal O}(h_{k}^{n+1} ).
\end{equation}
when the  data  around the spatial location corresponding
to the indices {$(2i-1,k)$} correspond to the point values of a
sufficiently smooth function.

In this paper, we shall make use of the 1D interpolatory MRF just
described and its 2D version obtained by using a tensor-product
approach \cite{BH97}. When
using the 1D prediction operators specified above in the 2D
tensor-product context, the prediction operators have the same
order of approximation, with respect to the
uniform mesh spacing of the (tensor product) hierarchical mesh
structure, as in the 1D setting. Hence, the decay of the scale coefficients in the 2D tensor-product interpolatory MRF is also given by \cref{eq:dk-interp}.

\section{ MR/OPT: A multiresolution approach for the efficient solution
  of  large-scale  optimization problems} 
\label{sec:MSO}

We shall describe next a multilevel strategy to compute the solution of
\cref{eq:problema}, assuming that the end-user provides an  initial
guess $\bar z$ and a preferred  optimization tool ${\cal D}$ so that
\begin{equation} \label{eq:problema-MR}
\quad  z_{\min} = \cD(\bar z, F, N)= \argmin\{ F(z), \, \, \,  z \in \R^{N}\}.
\end{equation}
The notation above
is meant to emphasize the fact 
that the optimization problem \cref{eq:problema},  with  $F$ as the objective
function on a  discrete  space with  $N$ degrees of freedom,  can be
solved using $\cal D$ as  optimization  tool  and  $\bar z$ as
the initial guess. 

Considering $N=N_L$ as the finest resolution level in a
sequence of $L+1$ levels of refinement 
corresponding to discrete spaces $\R^{N_k}$ (
$N_0<N_1<\cdots<N_L=N$, $L>0$) within   a MRF which is {\em appropriate} for the
given problem, the MR/OPT strategy   produces  a 
sequence of {\em `sub-optimal' solutions},  $\{z^{L,k}\}_{k=0}^{L+1}$,
satisfying 
\begin{equation}  \label{eq:Dalphas2}
z^{L,0} := \bar z, \quad F(\bar z)=F(z^{L,0}) \geq F(z^{L,1}) \geq \cdots 
\geq F(z^{L,L+1}), \quad  z^{L,L+1}= z_{\min}. 
\end{equation}
The $k$-th level sub-optimal solution satisfies 
\[ F(z^{L,k+1}) = \min \{ F(z), \, z  \in \Xi_k \} \]
where $\Xi_k $ is a subset 
of $\R^N$, whose construction depends on
the MRF being used and on the sequence of prediction operators ${\cal
  P}$. 

In this paper we always assume that the  prediction operators
in ${\cal P}$ are linear and satisfy {\em property S}. Hence the MR
transformations used are  also linear operators and \cref{eq:M-stab}
is satisfied. 

At the coarsest resolution level ($k=0$) we
set $z^{L,0}:=\bar z$, the  initial guess 
provided by the user. Its  MR representation, $ M_{0,L} z^{L,0}=M_{0,L} \bar z=:
(\bar z^0,\bar d^0, \ldots, \bar d^{L-1})$, serves to define the set
$\Xi_0\subset \R^{N_L}$ and the  associated 
$0$-th-level {\em objective function}
as  follows
\begin{align*}
\Xi_0 :=&\{ M_{0,L}^{-1} (\bar z^0+ \varepsilon^0, \bar d^0, \ldots, \bar
d^{L-1}), \, \,  \varepsilon^0 \in R^{N_0}\}, \\
F_0(\varepsilon^0):=&F(M_{0,L}^{-1}(\bar
z^0+ \varepsilon^0, \bar d^0, \bar d^1, \ldots, \bar d^{L-1})), &
\varepsilon^0 \in \R^{N_0}. 
\end{align*} 
Since  the MR operators are linear, taking into account
\cref{eq:Sk-M0k}, we can write
\[ M_{0,L}^{-1}(\bar z^0 + \varepsilon^0, \bar d^0, \bar
d^1, \ldots, \bar d^{L-1})=
\bar z + P_0^L \varepsilon^0. \]
Hence 
$\Xi_0$ is an affine space  with
$N_0$ degrees of freedom, which 
 is   defined by  considering
only  perturbations (of the  MR representation) of the 
initial data at the coarsest resolution level. These perturbations  are moved to
$\R^{N_L}$
by successive prediction. 

 We then
compute (assuming it is feasible\footnote{We shall see that this is
  indeed the case for convex optimization problems.}) 
\begin{equation} \label{eq:0-aux}
\varepsilon^0_* := \argmin \{ F_0(\varepsilon^0), \quad \varepsilon^0
\in \R^{N_0}\} = {\cal D}(0,F_0, {N_0}),
\end{equation}
and define the  {\em sub-optimal solution at level $0$}  as
\begin{equation} \label{eq:zL1-def}
z^{L,1}:=\bar z+ P_0^L \varepsilon^0_* =  z^{L,0}+ P_0^L\varepsilon^0_* .
\end{equation}
 Obviously
\[z^{L,1}= \argmin\{ F(z), \, \, z \in  \Xi_0 \}, \quad F(z^{L,1}) = F_0(\varepsilon^0_*) \leq F_0(0) = F(z^{L,0})= F(\bar z).\]
  Notice that the computation of $z^{L,1}$ can be carried out involving
only direct calls to the  objective function $F$, since $ F_0(\varepsilon^0)=F(\bar z+ P_0^L
\varepsilon^0)$. {Notice that the user only needs to provide $F$, and the MRF takes care of  the definition of   the  objective function $F_0$.} In addition, 
  considering $\varepsilon^0= 0 \in \R^{N_0}$ as the initial guess to
solve  \cref{eq:0-aux} (the {\em $0$-th level auxiliary optimization problem})
corresponds to considering $z^{L,0}=\bar z$ as the `corresponding' initial guess  in the
affine space $\Xi_0 \subset \R^{N}$. 
Notice also that 
\[ M_{0,L} z^{L,1}=
M_{0,L} \bar z + M_{0,L} P_0^L \varepsilon^0_*=
(\bar z^0+ \varepsilon^0_*, \bar d^0,\cdots, \bar d^{L-1}), \]
 i.e.
the difference between $z^{L,0}$ and  $z^{L,1}$  occurs only at
the coarsest resolution level, while  the `details'  in their MR representations
coincide  at all higher resolution levels. 
Taking into account  \cref{eq:M0L-cases-0}, we can write
\[ M_{1,L} z^{L,1}= (z^1_*, \bar d^1, \cdots,\bar d^{L}), \quad
M_{0,1} z^1_* = (\bar z_0+\varepsilon^*_0,\bar d^0) \, \equiv
\, z^1_*= M_{0,1}^{-1}(\bar  z_0,\bar d^0) + P_0^1 \varepsilon^0_* \] 
so that  we can repeat the process as if the level 1 was the coarsest resolution level. 

The general $k$-th step
of the algorithm,  $1 \leq k \leq L$, which will provide the {\em
  sub-optimal solution} $z^{L,k+1}$,  can be described as follows:  We assume that we have obtained
$\varepsilon_*^{m} \in \R^{N_m}$, $0 \leq m \leq k-1$ (the solution of
the $m$-th level auxiliary optimization  problem) and the associated  sub-optimal
solutions, $\{z^{L,m}\}_{m=1}^k$, which satisfy the
following properties  ($ \bar z=z^{L,0}  $)
\[z^{L,m}=z^{L,m-1}+P_{m-1}^L \varepsilon_*^{m-1}; \quad  M_{m,L}
z^{L,m}=(z^m_*,\bar d^{m},\ldots, \bar d^{L-1}), \quad  1 \leq m \leq  k, \, \, \]
 where {\(z^0_*=\bar z^0\) and } 
\[  \begin{cases} 
 z^m_* \! \! &= M_{m-1,m}^{-1}(z_*^{m-1}, \bar d^{m-1})
+ P_{m-1}^m \varepsilon^{m-1}_*, \\
\bar d^j \! \! &=d^j(\bar z), \, \, m-1 \leq j \leq L-1.  \end{cases}\]
Then, we proceed as follows: \\[3pt]
\noindent{\bf 1.} Consider the $k$-th level  space of approximation
	\begin{equation} \label{eq:Xik-def}
 \Xi_k:=\{ M_{k,L}^{-1}(z^k_*+ \varepsilon^k, \bar d^k, \cdots, \bar
	d^{L-1}), \, \, \varepsilon^k \in \R^{N_k}\}=
	z^{L,k} +\{  P_k^L\varepsilon^k, \, \, \varepsilon^k \in   \R^{N_k} \} 
\end{equation}
	which is an affine space in $\R^N$  with  ${N_k}$ degrees of
        freedom.

\noindent{\bf 2.} Define the $k$-th level  objective function
	$F_k: \R^{N_k}
	\rightarrow \R$, 
	\begin{equation} \label{eq:aux-Fk}
	F_k(\varepsilon^k):= 
	F(M_{k,L}^{-1}(z^k_*+\varepsilon^k, \bar d^{k},\ldots, \bar
	d^{L-1}))=F(z^{L,k}+ P_k^L \varepsilon^k).
	\end{equation}
{Notice that only $F$ and ${\cal P}=\{P_k^{k+1}\}_k$ is required to define $F_k$.}

\noindent{\bf 3.} Assuming that it is feasible,  compute the solution to the $k$-th
	level {\em auxiliary optimization problem}
	\begin{equation} \label{eq:alphaLk}
	\varepsilon^k_* := \argmin \{ F_k(\varepsilon^k),\, \, \, \varepsilon^k
	\in \R^{N_k}\} = {\cal D}(0,  F_k, N_k).
	\end{equation}
\noindent{\bf 4. } Define
	\begin{equation} \label{eq:k-auxP}
	\quad  z^{L,k+1}:=z^{L,k}+ P_k^L \varepsilon^k_* .
	\end{equation}
Then, $ z^{L,k+1}=\argmin\{ F(z), \, \, z \in
\Xi_k \}$, $ F(z^{L,k+1}) =F_k(\varepsilon^k_*) \leq F_k(0)  =
F(z^{L,k})$ and 
\begin{equation} \label{eq:MkL-zLk1}
 M_{k,L}z^{L,k+1}= M_{k,L}z^{L,k}+ M_{k,L} P_k^L \varepsilon^k_*=
(z^k_*+\varepsilon^k_*,  \bar d^{k}, \ldots,\bar d^{L-1})  
\end{equation}
hence
\begin{align}
&d^i(z^{L,k+1})=\bar d^i= d^i(\bar z), \   k \leq i\leq L-1,  \label{eq:dk-equal}\\ 
&M_{k+1,L} z^{L,k+1} =(z^{k+1}_*, \bar d^{k+1}, \cdots,
\bar d^{L-1}), \quad z^{k+1}_* :=M_{k,k+1}^{-1}(z^k_*, \bar d^k)
+ P_k^{k+1} \varepsilon^k_*  . \label{eq:Mk-1} \end{align}

We can easily prove the following properties of  the affine spaces
$\Xi_k$.
\begin{lemma} \label{lemma:contained} The spaces $\Xi_k$ in
  \cref{eq:Xik-def} satisfy
\begin{enumerate}
\item 	$ \Xi_{k} \subset \Xi_{k+1}, \quad  0\leq  k < L.$ \label{lem:xi-props1}
\item $ \Xi_{k} = \{z^{L,\ell} + P_k^L\varepsilon^k \ : \
  \varepsilon^k\in\R^{N_k}\}, \quad  \forall \ell, 0\leq\ell\leq k\leq
  L.$ \label{lem:xi-props2}
\end{enumerate}
\end{lemma}
\begin{proof}
	
The $\Xi_k$  are affine spaces, hence to prove \cref{lem:xi-props1} it is enough to check
that  $P_k^{L}(\R^{N_k}) \subset P_{k+1}^{L}(\R^{N_{k+1}})$ and
$z^{L,k}\in \Xi_{k+1}$.
For this, we simply notice that $\forall \varepsilon_k \in \R^{N_k}$
\[ P_k^L \varepsilon^k = P_{k+1}^L P_k^{k+1} \varepsilon^k,
 \quad 
 z^{L,k} = z^{L,k+1}  -P_{k+1}^LP_{k}^{k+1} \varepsilon^k_*, \quad P_k^{k+1} \varepsilon^k \in
	\R^{N_{k+1}}.\]
Since $z^{L,\ell}\in \Xi_\ell \subset \Xi_k$,  \cref{lem:xi-props2}
follows immediately.
\end{proof}
\begin{remark} 
	By construction,   $\Xi_L= \R^{N_L}= \R^N$, hence
	$z^{L,L+1}=z_{\min}$ and the relations in  \cref{eq:Dalphas2} are  obviously satisfied.
\end{remark}

The MR/OPT strategy  substitutes
the  direct computation of the original (large-scale) optimization problem,
by the computation of the solutions of a
sequence of auxiliary optimization problems, each one of them
associated to a level of refinement in a multiresolution ladder.  It
is reasonable to assume, when these auxiliary optimization problems
admit a solution,   it will be obtained  {\em quite fast}
 at low resolution levels, due to the
reduced number of degrees of freedom in the spaces where they are
defined. In addition, it is expected that 
if the distance between the
sub-optimal solutions and the true solution gets smaller while climbing
up the MR ladder, then, even though each {\em auxiliary } minimization 
problem in the MR ladder involves an increasing number
of variables, it is also expected that  the improved
initial guess  (the previous sub-optimal solution) chosen to carry
out the optimization process will improve the performance of the
optimizer. In this case,  the MR/OPT strategy leads to an overall gain
in functional evaluations.

In the following section, we shall see that, in some cases, it is possible to ensure
that  the auxiliary  optimization problems  are of the same
type  as the original one, hence their solution can be computed by the
same optimization technique, ${\cal D}$. In addition, we shall see
that   the previous  observations on the performance of the MR/OPT
strategy can be theoretically  justified in some cases. 

\subsection{Some theoretical results on MR/OPT} \label{sec:theory}
We have assumed that the auxiliary optimization problems can be solved
with the same optimization tool as the full problem \cref{eq:problema-MR}. It is
simple to see that this is the case
for quadratic minimization problems of the type
\begin{equation} \label{numex:ODE-optP}
\text{Find }\z_{\min} \in \mR^{N} \text{ such that } F(\z_{\min})=
\min_{\z \in \mR^{N}} F(\z), \quad F(z)=\frac12 z^T A z - b^T z +c.
\end{equation}
with $A$  a symmetric and positive definite matrix.

\begin{proposition} \label{prop:F-quad} Let us consider the  quadratic minimization problem 
	\cref{numex:ODE-optP}, with $A$ symmetric and positive
	definite. If the prediction operators in ${\cal P}$ satisfy
        \cref{eq:d1-d2},  then the auxiliary minimization problem at level $k$
	\cref{eq:alphaLk}  is also a quadratic minimization problem
	of the same type.
\end{proposition}
\begin{proof} Using the definition of $F$ in \cref{numex:ODE-optP}, we have
	\begin{align*} F_k(\varepsilon^k) &=
          F(z^{L,k}+P_k^L\varepsilon^k) =
	\frac12 (z^{L,k}+P_k^L\varepsilon^k)^T A
	(z^{L,k}+P_k^L\varepsilon^k) - b^T (z^{L,k}+P_k^L\varepsilon^k)+c \\
	&=F(z^{L,k}) + \frac12 (P_k^L\varepsilon^k)^T A
	P_k^L\varepsilon^k + (z^{L,k})^T A	P_k^L\varepsilon^k - b^T P_k^L\varepsilon^k \\
	&= \frac12 (\varepsilon^k)^T  (P_k^L)^T A P_k^L  \varepsilon^k 
- ((P_k^L)^T (b-Az^{L,k}))^T \varepsilon^k +F(z^{L,k}).
	\end{align*}
Hence,
	\begin{equation} \label{eq:auxP-A}
	F_k(\varepsilon^k)
	= \frac12 (\varepsilon^k)^TA_k\varepsilon^k - b_k^T
	\varepsilon^k + c_k, \quad \begin{cases} 
	A_k\!\!\!\!&=(P_k^L)^T A P_k^L \in \R^{N_k\times N_k}, \\
	b_k\!\!\!\!&=(P_k^L)^T (b-Az^{L,k}) \in \R^{N_k},\\
	c_k &= F(z^{L,k}) \in \R. \end{cases}
	\end{equation}
	Obviously  $A_k$ is symmetric. If $A$  is positive definite,
        then $A_k$ is   also positive definite, since 
	$P_k^L$ is injective, by \cref{eq:d1-d2}:
	\[ \varepsilon^k\neq 0 \quad \longleftrightarrow \quad
	P_k^L\varepsilon^k\neq0 \quad \longleftrightarrow \quad 0 < (P_k^L
	\varepsilon^k)^T A (P_k^L\varepsilon^k) = (\varepsilon^k)^T A_k
	\varepsilon^k.\]
\end{proof}

These quadratic optimization problems  are special cases
of a larger class of  objec\-tive functions for which the same
property holds: The auxiliary optimization problems are of the same
type as the original one. 
\begin{proposition} \label{prop:propF-Fk} Let $F: \R^{N_L} \rightarrow \R$ be.  If the
  prediction operators in ${\cal P}$ satisfy  {\em property S},
the auxiliary functions $F_k: \R^{N_k} \rightarrow \R$ defined in
	\cref{eq:aux-Fk} satisfy:
	\begin{enumerate}
		\item If $F$ is convex and/or $F \in {\cal
                    C}^2(\R^{N_L},\R)$, then  $F_k$ is also convex
                  and/or $F_k \in {\cal C}^2(\R^{N_k},\R)$.
		\label{it:convex-smooth}
	\item  If  the hessian matrix $\nabla^2 F(\xi^L)$ is a positive definite
                  matrix $\forall \xi^L \in \R^{N_L}$, then $\nabla^2
                  F_k(\xi^k)$ is  a positive definite matrix $\forall
                  \xi^k \in \R^{N_k}$. \label{it:hessiank}
		\item If $F$ is coercive, i.e. {$\lim_{||z^L||_\infty
                      \rightarrow \infty}
F(z^L) = + \infty$}, then $F_k$ is coercive. \label{it:coercive}
	
	\end{enumerate}
\end{proposition} 
\begin{proof}
Since $P_k^L$ is a linear operator,	the proof of \cref{it:convex-smooth} follows
easily  from  the expression of $F_k$ in \cref{eq:aux-Fk}. To  prove
\cref{it:hessiank},  let $\xi^k \in \R^{N_k}$ be and notice that $\forall \varepsilon^k \neq 0$,
	\[ (\varepsilon^k)^T\nabla^2 F_k(\xi^k)\varepsilon^k = (P_k^L
        \varepsilon^k)^T \nabla^2F(z^{L,k}+ P_k^L \xi^k)
        (P_k^L\varepsilon^k) >0,\] 
  since $\varepsilon^k\neq 0 \, \rightarrow  \,
  P_k^L\varepsilon^k\neq0$, by \cref{eq:d1-d2}.

To prove \cref{it:coercive}, let  $0<\eta \in \R$. Since $F$ 
is coercive,  $\exists
\delta=\delta(F,\eta)$ such that $F(z^L)>\eta$, $\forall z^L \in
\R^{N_L}$:  $||z^L||_\infty
>\delta$. Define $\delta':=d_1^{-1}(\delta + ||z^{L,k}||_\infty)$, with
$d_1$ in \cref{eq:d1-d2}. Then $\forall \varepsilon^k : $
$||\varepsilon^k||_\infty >\delta'$ we have 
\[ \|z^{L,k} + P_k^L \varepsilon^k\|_\infty \geq \|P_k^L
\varepsilon^k\|_\infty  -
\|z^{L,k}\|_\infty \geq d_1\|\varepsilon^k\|_\infty  - \|z^{L,k}\|_\infty > d_1 \delta'  -
\|z^{L,k}\| = \delta,\]
hence, 
\[ F_k(\varepsilon^k) = F(z^{L,k}+ P_k^L \varepsilon^k)> \eta , \quad \forall
\varepsilon^k \, : \, ||\varepsilon^k||_\infty > \delta'.\]

\end{proof} 

The following results aim to estimate the difference between
sub-optimal solutions. For the sake of completeness and ease of reference, we include the following  technical lemma.
\begin{lemma} \label{prop:bound1}
	Let $F\in\cC^2(\R^N,\R)$ be a convex function,
	 such that its Hessian matrix is always
	positive definite. Let $A\subset \R^{N}$ be a compact convex
	set and let $x_A\in A$ be such that $\nabla F(x_A) = 0$. Let $B\subset A$ be a compact
	convex set and $x_B=\argmin \{ F(x), \, x \in B
	\}$. Then there exists $c=c(A,F, \| \cdot \|) \geq 1$ such that
	\[ \| x_A - x_B \| \leq c \ \text{dist}_{||\cdot||}(x_A,B), \]
	where $ \text{dist}_{||\cdot||}(x_A,B) = \min \{ || x-x_A||, \, x \in B\} $, for any norm in $\R^N$.
	
\end{lemma}
\begin{proof}
Since $\nabla F(x_A) = 0$, $B\subset A$ and $A$  is convex,  for any  $x \in
B$, $\exists \xi_x \in A$ such that
\begin{equation} \label{eq:taylor}
F(x) - F(x_A) = \frac12 (x - x_A)^T \nabla^2 F(\xi_x) (x - x_A).
\end{equation}

The properties of $F$ and $A$ allow us to define $\rho_{\min}$,
$\rho_{\max}$ as follows
\begin{align*}
	\rho_{\min}:= \min_{\xi\in A} \min_{x\neq 0} \frac{x^T \nabla^2 F(\xi) x}{\|x\|^2} >0, \qquad
	\rho_{\max}:= \max_{\xi\in A} \max_{x\neq 0} \frac{x^T \nabla^2 F(\xi) x}{\|x\|^2}>0,
\end{align*}
so that
\[
\rho_{\min} \|x\|^2 \leq x^T \nabla^2 F(\xi) x \leq \rho_{\max} \|x\|^2, \qquad \forall\xi\in A, \ \forall x\in\R^N.
\]
Let $w$ be an arbitrary but fixed point in $B$. Since $F(x_B) \leq F(w)$,
applying \cref{eq:taylor} with $x=w$ and $x=x_B$ and using the
inequalities above, we get
\[ \frac{\rho_{\min}}2 \|x_B - x_A\|^2 \leq  F(x_B) - F(x_A) \leq F(w) - F(x_A) \leq \frac{\rho_{\max}}2 \|w-x_A\|^2.\]
Then,
\begin{equation} \label{eq:c-lemma}
\| x_A - x_B \| \leq c \|x_A-w\|, \qquad  c := \sqrt{\rho_{\max} /
	\rho_{\min}}.
\end{equation}
The result follows from observing that the last inequality is
true for every $w\in B$.
\end{proof}

\begin{proposition} \label{prop:dist_z_xi}
	Let  $F\in\cC^2(\R^{N_L},\R)$  be a convex coercive function such
	that $\nabla^2 F(\xi^L) $ is a positive definite matrix $\forall \xi^L
	\in \R^{N_L}$.  If   ${\cal P}=\{P_k^{k+1}\}_{k=0}^{L-1}$
        satisfies {\em property S}, then
        the sub-optimal solutions 
	$\{z^{L,k}\}_{k=1}^{L+1}$ in \cref{eq:k-auxP}  are well defined and unique. In
        addition,  for any  compact and  convex set $D$ which contains
        in its interior the sequence $\{z^{L,k}\}_{k=0}^{L+1}$,
  there exists  $C=C(F,D, {\cal P}) >0  $ such that for $\ell, k$:  $ 0\leq	\ell \leq k < L$,
\begin{equation} \label{eq:zk1-zl1}
 \| z^{L,k+1} - z^{L,\ell+1} \|_\infty \leq C \
 \text{dist}_\infty(z^{L,k+1}, \Xi_{\ell}\cap D).
\end{equation}
\end{proposition}
\begin{proof}
By \cref{prop:propF-Fk}, $\varepsilon^k_*$
in \cref{eq:alphaLk} is uniquely defined, and, thus, so is
$z^{L,k+1}$, for   each $k$, $0 \leq k \leq L$.
	
	Now, let  $\ell, k$ be   a pair of (fixed) indices,  $0\leq \ell\leq k
	\leq L$.  From \cref{eq:k-auxP}, we can write
	\begin{align*}
	z^{L,k+1}&= z^{L,k}+ P_k^L \varepsilon^k_*, \quad 
	\varepsilon^k_*=\argmin \{ F_k(\varepsilon^k), \,
	\varepsilon^k\in \R^{N_k} \}. \\
	z^{L,\ell+1}&=z^{L,\ell}+P_\ell^L \varepsilon^\ell_*, \quad 
	\varepsilon^\ell_*=\argmin \{ F_\ell(\varepsilon^\ell), \,
	\varepsilon^\ell\in \R^{N_\ell} \}.      
	\end{align*}
Thus,  since $P_\ell^L= P_k^L P_\ell^k$, 
\[ z^{L,k+1}-z^{L,\ell+1}= z^{L,k}-z^{L,\ell} +
P_k^L(\varepsilon_*^k- P_\ell^k \varepsilon_*^\ell). \]
Since $\Xi_\ell \subset \Xi_k$, 
	$\exists ! \,  \varepsilon^k_\ell$ : 
	$z^{L,\ell}= z^{L,k}+ P_k^L \varepsilon^k_\ell$\footnote{It is
          unique because of \cref{eq:d1-d2}.}. Hence we can write
	\begin{equation} \label{eq:zk-zl}
	z^{L,k+1}-z^{L,\ell+1}=
P_k^L (\varepsilon^k_A-\varepsilon^k_B),
	\quad  \begin{cases} 
	\varepsilon^k_A&= \varepsilon^k_*-\varepsilon^k_\ell, \\
	\varepsilon^k_B&=P_{\ell}^{k} \varepsilon^\ell_*.
	\end{cases}
	\end{equation}
Thus, by \cref{eq:d1-d2}, 
\[ \|	z^{L,k+1}-z^{L,\ell+1}\|_\infty \leq d_2 \|
\varepsilon^k_A-\varepsilon^k_B \|_\infty. \]
To estimate the right-hand-side of the above relation, {we seek to apply \cref{prop:bound1}}.
For this, we consider $D$ to  be a compact and  convex set which contains
        in its interior the sequence  of iterates
        $\{z^{L,k}\}_{k=0}^{L+1}$, and define the sets
\[	A_{\ell,k} := \{\varepsilon^k \in\R^{N_k}  : z^{L,\ell} + P_k^L\varepsilon^k \in D\},\quad 
	B_{\ell,k} := \{\varepsilon^k \in A_{\ell,k}  :
	\varepsilon^k = P_{\ell}^{k}\varepsilon^\ell, \,  
	\varepsilon^\ell\in\R^{N_\ell}\} .
\]
 Since $D$ is compact and convex, and the
        prediction operators are  linear, $A_{\ell,k}$ and
$B_{\ell,k}$ are compact and convex 	sets in $\R^{N_k}$. 
Notice
that  \cref{lemma:contained} and \cref{eq:d1-d2} imply that there are
one-to-one correspondences between $A_{\ell,k} \, \Leftrightarrow \, \Xi_k \cap
D$ and $B_{\ell,k}\, \Leftrightarrow \, \Xi_\ell \cap   D$.  The sets
$\Xi_k \cap D$ and $\Xi_\ell \cap D$ are compact and convex in $\R^{N_L}$.

Taking into account the relations above, we easily get
\begin{equation} \label{eq:eka-ekb}
z^{L,\ell+1} =z^{L,\ell}+ P_k^L
	P_{\ell}^{k} \varepsilon^\ell_* \, \, \rightarrow \, \, \varepsilon^k_B
	\in  B_{\ell,k}, \, \, z^{L,k+1}=z^{L,\ell}+P_k^L \varepsilon^k_A\, \, \rightarrow \, \,
	\varepsilon^k_A \in A_{\ell,k}.
\end{equation}
Hence, considering the function
	\[ F_{\ell,k}: \R^{N_k} \rightarrow \R, \quad
	F_{\ell,k}(\varepsilon^k):=F(z^{L,\ell}+P_k^L \varepsilon^k) \]
 we have
\[ 	F_{\ell,k}(\varepsilon^k_A)=  F(z^{L,k+1} )=
 \min_{z^L \in \Xi_k} F(z^L), \quad
 F_{\ell,k}(\varepsilon^k_B)=  F(z^{L,\ell+1})= \min_{ z^L \in
   \Xi_\ell} F(z^{L}), \]
therefore
\[ F_{\ell,k}(\varepsilon^k_A)= \min_{\varepsilon^k \in A_{\ell,k}}
F_{\ell,k}(\varepsilon^k), 
\quad 
F_{\ell,k}(\varepsilon^k_B)= 
\min_{\varepsilon^k \in B_{\ell,k}} F_{\ell,k}(\varepsilon^k).
\]
Since  $F_{\ell,k}(\varepsilon^k)=F_k(\varepsilon^k_\ell+ \varepsilon^k)$, we
have that  $\nabla F_{\ell,k}(\varepsilon^k_A)= \nabla F_k(\varepsilon_*^k)= 0$.
Hence, applying \cref{prop:bound1}, there
	exists $  c_{\ell,k}= c_{\ell,k}(A_{\ell,k},F_{\ell,k})$ such that
	\[ \| \varepsilon_A^k - \varepsilon_B^k \|_\infty \leq  
	c_{\ell,k} \ \text{dist}_\infty(\varepsilon_A^k,B_{\ell,k}).
 \] 
Let $\varepsilon^k \in B_{l,k}$ and $\varepsilon^{\ell} \in
\R^{N_\ell}$: $\varepsilon^k= P_{\ell}^k \varepsilon^\ell$. Then, 
 using  \cref{eq:d1-d2} and \cref{eq:eka-ekb}  we have
\[ ||P_\ell^{k} \varepsilon^l - \varepsilon^k_A||_\infty \leq d_1^{-1} ||P_k^L (P_\ell^{k}
\varepsilon^l - \varepsilon^k_A)||_\infty 
= d^{-1}_1||z^{L,\ell}+P_\ell^L \varepsilon^\ell - z^{L,k+1}\|_\infty, \]
thus
	\begin{align*}
	\text{dist}_\infty(\varepsilon_A^k,B_{\ell,k}) &
	\leq 
	d_1^{-1} \inf \left\{ \| z^{L,\ell} + P_\ell^L\varepsilon^\ell
          - z^{L,k+1} \|_\infty \ : \ P_\ell^{k} \varepsilon^\ell
        \in B_{k,\ell} \right\} \\ 	&= 
d_1^{-1} \text{dist} (z^{L,k+1}, \Xi_\ell\cap D).
	\end{align*}	
Therefore
\begin{equation} \label{eq:zk1-zl1-0}
\| z^{L,k+1} -  z^{L,\ell+1} \|_\infty \leq d_2
        ||\varepsilon^k_A-\varepsilon^k_B||_\infty \leq c_{\ell,k}
        \frac{d_2}{d_1} \text{dist} (z^{L,k+1}, \Xi_\ell\cap D).
\end{equation}
To see that $c_{\ell,k}$ on the right hand side can be substituted by
a universal constant, independent of $\ell, k$, we proceed as follows:
Recall that, from \cref{prop:bound1},  $c_{\ell,k} =c(A_{\ell,k},F_{\ell,k}) =\sqrt{\rho_{\max}^{l,k} / \rho_{\min}^{l,k}}$ with
	\begin{align*}
	\rho_{\min}^{l,k}= \min_{\xi^k\in A_{\ell,k}} \min_{\varepsilon^k\neq 0} \frac{(\varepsilon^k)^T \nabla^2  F_{\ell,k}(\xi^k) \varepsilon^k}{\|\varepsilon^k\|_\infty^2}, \qquad
	\rho_{\max}^{l,k}= \max_{\xi^k\in A_{\ell,k}} \max_{\varepsilon^k\neq 0} \frac{(\varepsilon^k)^T \nabla^2  F_{\ell,k}(\xi^k) \varepsilon^k}{\|\varepsilon^k\|_\infty^2}.
	\end{align*}
Taking into account  that   $\forall \, \xi^k  \in \R^{N_k}$ 
\[\quad \nabla^2 F_{\ell,k}(\xi^k)= (P_k^L)^T
	\nabla^2 F( z^k) P_k^L, \quad z^k=z^{L,\ell}+ P_k^L\xi^k, \]
and  recalling that \cref{eq:d1-d2} implies $P_k^L
 \varepsilon^k \neq 0$ when $\varepsilon_k \neq 0$,  we can write
	\begin{align*}
	\frac{(\varepsilon^k)^T \nabla^2  F_{\ell,k}(\xi^k)
		\varepsilon^k}{\|\varepsilon^k\|_\infty^2}&= 
	\frac{(P_k^L \varepsilon^k)^T \nabla^2F(z^k) (P_k^L
		\varepsilon^k)}{\|P_k^L \varepsilon^k\|_\infty^2} \frac{\|P_k^L
		\varepsilon^k\|_\infty^2}{\|\varepsilon^k\|_\infty^2} \\ & \geq  
	\frac{(P_k^L \varepsilon^k)^T \nabla^2F(z^k) (P_k^L
\varepsilon^k)}{\|P_k^L \varepsilon^k\|_\infty^2}  d_1^2  
\geq \min_{\varepsilon^L\neq 0} 
\frac{(\varepsilon^L)^T \nabla^2
  F(z^k)\varepsilon^L}{\|\varepsilon^L\|_\infty^2} d_1^2 .
	\end{align*}
Since $\xi^k \in A_{\ell,k} \, \Rightarrow \, z^k=z^{L,\ell}+ P_k^L
\xi^k \in \Xi_k \cap D$, we have  that
\begin{equation}  \label{eq:rhomin-D}
	\rho_{\min}^{l,k}\geq  \rho_{\min}(F,D)d_1^2, 
 \quad \text{ with } \quad 
\rho_{\min}(F,D) \!\!:= \min_{z \in D}
 \min_{\varepsilon^L\neq 0} 
\frac{(\varepsilon^L)^T \nabla^2 F(z)\varepsilon^L}{\|\varepsilon^L\|_\infty^2}, 
\end{equation}
proceeding analogously with the upper bounds, we have that
\begin{equation}  \label{eq:rhomax-D}
\rho_{\max}^{l,k} \leq
        \rho_{\max}(F,D)d_2^2, \quad \text{ with } \quad 
 \rho_{\max}(F,D)\!\! :=\max_{z \in D} \max_{\varepsilon^L\neq 0}
	\frac{(\varepsilon^L)^T \nabla^2 F(z)
		\varepsilon^L}{\|\varepsilon^L\|_\infty^2}.
\end{equation}
Thus, we obtain \cref{eq:zk1-zl1} with 
\[  C=C(F,D, {\cal   P}):=
\sqrt{\frac{\rho_{\max}(F,D)}{\rho_{\min}(F,D)}}{\frac{d_2^2}{d_1^2}}. \]
\end{proof}

This result allows us to relate the difference between sub-optimal
solutions to the size of certain detail coefficients.
\begin{theorem} \label{thm:bound3}
	Let $F \in {\cal C}^2(\R^{N_L} ,\R)$ and ${\cal P}$  as in \cref{prop:dist_z_xi}.
	Then, for any pair of indices $\ell,k$, $0\leq \ell \leq k < L$
        there exists $C > 0$, independent of $\ell,k$, such that
\begin{equation} \label{eq:zl1-zk1}
\| z^{L,k+1} - z^{L,\ell+1} \|_\infty \leq C \|
(d^{\ell}(z^{L,k+1}- z^{L,0}), \ldots, d^{k-1}(z^{L,k+1}-z^{L,0}))
\|_\infty. 
\end{equation}
\end{theorem}
\begin{proof}
From \cref{eq:MkL-zLk1} and \cref{eq:dk-equal}, we can write 
\[
M_{\ell,L} z^{L,\ell+1}  = (z^\ell,\bar d^{\ell},  \ldots, 
\bar d^{L-1}), \qquad 
M_{k,L} z^{L,k+1}=(z^k,\bar d^{k},  \ldots,  \bar d^{L-1}),
\]
with  $z^m:=z^{m}_*+\varepsilon^m_* \in \R^{N_m}$, $m=k, \ell$ and
$\bar d^j=d^j(z^{L,0})$. 
Since $\ell \leq k$, by \cref{eq:M0L-cases-0},
\begin{align*}
M_{\ell,L} z^{L,k+1}  &= ( M_{\ell, k} z^k,\bar d^{k}, \ldots,
\bar d^{L-1})
= (x^\ell, \hat d^{\ell}, \ldots, \hat d^{k-1}, \bar d^{k}, \ldots,  \bar d^{L-1}), \quad 
\end{align*}
with  $x^\ell \in \R^{N_\ell}$, $\hat d^{j}= d^j(z^{L,k+1})$ (see \cref{rem:d-zL}).
Then,
\[ M_{\ell}^L (z^{L,k+1}- z^{L,\ell+1}) = (x^\ell - z^\ell , \hat
d^{\ell}-\bar d^\ell, \cdots, \hat d^{k-1}-\bar d^{k-1}, 0 , \cdots, 0 ). \]	
Thus,  $w^{\ell,k} :=  z^{L,\ell+1} + P_\ell^L (x^\ell-
z^{\ell}) \in \Xi_\ell$ satisfies
\begin{equation} \label{eq:MlLz-w}
 M_{\ell,L}(z^{L,k+1} - w^{\ell,k}) = 
\left (0, \hat d^{\ell} -\bar d^{\ell}, \ldots, \hat d^{k-1} - \bar d^{k-1}, 0, \ldots,  0\right ).
\end{equation}
	Let $D$ be a compact set
	containing all the iterates $\{z^{L,m}\}_{m=0}^{L+1}$ and all the points
	$w^{\ell,k}$, $\forall \ell, k :  0\leq \ell \leq k
        <L$. Applying  \cref{prop:dist_z_xi}  $ \exists \, C=C(F,D,{\cal P})$ 
such that
	\begin{align*}
	\| z^{L,k+1} - z^{L,\ell+1} \|_\infty &\leq C \
	\text{dist}(z^{L,k+1},\Xi_{\ell}\cap D) \leq
	C\|z^{L,k+1} - w^{\ell,k}\|_\infty.
	\end{align*}
Since $ \|z^{L,k+1} - w^{\ell,k}\|_\infty = \| M_{\ell,k}^{-1} M_{\ell, k}
(z^{L,k+1} - w^{\ell,k})\|_\infty$, by \cref{eq:MlLz-w}  and  \cref{eq:M-stab},  we have 
\begin{align*}
\| z^{L,k+1} - z^{L,\ell+1} \|_\infty &\leq C
\left \|M_{\ell,L}^{-1}\left (0, \hat d^{\ell}-\bar d^\ell,
\ldots, \hat d^{k-1}-\bar d^{k-1}, 0, \ldots, 0\right )\right \|_\infty\\
&\leq C \tilde C \left \| (d^{\ell}(z^{L,k+1}-z^{L,0}), \ldots, d^{k-1}(z^{L,k+1}-z^{L,0})) \right \|_\infty,
\end{align*}
since $\hat d^{m}-\bar d^m= d^m(z^{L,k+1}-z^{L,0})$, by the linearity of the MR transformations. 
\end{proof}

 Under smoothness assumptions, the
decay of the detail coefficients is related to the approximation
properties of the prediction scheme, which allows to obtain
quantitative estimates for the  distance between sub-optimal solutions. .
\begin{corollary} \label{cor:distance1}
	Let $F \in {\cal C}^2(\R^{N_L} ,\R)$  as in
        \cref{prop:dist_z_xi},  and   
${\cal P}=\{P_{k}^{k+1}\}_{k\geq 0} $ the sequence of interpolatory
prediction operators given in \cref{sec:1D-int}.
If the initial guess $\bar z=z^{L,0}$
 and $z_{\min}=\argmin \{ F(z), \, z \in \R^{N_L}\}$ can be associated 
 to  the  point evaluations on ${\cal G}^L$ of
 sufficiently  smooth functions, then for $ 0\leq \ell < L$, \\[-5pt]
\begin{enumerate}
\item $ \| z_{\min} - z^{L,\ell+1} \|_\infty =  O(h_{\ell+1}^{n+1})$, \label{it:zmin-zl1}\\[-5pt]

\item $\|z^{L,\ell+1} - z^{L,\ell}\|_\infty =  O(h_{\ell}^{n+1})$,  \label{it:zl-zl1} \\[-5pt]

\item $ \| \varepsilon^\ell_* \|_\infty =  O(h_{\ell}^{n+1})$.  \label{it:epstar}

\end{enumerate}
\end{corollary}
\begin{proof}
If 	$z_{\min} -\bar z$ can be considered as   point evaluations of
a  sufficiently smooth function, we have that  $\|d^j(z_{\min}-\bar
z)\|_\infty= O(h_{j+1}^{n+1})$, $0\leq j <  L$. To prove \cref{it:zmin-zl1} we take  $k = L$ in
        \cref{thm:bound3}. Since $z^{L,L+1} = z_{\min}$, $\bar z=
        z^{L,0}$, we have  
	\[ \| z_{\min} - z^{L,\ell+1} \|_\infty \leq C \|
        (d^{\ell}(z_{\min} - \bar z), \ldots, d^{L-1}(z_{\min} - \bar
        z)) \|_\infty = O(h_{\ell+1}^{n+1}).\]

\cref{it:zl-zl1} and  \cref{it:epstar} follow from the previous result, since
 \[\| z^{L,\ell} - z^{L,\ell+1} \|_\infty \leq \| z_{\min} - z^{L,\ell} \|_\infty + \| z_{\min} - z^{L,\ell+1} \|_\infty \]
and 	$\z^{L,\ell+1}-\z^{L,\ell} =  P_\ell^L\varepsilon^\ell_*$, 
hence, by \cref{eq:d1-d2},
	\[ \|\varepsilon^\ell_*\|_\infty \leq d_1^{-1}  ||
        P_\ell^L\varepsilon^\ell_*||_\infty = d_1^{-1} \ \|\z^{L,\ell+1}-\z^{L,\ell}\|_\infty. \]
	
\end{proof}

We summarize the  results shown in this section as follows: For
(suitable) convex optimization
problems,  the distance between sub-optimal solutions is related to
the size of certain detail coefficients.   Provided that both the solution and
the   initial guess  are 
{\em smooth},  the $\infty$-distance between 
consecutive sub-optimal solutions  decreases when climbing up the
MR ladder (at a rate  which depends on the properties of the 
prediction operators). In this case, even though the auxiliary minimization problems
\cref{eq:alphaLk}  involve an increasing number of
degrees of freedom when the resolution level increases,  it  seems 
reasonable to assume that they  will be
{\em efficiently}  solved by the given optimization
technique, due to the fact that the initial guess and the solution of
each auxiliary problem are  increasingly {\em closer}.

Our  theoretical
results are numerically  validated in the  
next section.

\section{Numerical Experiments}
\label{sec:numex}

The common setting for our numerical experiments  is
as follows:   We 
seek to solve \cref{eq:problema}  by embedding the problem into an
appropriate interpolatory MR Framework and using  the  MR/OPT 
strategy  with a given
minimization tool (${\cal D}$)  and an initial guess, $\bar z$,  which we assume
supplied  by the end-user. 
We recall the following two important
points in our approach: \\
{\bf 1. }The solution of  the {\em
	auxiliary minimization problems} \cref{eq:alphaLk}   can be
carried out using only the objective function $F$ (as a black box routine).
Moreover, in practice only the prediction rules of the MR transformation are
really required for the application of our  algorithm. \\
{\bf 2. } The chosen optimizer, ${\cal D}$, is also a black-box
routine in  the 	algorithm.

In this paper we shall consider
two different optimization tools from    MATLAB:
\texttt{fminunc}, a \emph{steepest descent method} which requires a
gradient  {estimation}, and  the general purpose optimization tool \texttt{patternserach}, which is a \emph{coordinate search method}. 
When using them as  black-box  tool, a series of parameters
need to be specified. We set the running parameters  seeking to stop  ${\cal D}$ when   the max-norm of the difference
between two consecutive iterates within the call falls below  a
specified tolerance ($tol_{\cD}$)\footnote{In our  
	implementation the running parameters (consult its documentation for further information) are \texttt{TolX} $=tol_{\cD}$, \texttt{TolFun}=0 and \texttt{MaxFunEvals}=\texttt{MaxIter}=Inf. For \texttt{fminunc} we additionally established \texttt{OptimalityTolerance}=0, \texttt{Algorithm}=\texttt{quasi-newton}, \texttt{FinDiffRelStep}=1e-13 and \texttt{FiniteDifferenceType}=\texttt{central}.}.

In addition, it seems reasonable to consider
\begin{equation} \label{eq:tolM}
\|z^{L,k+1}-z^{L,k}\| \leq tol_M
\end{equation}
as a  suitable stopping criteria for the MR/OPT strategy, in
order  to avoid unnecessary calls to the optimizer.  In our numerical experiments  we always  choose  $tol_M = tol_{\cD}$.

The tests considered in the following subsections 
are  taken  from the specialized literature. In
all the cases considered,  the MR/OPT  
strategy leads to a  {\em more efficient} computation  than the direct use of $ {\cal D}$ to
solve \cref{eq:problema}. To measure the performance of the MR/OPT
strategy, we have considered, as in \cite{FP15},  the 
total number of eva\-luations of the objective function $F$, since
often  the remaining ope\-ra\-tions may 
have a negligible cost compared  to that. 
In addition, we also examine  the
ratio between the 
number of functional evaluations required to solve each $k$-th auxiliary
problem versus its number of degrees of freedom, because it also
provides very valuable information on the behavior of the strategy. 

All the computations have been done with MATLAB R2020b on a
MacBook Air
with a 1,8 GHz Intel Core i5 double core processor and a 8 GB 1600 MHz DDR3 memory.

\subsection{ Quadratic  Optimization Problems}

\label{sec:1Dpoisson}

Boundary Value Problems (BVP) in ordinary and partial differential
equations are a common source of large scale
optimization problems. In  many cases, the discrete setting
corresponds to a quadratic minimization problem of the type considered
in \cref{prop:F-quad}.  Here, we shall consider 1D and 2D  BVP with
smooth solutions in order  to check
numerically the theoretical results stated in the previous section.

We consider first the 1D   BVP  (from \cite{Nash00}),
\begin{equation} \label{eq:poisson1D}
\begin{cases}
-u''(t) + 2u(t) = f(t), & t\in(0,1)\\
u(0)=u(1)=0.
\end{cases}
\end{equation}
where $f(t):=10^6 t(1-t)(t-1/2)(t-1/4)(3/4-t)$. Using
the standard centered second order discretization  for $u''$ on the uniform grid
$\mathcal{G}:=(i/J)_{i=0}^{J}$ leads to  the linear system
\[ (-z_{i-1} +2z_i - z_{i+1})J^2 + 2z_i = f(i/J), \qquad i=1,2,\ldots,J-1,\]
for the unknown values $z_i= u(i/J)+ O(h^2)$,
$h=1/J$, $1 \leq i \leq J-1$  ($z_{0}=z_{J} = 0$ because of  the boundary
conditions). In practice,  $J$ is {\em sufficiently large} so that the 
(discrete) solution of the  linear  system is {\em sufficiently close}
to the true solution of
the ODE, according to the end-user. 

The coefficient matrix for the system 
$A\in\R^{(J-1) \times (J-1)}$, is banded,
symmetric and definite positive, hence the solution can be
found by solving a  quadratic minimization problem of the type stated
in \cref{prop:F-quad}. We compute the solution of this minimization problem using the
MR/OPT strategy in combination with  the   1D Interpolatory MR
setting described in  \cref{sec:1D-int}. For this, we consider   $J=J_L$,  ${\cal G}=
{\cal G}_L$  as the finest  mesh in a ladder of nested
uniform grids on $[0,1]$ of the form ${\cal G}_k=
(i/J_k)_{i=0}^{J_k}$, $J_k=J_{k+1}/2$, 
$0 \leq k < L$. Hence, the discrete
spaces in  the corresponding 2D-interpolatory MR 
setting have dimension $(J_k+1)$.  In our computations, we take $J_L=128=2^{7}$,
$L= 5$ (i.e. $J_0=4$),   $tol_{\cD}=
10^{-6}=tol_M$. 

Considering  the zero vector as initial data, we
comply with the boundary conditions, i.e {$\bar z=
	z^{L,0}=\vec 0 \in \R^{N_L}$}. At each resolution level we enforce 
$\varepsilon^k_0=0=\varepsilon^k_{J_k}$  when applying the prediction
operators and, as a result, the sub-optimal
solutions associated to the  auxiliary optimization
problems
naturally satisfy the  boundary conditions.
In \cref{fig:poisson1D_fminunc_tol6_z}, we display $z^{L,k+1}$ for
$k=0,1,2$ and the different prediction operators  in
\cref{sec:1D-int}. For $n=1$, the
sub-optimal solution is a piecewise linear function, while for
$n=3,5$ we can clearly observe that the 
sub-optimal solutions are {\em smoother}\footnote{These facts are
	related  to the properties of the prediction scheme.}. The plots
shown correspond to ${\cal D}=${\tt fminunc}, but are
indistinguishable from those obtained with ${\cal D}=${\tt pattersearch}.
\begin{figure}[!h]
	\centering
	\begin{tabular}{ccc} 
		\includegraphics[width=0.3\textwidth,clip]{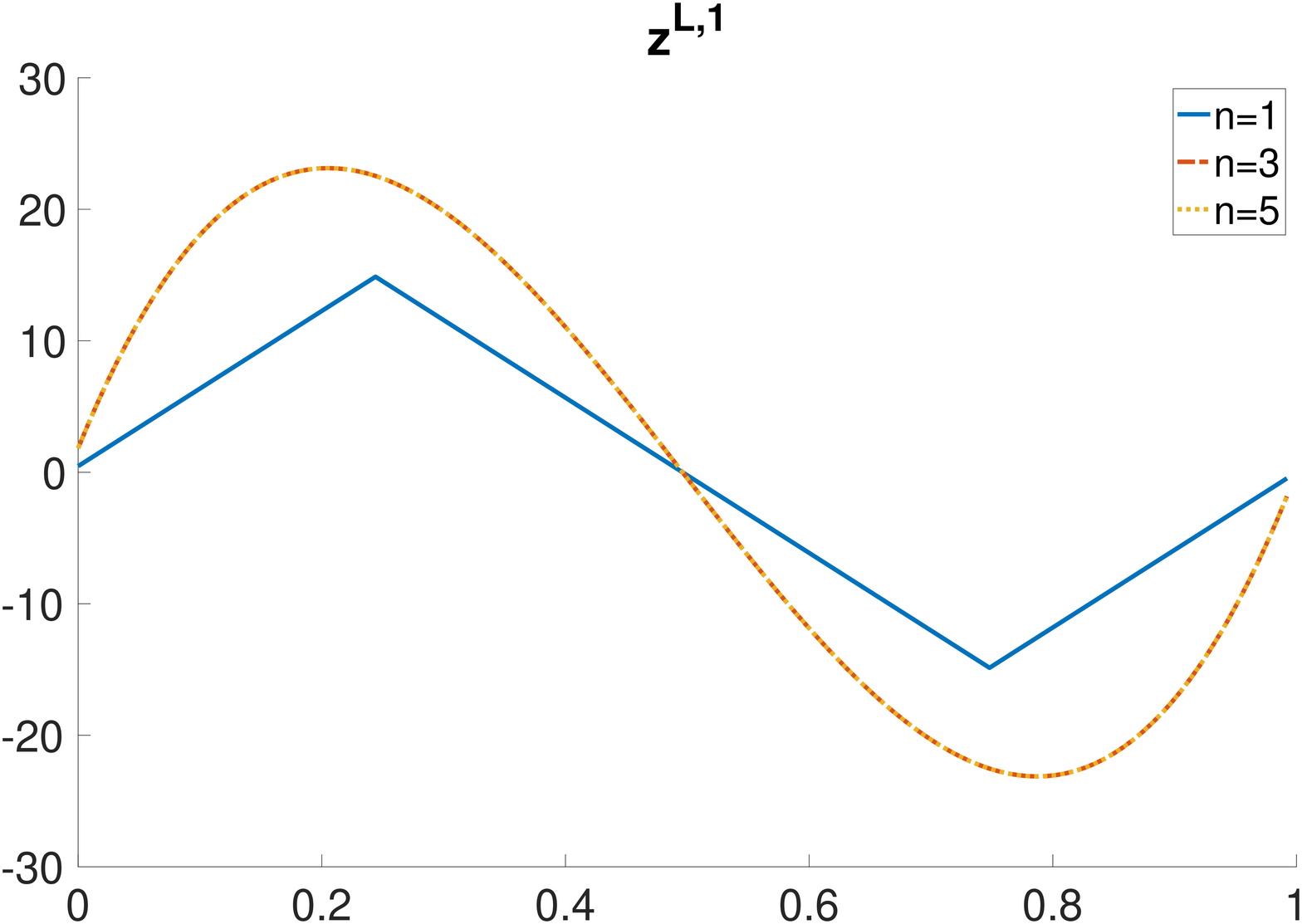} &
		\includegraphics[width=0.3\textwidth,clip]{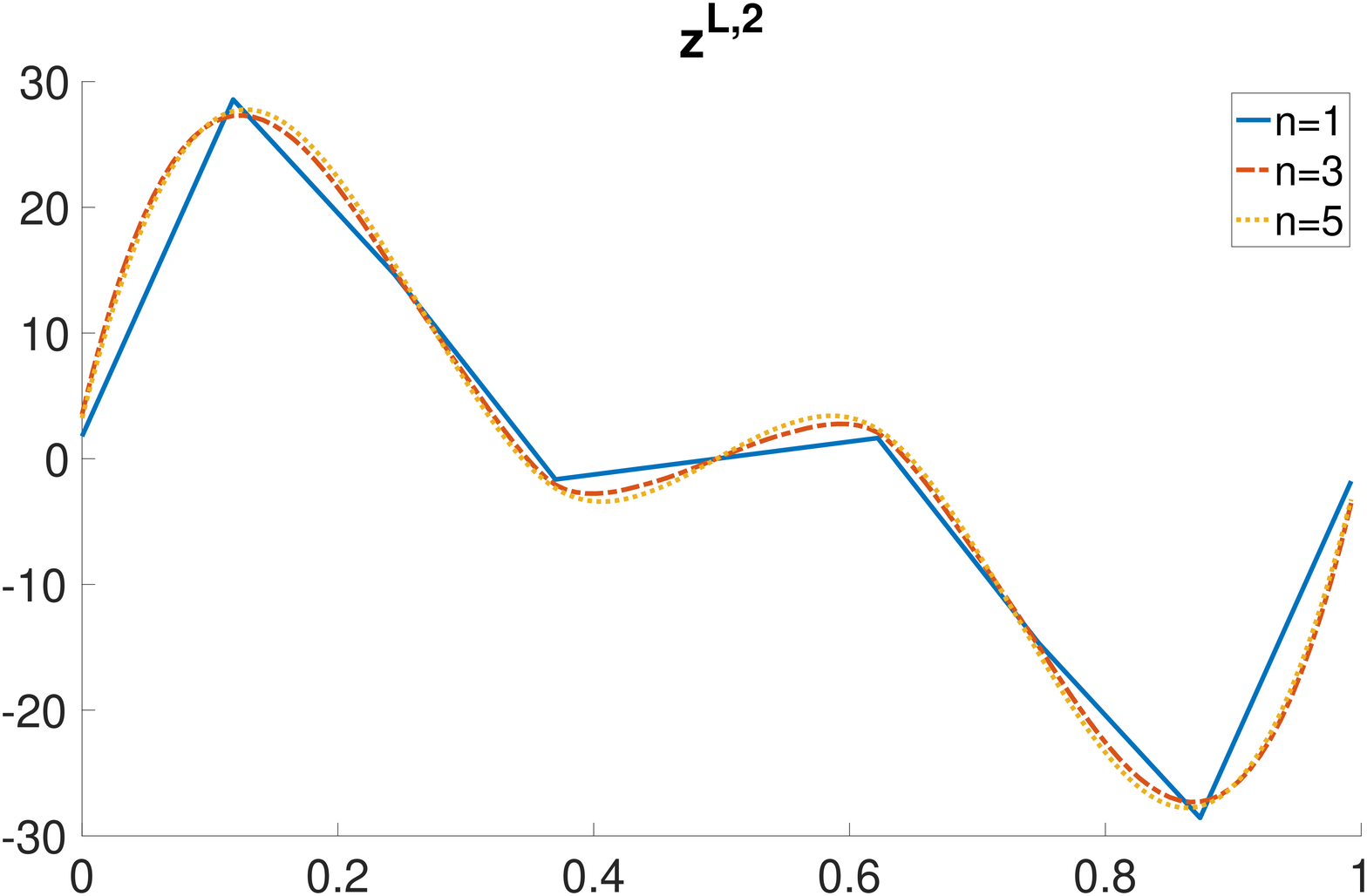} &
		\includegraphics[width=0.3\textwidth,clip]{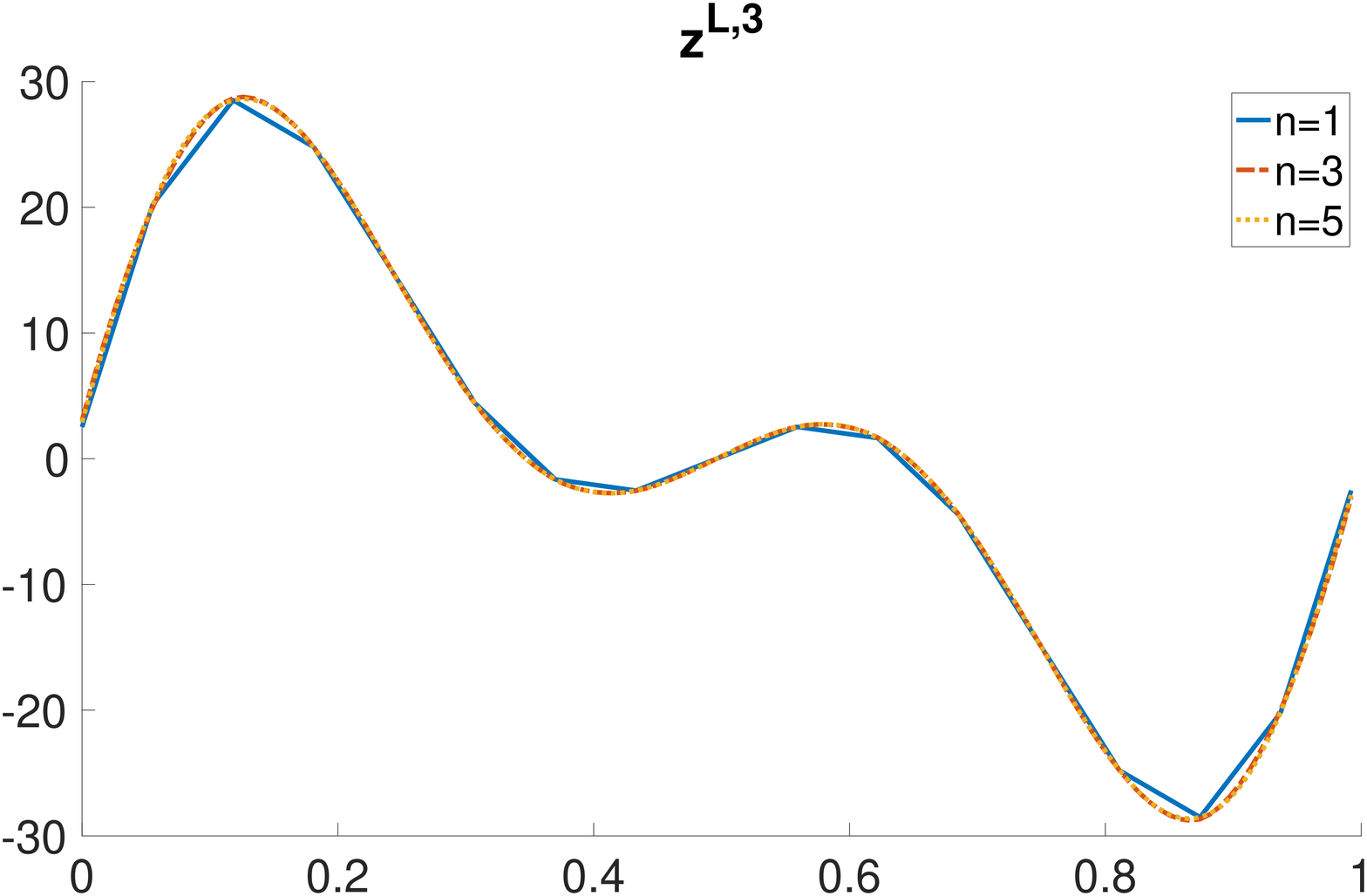}
	\end{tabular}
	\caption[1D Poisson problem]{1D BVP ($\cD=$
		{\tt fminunc}). The sub-optimal solutions $z^{L,1},z^{L,2},z^{L,3}$, from left to right, for $n=1,3,5$. 
		\label{fig:poisson1D_fminunc_tol6_z}}
\end{figure} 
\begin{figure}[!h]
	\centering
	\begin{tabular}{ccc} 
		\includegraphics[width=0.4\textwidth,clip]{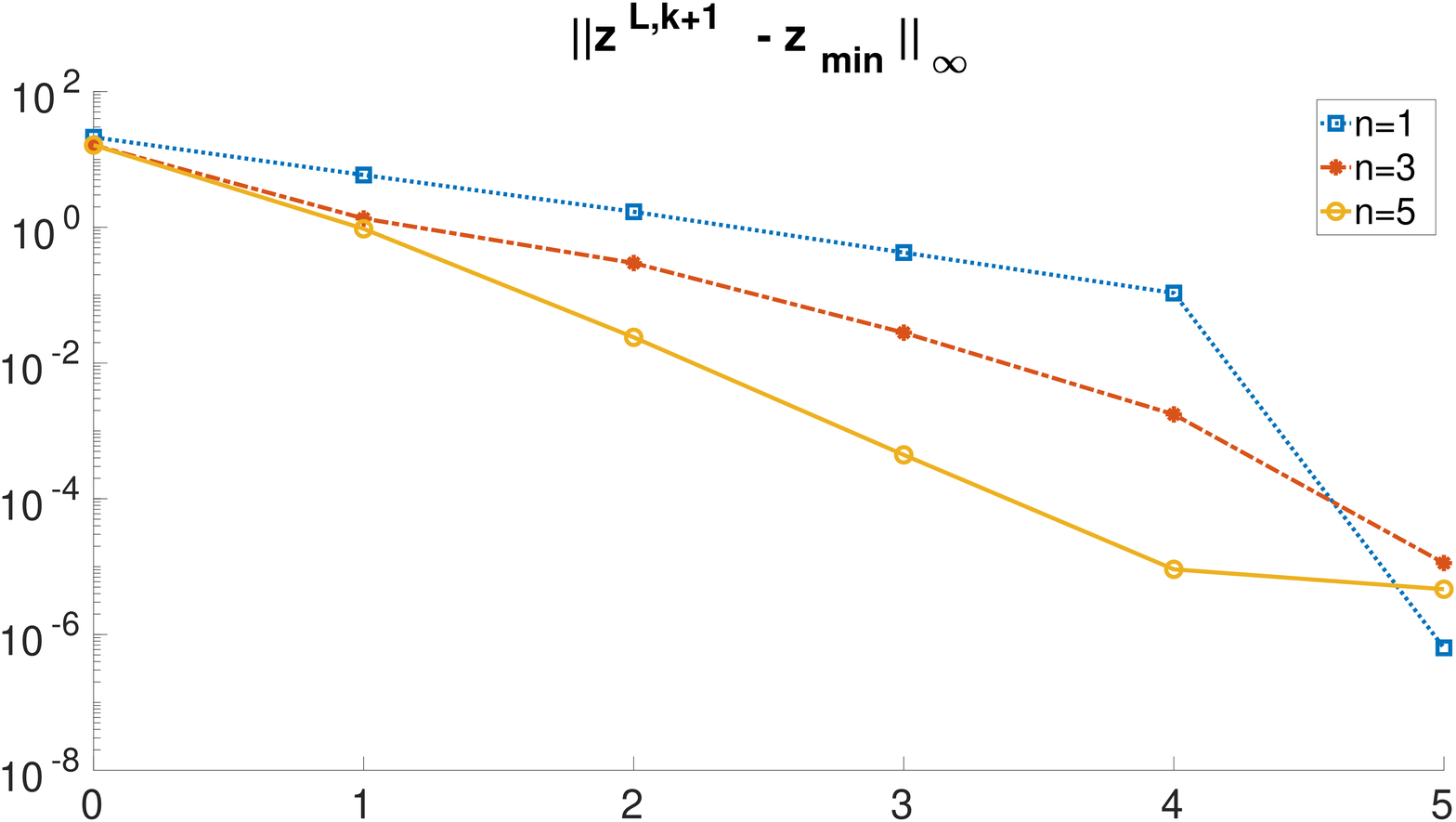} 
		&\includegraphics[width=0.4\textwidth,clip]{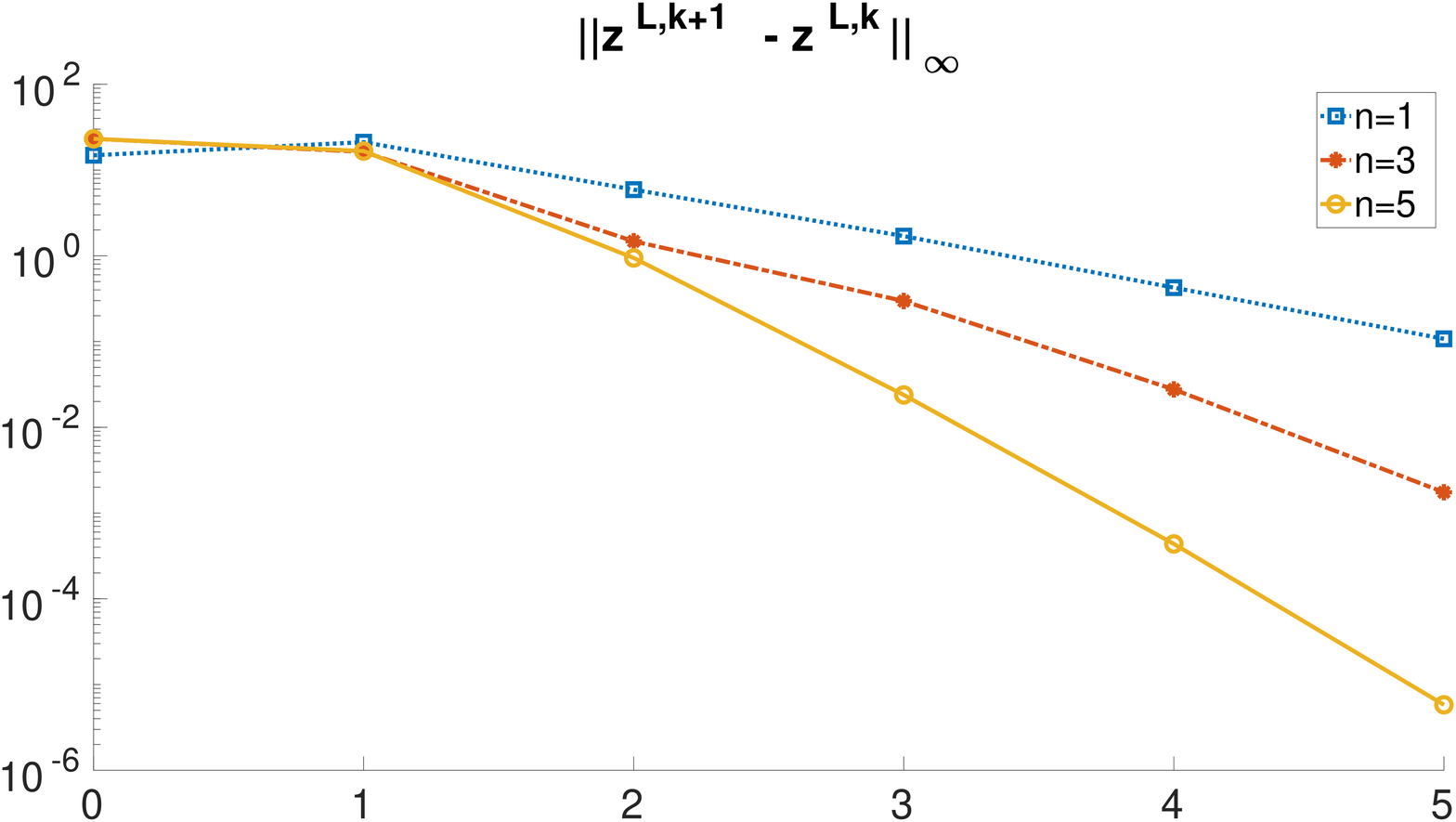} \\[3pt]
		\includegraphics[width=0.4\textwidth,clip]{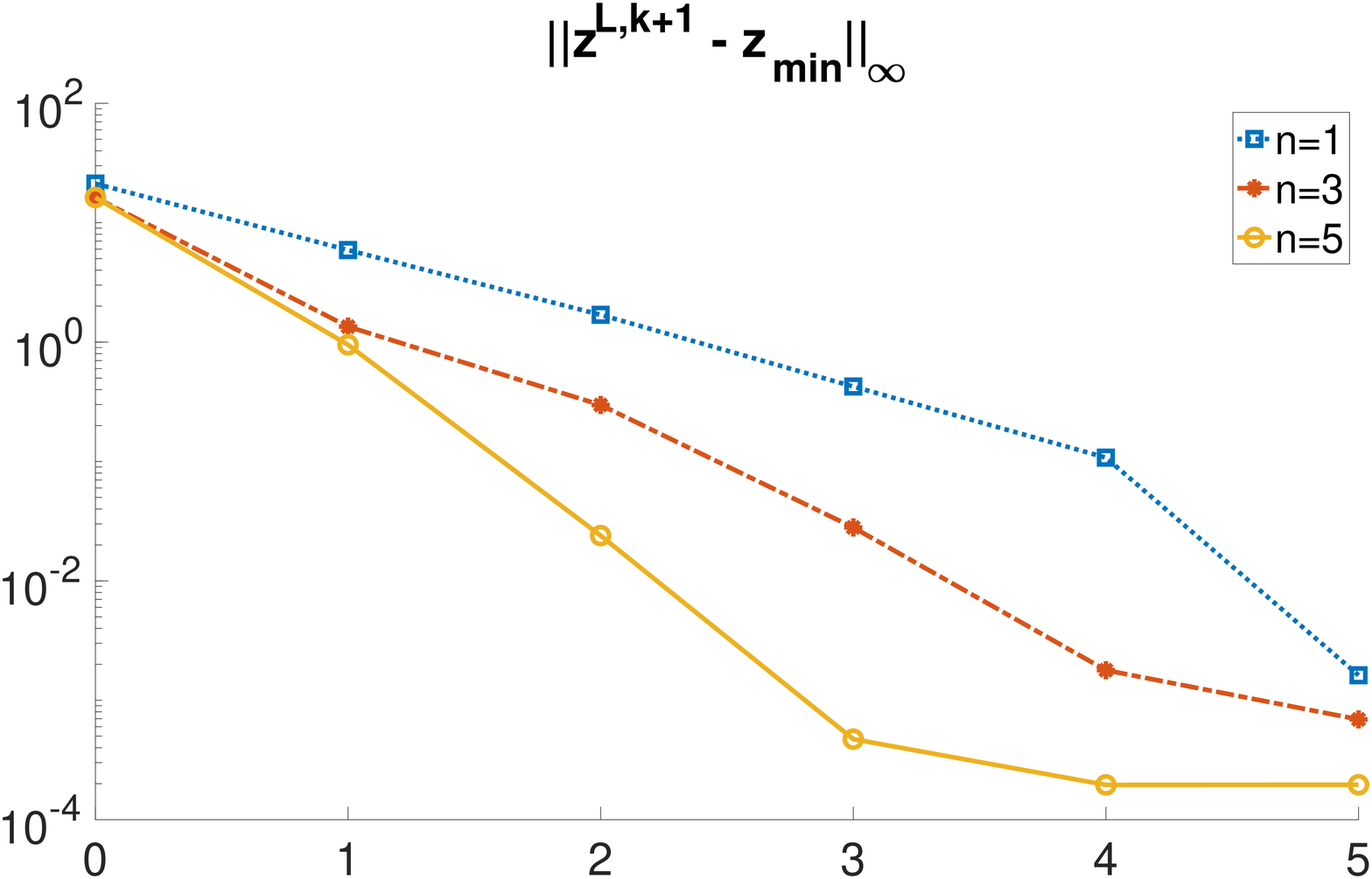}
		&\includegraphics[width=0.4\textwidth,clip]{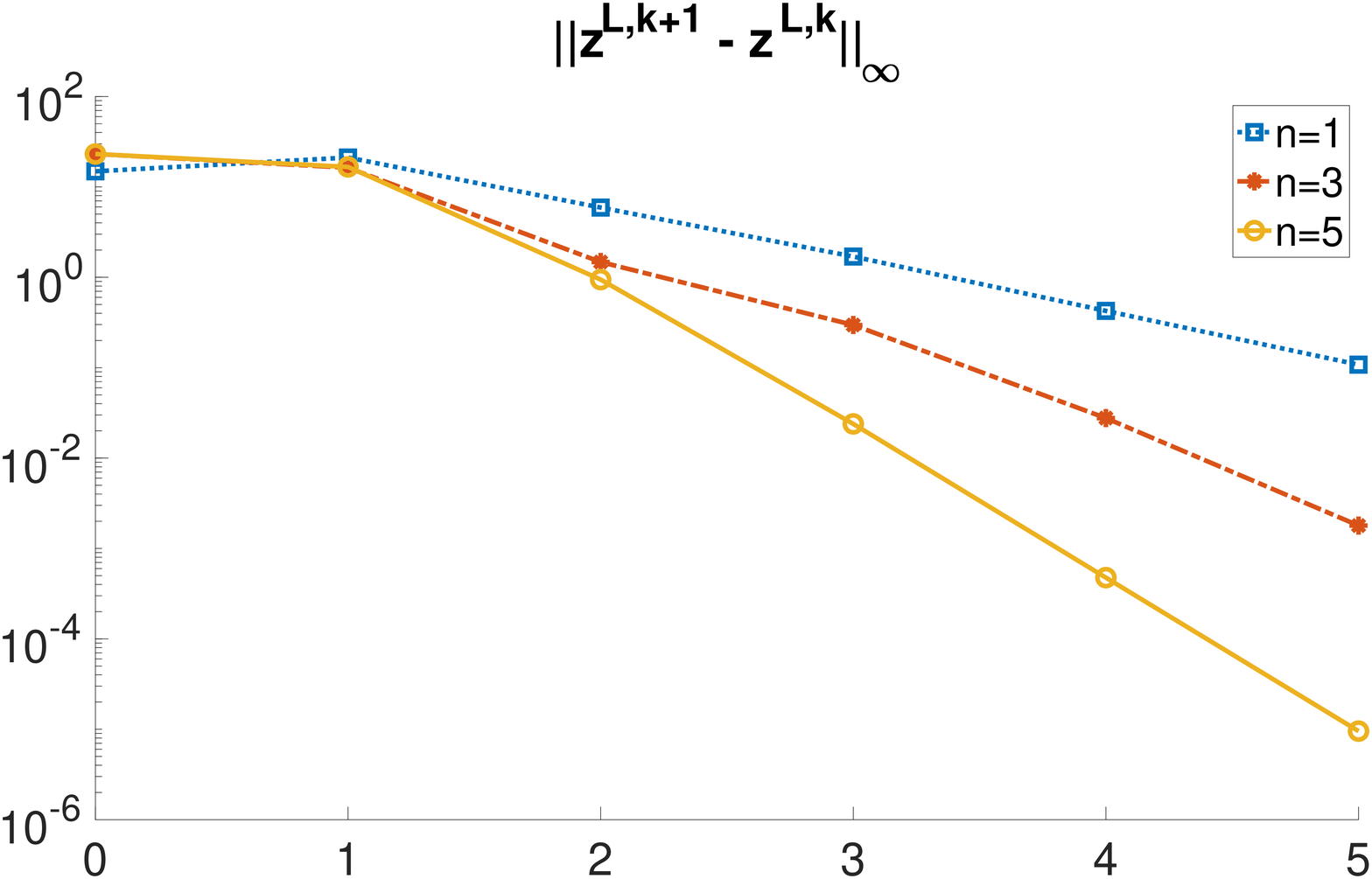} \\[5pt]
	\end{tabular}
	
	\begin{tabular}{|c|c|c|c|c|c|c|}
		\hline
		\(n\backslash k\) & 1 & 2 & 3 & 4 & 5 & Theoretical rate\\\hline
		1&-0.51&1.84&1.80&2.00&1.98&2\\\hline
		3&0.48&3.54&2.47&3.37&3.80&4\\\hline
		5&0.45&4.24&5.33&5.78&6.09&6\\\hline
	\end{tabular}
	\caption[1D Poisson problemB]{1D BVP. $tol_{\cal D}=tol_M= 10^{-6}$.
		Horizontal axis, $k$ (resolution level). Top row,
		$\cD=\texttt{fminunc}$; Bottom row, $\cD=\texttt{patternsearch}$.
		Left column,  $\infty$-distance between $z^{L,k+1}$ and $z_{\min}$.
		Right column, $\infty$-distance between consecutive sub-optimal
		solutions. Table,  numerical  estimation of  $r$   in
		\cref{eq:decay-num}, for $\cD=\texttt{fminunc}$. }
	\label{fig:new-poisson1D_fminunc_tol6}
\end{figure}

In
\cref{fig:new-poisson1D_fminunc_tol6}, we clearly  observe that  the
decay rates  of (the max-norm 
of)  the difference between 
the sub-optimal solutions and $z_{\min}$\footnote{Computed with $A\backslash b$ of
	MATLAB.}, and  the difference between  
two consecutive sub-optimal  solutions, depend on the prediction
scheme, but not  on  the optimization tool, just  as
predicted by our theoretical results.   The table below
the plots, shows a numerical estimation of  the decay 
rate corresponding to  the values $||z^{L,k+1}-z^{L,k}||_\infty$
obtained when using  \texttt{fminunc} (the results for \texttt{patternsearch}
are similar and are not shown), computed as 
\begin{equation} \label{eq:decay-num}
{r= \log_2 \frac{\|z^{L,k}-z^{L,k-1}\|_\infty}{\|z^{L,k+1}-z^{L,k}\|_\infty}.}
\end{equation}
Since the optimization
problem fulfills all the hypothesis stated in 
\Cref{thm:bound3}
and \Cref{cor:distance1}, the numerical  decay rates  coincide with   the
expected  decay rates of  the detail coefficients
\cref{eq:dk-interp} for each choice of the prediction scheme.

\begin{figure}[!h]
	\centering
	\begin{tabular}{cc} 
		\includegraphics[width=0.4\textwidth,clip]{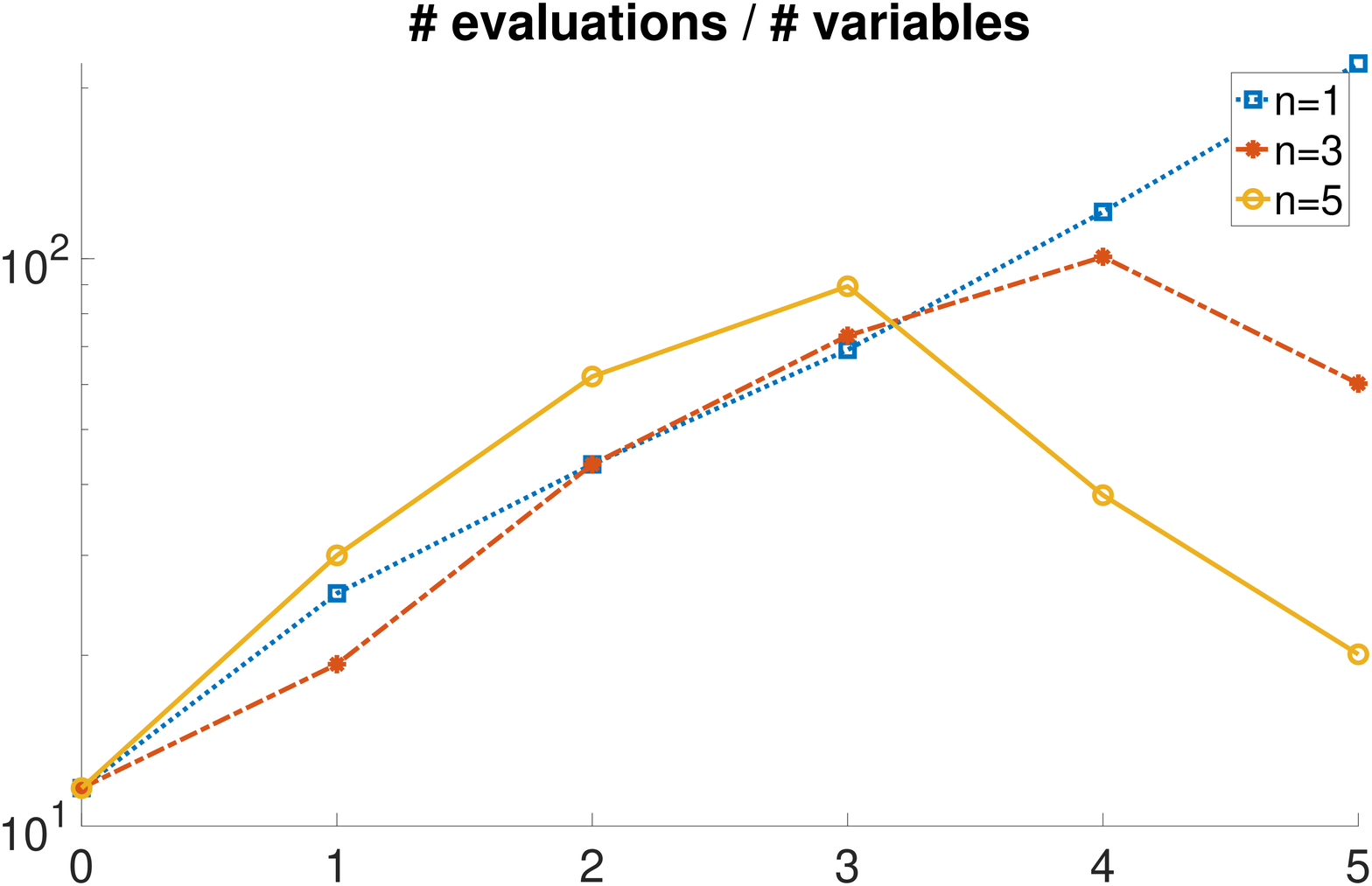} &
		\includegraphics[width=0.4\textwidth,clip]{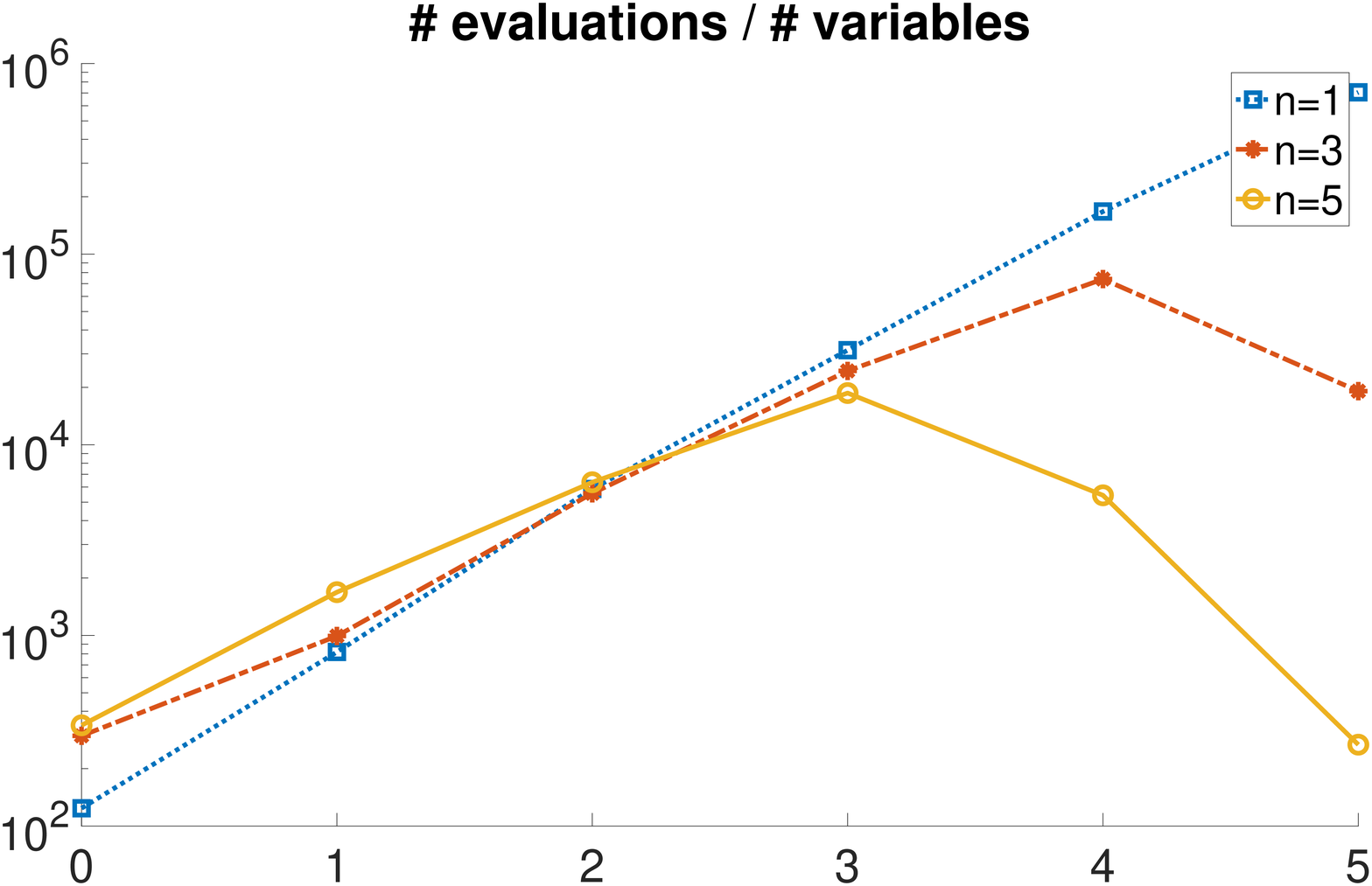}
	\end{tabular}
	
	\vspace{5pt}
	
	\begin{tabular}{|c|c|c|c|c|}
		\hline
		\# $F$-evaluations by & $n=1$ & $n=3$ 	&
		$n=5$ & Direct-$\cD$ \\\hline
		\texttt{fminunc} & 38\,678   &    17\,089    &    8\,910 & 49\,980\\
		\texttt{patternsearch} & 101\,307\,742  &   7\,938\,578  &   1\,063\,433  & 176\,168\,800 \\\hline
	\end{tabular} 
	\caption[1D Poisson problem]{1D BVP, as in
		\cref{fig:new-poisson1D_fminunc_tol6}.
		Ratio between the number of functional  evaluations and the
		number of degrees of freedom involved in the solution of  the $k$-th
		auxiliary problem, versus $k$. Left:
		$\cD=\texttt{fminunc}$; Right:
		$\cD=\texttt{patternsearch}$. Table:  Total number of 
		functional evaluations required to find $z^{L,L+1}$  in each
		case.
		\label{fig:eff-poisson1D_fminunc_tol6}} 
\end{figure} 

In \cref{fig:eff-poisson1D_fminunc_tol6} we examine the {\em
	efficiency} of the MR/OPT strategy. 
The table below the plots
clearly  shows that  the MR/OPT strategy leads to a significant  improvement in
efficiency compared to the direct   application of the optimization
tool.
Since  {\tt fminunc} is a
much more {\em efficient} optimizer for this type of  problems, its use leads to
a smaller number of functional evaluations, both 
when applied directly or when used in the MR/OPT
strategy. It is worth observing that  the  MR/OPT strategy  leads  to large  improvements in
efficiency: the total number of evaluations carried out with the
MR/OPT for  $n=5$, is less than  18\% for {\tt fminunc} and 0.7\%
for {\tt patternsearch}. 

We can see in the plots of  \cref{fig:eff-poisson1D_fminunc_tol6} that
the number of functional evaluations over the number of degrees of
freedom increases in a similar way for $n=1,3,5$ for the lower
resolution levels, but this ratio decays at the higher resolution
levels for $n=3,5$. A plausible explanation is that   the  distance between
the initial guess and the solution 
of  the $k$-th auxiliary problem  decreases with $n$ and $k$, as
specified in \cref{cor:distance1}, hence the
corresponding optimization problems can be solved in less
iterations. Thus, the efficiency improves with the approximation 
order of the prediction scheme used.

Notice, however, that even in the
`least efficient' case ($n=1$), the MR/OPT 
reaches  the desired solution more  
efficiently than with the direct optimizer,  {\em without any additional effort
	by the end-user}.

Our second test case is  the  P2D in \cite{FP14},
\begin{equation} \label{eq:2D-Pf1}
\begin{array}{rcl}
-(u_{x x}(x,y)+ u_{yy}(x,y))& = &f(x,y), \\[2pt]
u(x,y)& = &0, \end{array} \qquad
\begin{array}{l}
(x,y)\in\text{int}(\Omega), \\[2pt]
(x,y)\in\ \partial\Omega,\end{array}
\end{equation}
$\Omega=[0,1]\times [0,1]$, $\partial \Omega$ and $\text{int}(\Omega)$
denoting the boundary and the interior of $\Omega$, but  with a smooth, 
non-polynomial right hand side, $f(x,y)=\sin(4\pi x (1-x) y (1-y))$. 

A standard discretization of \cref{eq:2D-Pf1}  (using the classical  5-point  
Laplacian, see \cite{FP14}), 
on a uniform grid 
$\mathcal{G}:=(i/J,j/J)_{i,j=0}^{J}$,
leads to a  system of equations that can be
written in  matrix form  as $Az=b$, {$z(i,j)=u(i/J,j/J)+ O(h^2)$}, where
$h=1/J$ is the uniform mesh spacing of ${\cal G}$ and 
$A$ is symmetric and definite positive.
As in the 1D case, finding $z$ may be
formulated as a quadratic, convex,  minimization problem, but in  the 2D
case there are $N=(J-1)^2$ degrees of freedom, so that solving the
associated  minimization
problem is a much more demanding task, for the same accuracy in the
resulting approximation.

We solve the  minimization problem
using  the MR/OPT strategy  in combination with
the 2D-tensor product interpolatory MR  described in
\cref{sec:1D-int}.  For this, we consider   $J=J_L$,  ${\cal G}=
{\cal G}_L$  as the finest  mesh in a ladder of nested
uniform grids on $\Omega$ of the form ${\cal G}_k=
(i/J_k,j/J_k)_{i,j=0}^{J_k}$, $J_k=J_{k+1}/2$, 
$0 \leq k < L$. The discrete
spaces in  the corresponding 2D-interpolatory MR 
setting have dimension $(J_k+1)^2$.  In our computations, we take
zero initial data, 
$J_L=128$,   $L=5$,  $tol_M = tol_{\cD} = 10^{-7}$ and ${\cal D}=${\tt
	fminunc}.

\cref{fig:poisson2D_sin_fminunc_tol6}  shows several plots  analogous
to  those in 
\cref{fig:new-poisson1D_fminunc_tol6}-\cref{fig:eff-poisson1D_fminunc_tol6},
where
we observe  the same behavior as in the 1D case: The actual errors between the
auxiliary solutions and the true solution decrease, and  also the efficiency of the MR/OPT
strategy increases,  with the
approximation order of the prediction scheme. 
The decay rate of the  max-norm errors between the sub-optimal solutions and the
true solution, and also between successive sub-optimal solutions, coincides
with the order of approximation of the prediction 
operators   (see
\cref{tab:poisson2D_sin_fminunc_tol6_decay_x}),  which  is
consistent with \cref{cor:distance1}.

It is worth noticing that in this test case  the max-norm of the
difference between sub-optimal solutions falls below $tol_M=tol_{\cal
	D}$ at $k=3$ for $n=5$, and $k=4$ for $n=3$,  which makes  the
MR/OPT strategy stop\footnote{For larger  values of $tol_{\cal
		D}=tol_M$, the strategy stop even at lower resolution levels.}.  We
also remark the very large increment in efficiency obtained for
$n=3,5$, compared to $n=1$, which is  probably due to the smoothness
of the solution. For $n=5$, the solution has been obtained with only 
$0.1\% $ of the functional evaluations required by the direct use of
{\tt fminunc}.

\begin{figure}[!h]
	\centering
	\begin{tabular}{ccc} 
		\includegraphics[width=0.3\textwidth,clip]{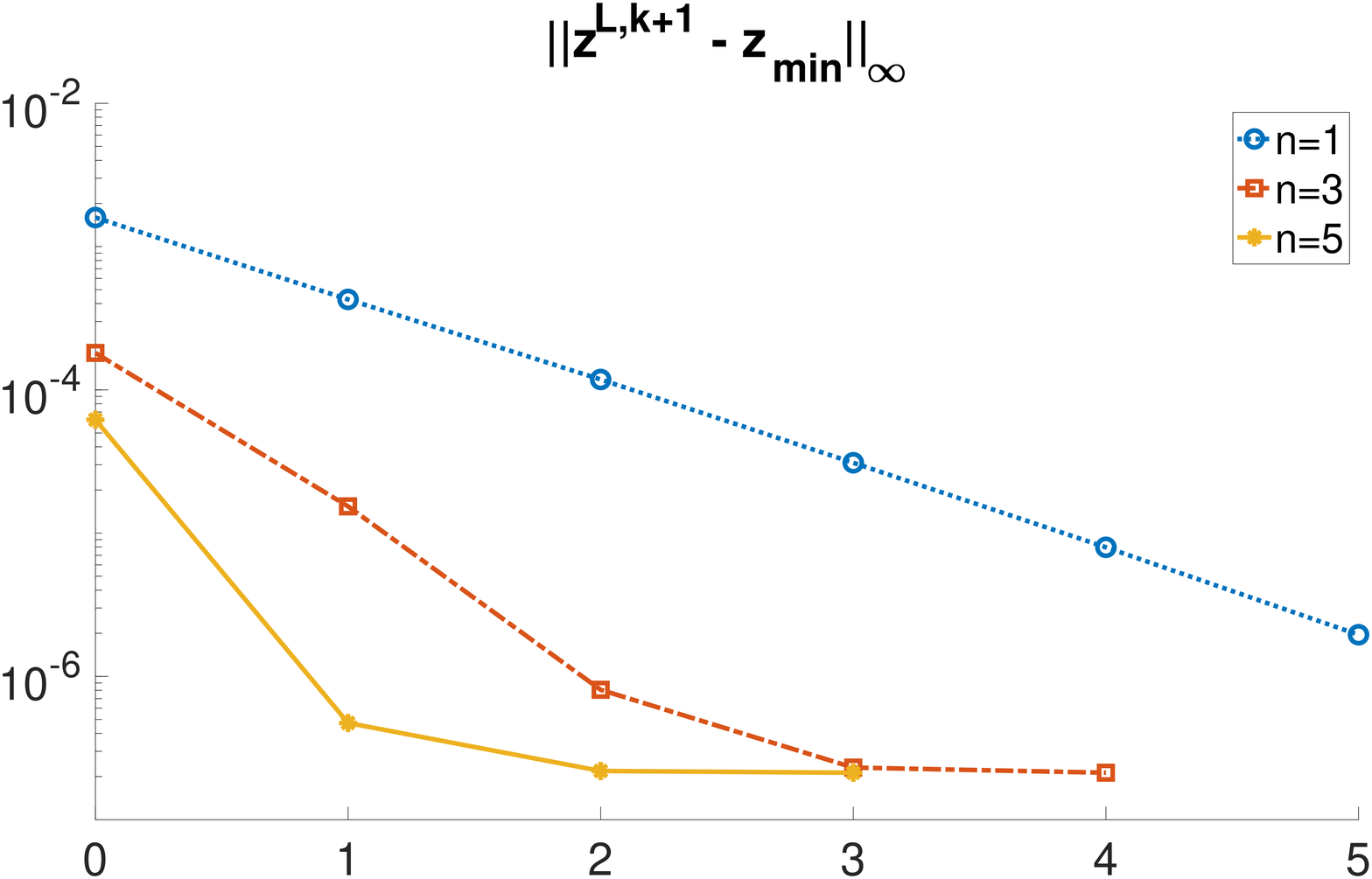}
		&
		\includegraphics[width=0.3\textwidth,clip]{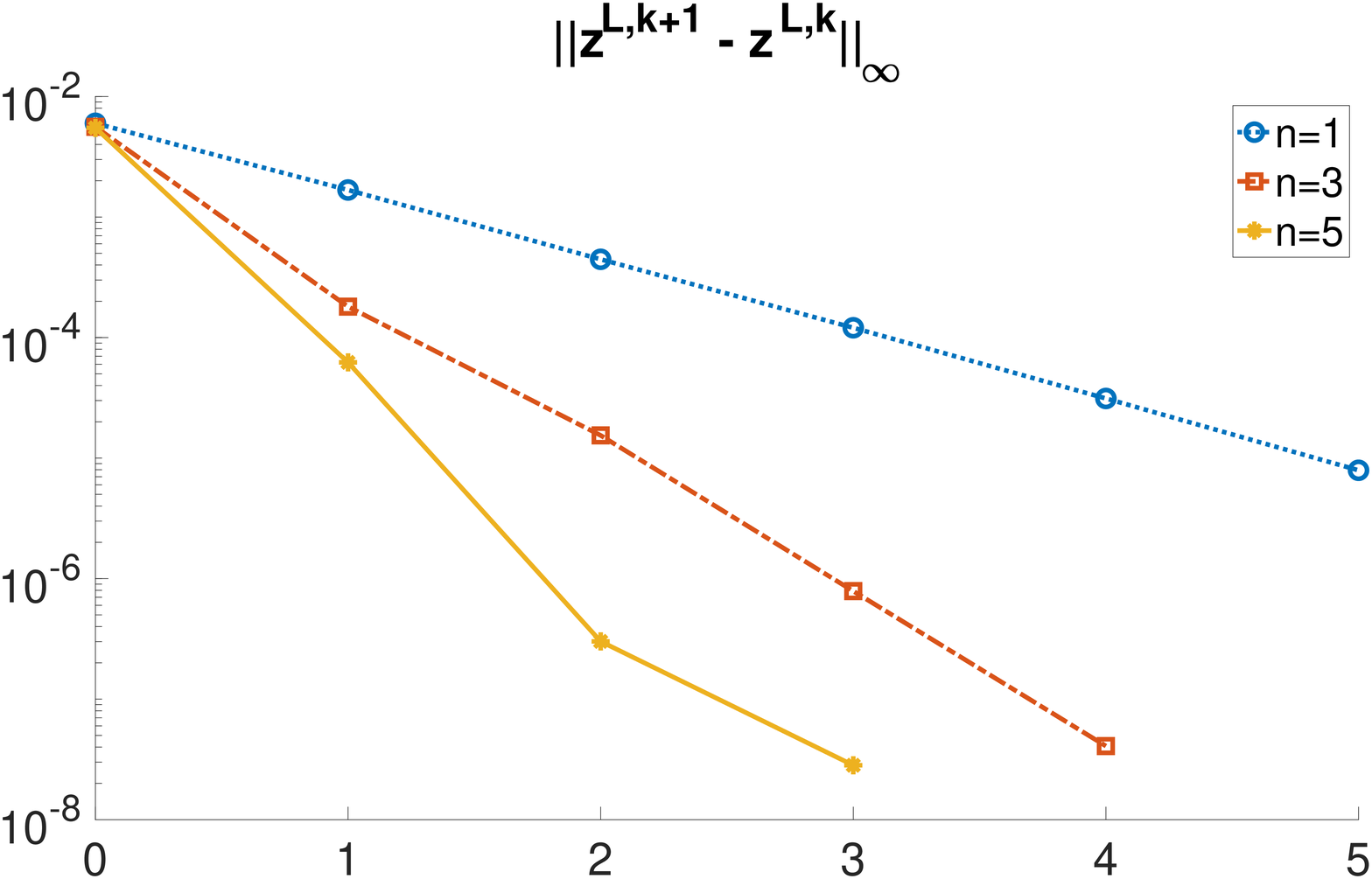} &
		\includegraphics[width=0.3\textwidth,clip]{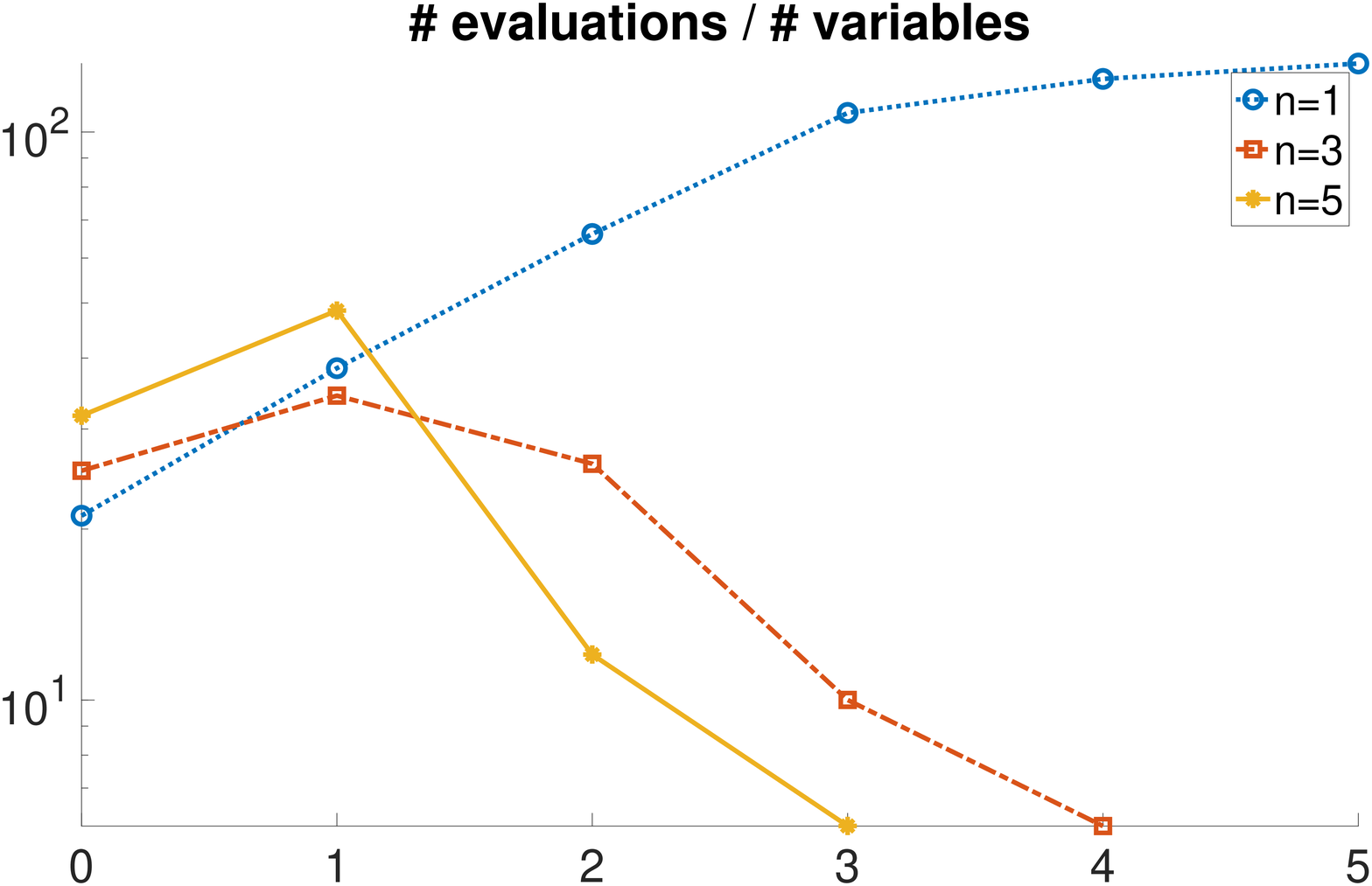}
	\end{tabular} \\
	\begin{tabular}{|c|c|c|c|c|}
		\hline
		n & 1 & 3 	& 5 & Direct-$\cD$\\\hline
		\#  of $F$-evaluations& 2\,742\,108   &   41\,206    &   11\,136 & 12\,581\,010 \\\hline
	\end{tabular}
	\caption[2D Poisson problem]{ 2D Poisson problem
		\cref{eq:2D-Pf1}, $L=5$, $tol_M = tol_{\cD} =
		10^{-7}$. ${\cal D}$ is
		\texttt{fminunc}.
		\label{fig:poisson2D_sin_fminunc_tol6}}
\end{figure} 
\begin{table}[!h]
	\centering
	\begin{tabular}{|c|c|c|c|c|c|c|}
		\hline
		\(n\backslash k\) & 1 & 2 & 3 & 4 & 5 & Theoretical rate\\\hline
1&1.84&1.91&1.89&1.95&1.98&2\\\hline
3&5.27&3.36&4.14&4.27&-&4\\\hline
5&6.72&7.59&3.62&-&-&6\\\hline
	\end{tabular}
	\caption{Test-case: 2D Poisson problem.  Decay rate
		\cref{eq:decay-num}  for the data in
		\cref{fig:poisson2D_sin_fminunc_tol6}.   } \label{tab:poisson2D_sin_fminunc_tol6_decay_x}
\end{table}
\subsection{Non-quadratic problems} \label{sec:nonquadratic}
\label{sec:minimal}
In this subsection we shall consider:  a convex, but 
non-quadratic, minimization problem 
and a non-convex problem, in order  to evaluate the performance of our MR/OPT
strategy, in more general  situations 

We consider  first  the MINS problem in \cite{FP14}, which arises from
a discretization of the  following  minimal surface problem
\begin{equation} \label{eq:minimal}
\min_u \int_{\Omega} \sqrt{1+\|\nabla u(x,y)\|_2^2} \ d(x,y),
\end{equation}
with the boundary conditions
\[u_0(x,y)=\begin{cases} x(1-x), & \text{if } y\in\{0,1\}, \\
0, & \text{otherwise}.
\end{cases}\]
Its solution is approximated in \cite{FP14} by the solution of the  minimization
problem that results from considering  as objective function
\[ F(z):= \frac{1}{2J^2} \sum_{i,j=0}^{J-1} \sqrt{1+a^2+b^2} + \sqrt{1+c^2+d^2},\]
with
\begin{eqnarray*}
	a =& J(z_{i,j+1}-z_{i,j}) , \qquad &b = J(z_{i+1,j+1}-z_{i,j+1})  , \\
	c =& J(z_{i+1,j+1}-z_{i+1,j}) , \quad  &d = J(z_{i+1,j}-z_{i,j}).
\end{eqnarray*}

We use the interpolatory 2D framework and the MR/OPT strategy to solve
this  nonlinear, convex, optimization problem,  with
$J_L=128$,   $L=5$,   $tol_M = tol_{\cD} = 10^{-6}$ and 
{$z^{L,0}_{i,j} = \frac{i}{J}\left (1-\frac{i}{J}\right )$,
	$i,j\in\{0,1,\ldots,J\}$}, which complies with the specified boundary conditions. The solution obtained with MR/OPT and
$\cal D$={\tt fminunc} with the above parameters is represented in
\cref{fig:minimal_morebv_solutions}-left.

In  this case, we do not know the exact solution and in
\cref{fig:minimal_fminunc_tol6}-left  we only plot  the max-norm  of the difference between sub-optimal solutions.  The
observed decay rate is $r \approx 2$, for all prediction schemes. To explain
this apparent contradiction with the results in \cref{cor:distance1},  we carry out a  numerical
inspection of the derivatives of the  solution shown in
\cref{fig:minimal_morebv_solutions}-left.  We observe large spikes in
third order derivatives at the four corners of the unit square,
which 
is an indication  that the obtained discrete data can only be associated to a $\cC^2$ function.  Hence,
the observed decay rate is, in
fact, the one predicted by the theoretical result in
\cref{cor:distance1} since,  in this case, the interpolation error corresponding to
an $n$-th degree polynomial $(n\geq 1)$ is only $O(h^2)$, with $h$ the
distance between the points in the grid.

The plot in \cref{fig:minimal_fminunc_tol6}-right shows, again, the
increase in efficiency of the MR/OPT strategy for $n=3,5$, which we
assume to be related to the fact that the errors between consecutive
sub-optimal solutions are smaller than for $n=1$, hence the
corresponding auxiliary optimization problems require less
iterations. For $n=5$, the solution of the original optimization
problem, shown in \cref{fig:minimal_morebv_solutions}-right,  is obtained with less than 2 \% of the total number of
evaluations of $F$, with respect to the direct aplication of ${\cal D}$.
\begin{figure}[!h]
	\centering
	\begin{tabular}{cc} 
		\includegraphics[width=0.45\textwidth,clip]{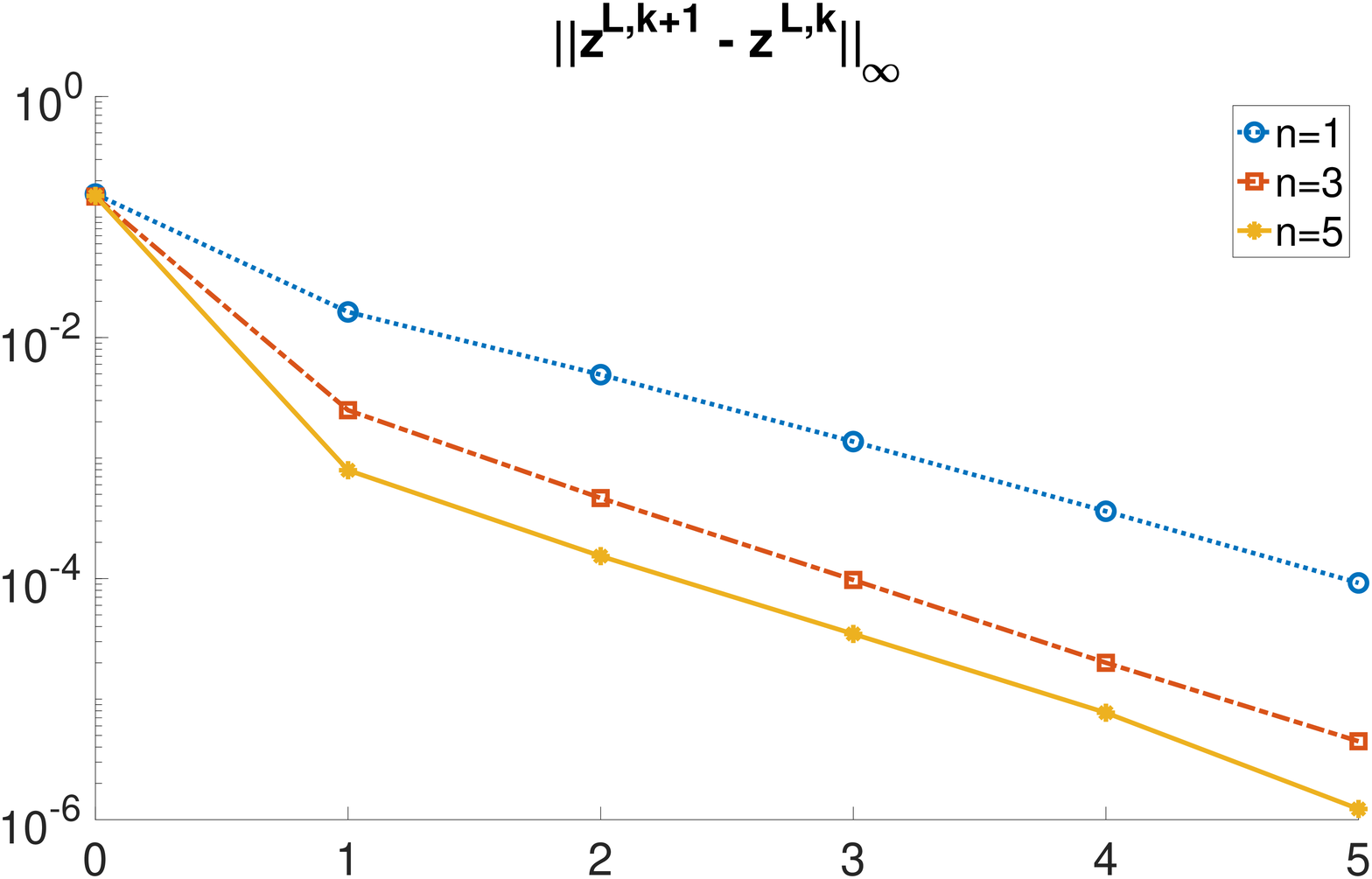} &
		\includegraphics[width=0.45\textwidth,clip]{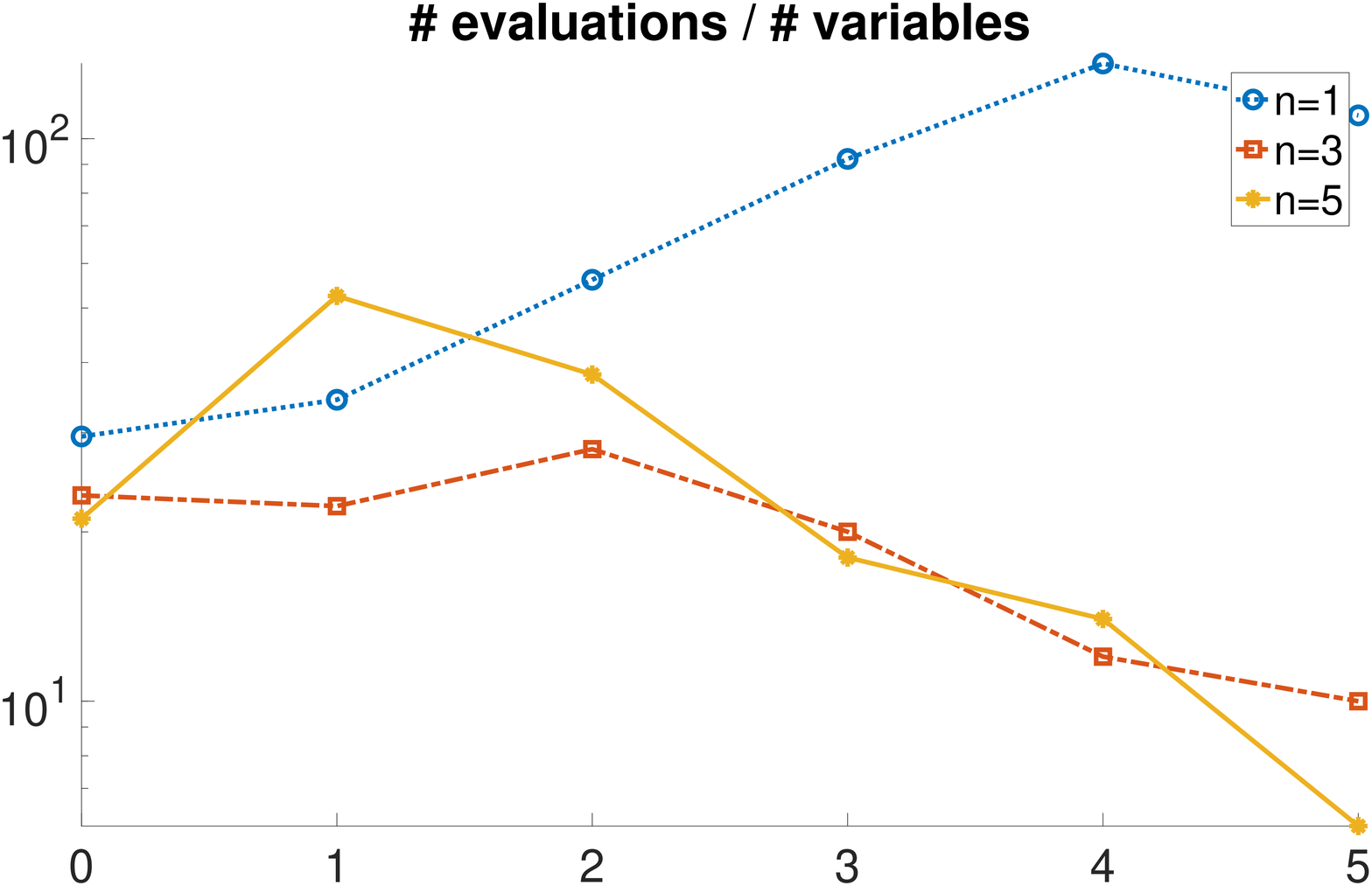} \\
	\end{tabular}
	\begin{tabular}{|c|c|c|c|c|}
		\hline
		interpolation degree & n=1 & n=3 	& n=5 & Direct-$\cD$\\\hline
		total \# evaluations &       2\,417\,132   &   235\,771    &   180\,990 & 10\,226\,103 \\\hline
	\end{tabular}
	\caption[Minimal surface problem]{Minimal surface problem. $L=5$, $tol_M = tol_{\cD} = 10^{-6}$ and ${\cal D}$ is
		\texttt{fminunc}.  \label{fig:minimal_fminunc_tol6}}
\end{figure} 

Our last test concerns  a nonlinear
non-convex optimization problem from  \cite[section 5.1.4]{FP14}: the MOREBV problem,
which is obtained from a discretization on
a uniform grid  on $\Omega=[0,1]\times [0,1]$ of the following BVP
\begin{equation} \label{eq:morebv}
\begin{array}{rcl}
-(u_{x x}(x,y)+ u_{yy}(x,y))+ \frac12 (u(x,y)+x+y+1)^3& = &0, \quad (x,y) \in \text{int}(\Omega),\\[2pt]
u(x,y)& = &0, \quad (x,y) \in {\partial \Omega}. \end{array} 
\end{equation}
Using the classical 5-point discretization of the Laplacian, the
resulting system of nonlinear equations can be rewritten as a
nonlinear  least-squares problem with a non-convex  objective function
given by the expression (see \cite{FP14} and references therein). The
theoretical results of \cref{sec:theory} do not apply in this case.
\begin{multline}
	F(z):= \sum_{i,j=1}^{J-1} \Big ( (4z_{i,j}-z_{i-1,j} - z_{i+1,j}-z_{i,j-1} - z_{i,j+1})  \\
	+ \frac1{2J^2}\left ( z_{i,j} + i/J+j/J+1\right )^3 \Big )^2.
\end{multline}
We solve this problem by using the MR/OPT strategy within the 2D
interpolatory MR framework used in the previous two test cases. As before, setting $\bar z= \vec{0}$  leads to a
sequence of sub-optimal solutions satisfying the boundary
conditions.  The solution obtained  with the same parameters as in the
previous test case ($n=5$) is shown in
\cref{fig:minimal_morebv_solutions}-right. 

In \cref{fig:morebv_papini_fminunc_tol6}  we show the results obtained
for this test case. 
The left plot in this figure shows that  when taking $n=1$ the distance
between consecutive 
sub-optimal solutions does not decay\footnote{The set of sub-optimal solutions can be found in the Appendix.}, but it does decrease for $n=3,5$
(at a rate $r \approx 2$).  As a consequence, the  number of
calls to  the objective function increases dramatically for $n=1$,
while it remains moderate for $n=3,5$. More research is required to
explain the behavior of the MR/OPT strategy for this non-convex
optimization problem, for which no theoretical results are provided in this paper. We  have observed  that, as in the MINS problem, certain  third order  finite differences of $z^{L,L+1}$ also  exhibit peaks in the corners of $\Omega$. We do remark that  the solution
corresponding at $n=5$ has been obtained with less than $0.3\%$ of the
total number of functional evaluations than for $n=1$ (which is, at least, less than 69\% of the  total number of evaluations for the
direct use of ${\cal D}$).

\begin{figure}[!h]
	\centering
	\begin{tabular}{cc} 
		\includegraphics[width=0.45\textwidth,clip]{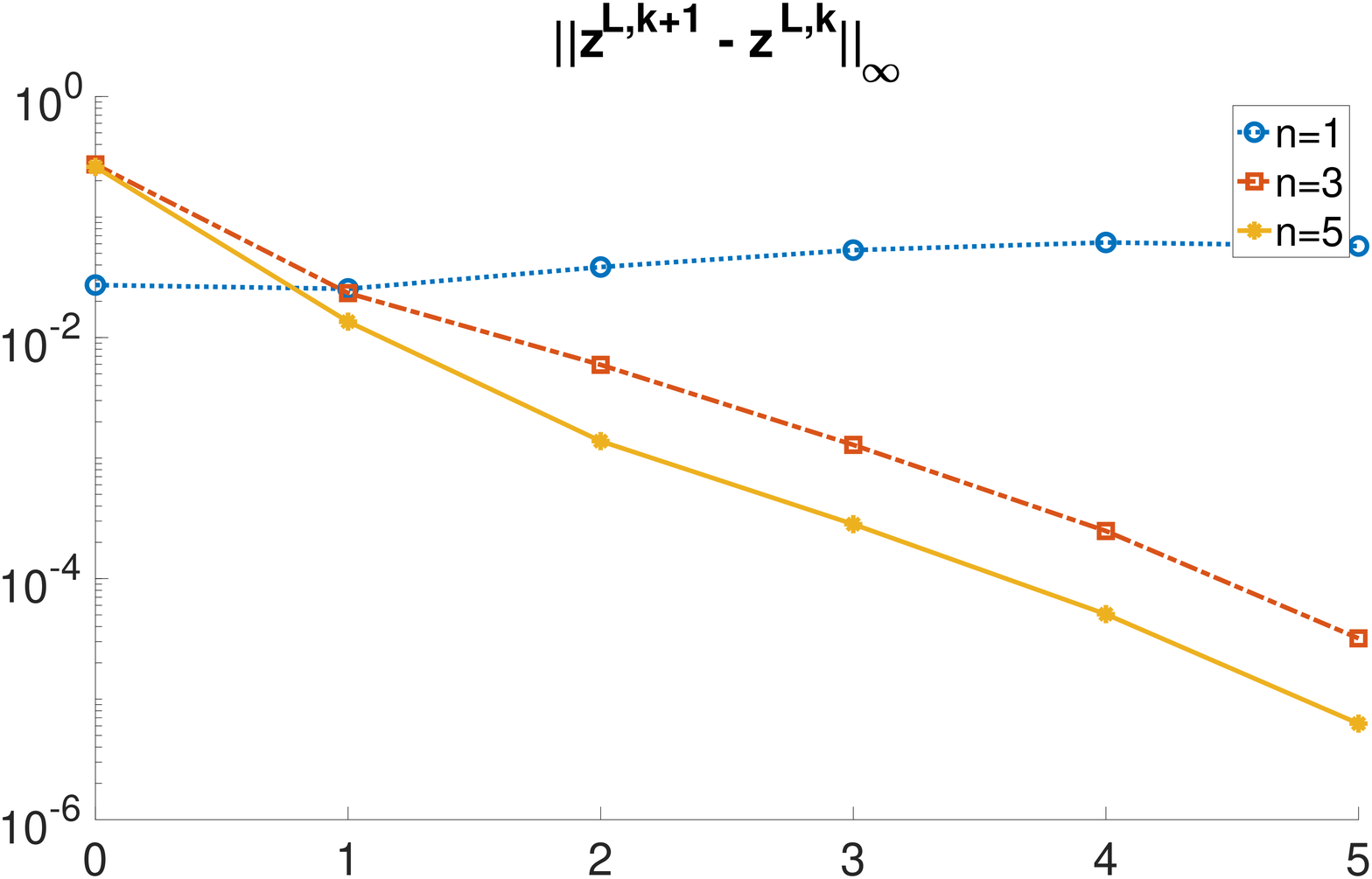} &
		\includegraphics[width=0.45\textwidth,clip]{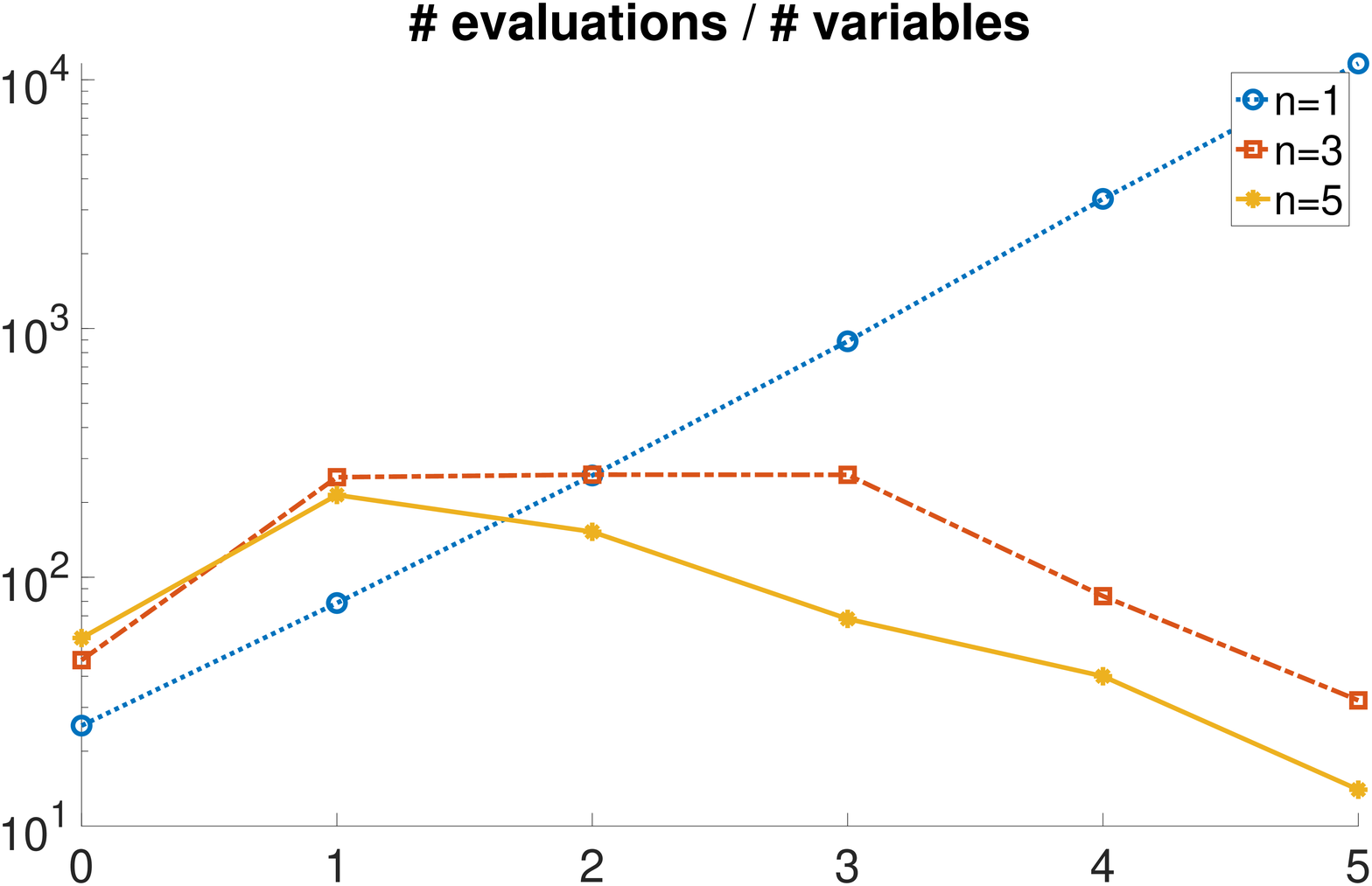}
	\end{tabular}
	\begin{tabular}{|c|c|c|c|c|}
		\hline
		interpolation degree & n=1 & n=3 	& n=5 & Direct-$\cD$\\\hline
		total \# evaluations &     201\,583\,677  &    1\,168\,621   &   495\,258 & 294\,879\,519 \\\hline
	\end{tabular}
	\caption[MOREBV problem]{ MOREBV problem. $L=7$, $tol_M = tol_{\cD} = 10^{-6}$ and ${\cal D}$ is
		\texttt{fminunc}.  \label{fig:morebv_papini_fminunc_tol6}}
\end{figure}

\begin{figure}[!h]
	\centering
	\includegraphics[width=0.45\textwidth,clip]{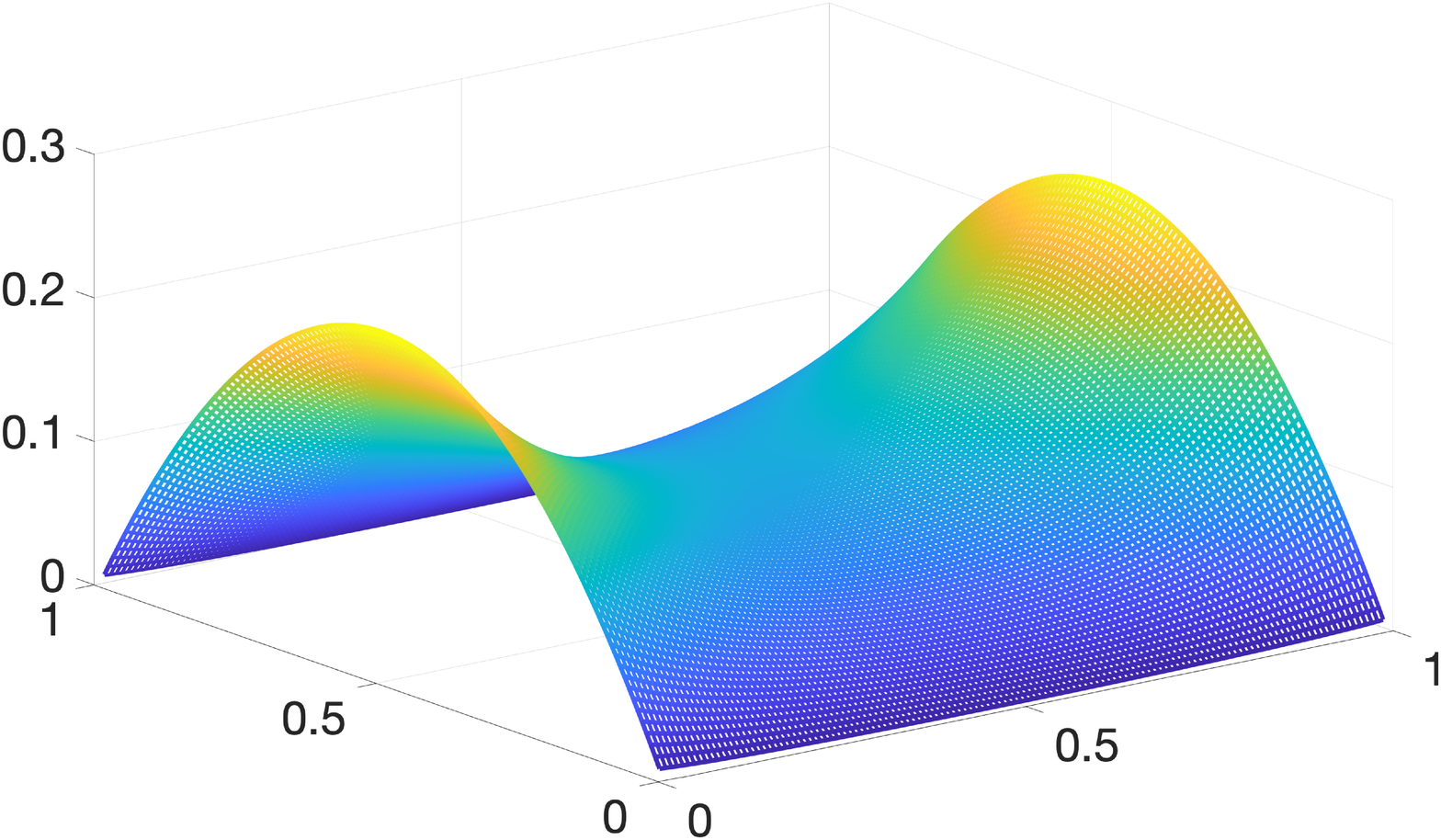}
	\includegraphics[width=0.45\textwidth,clip]{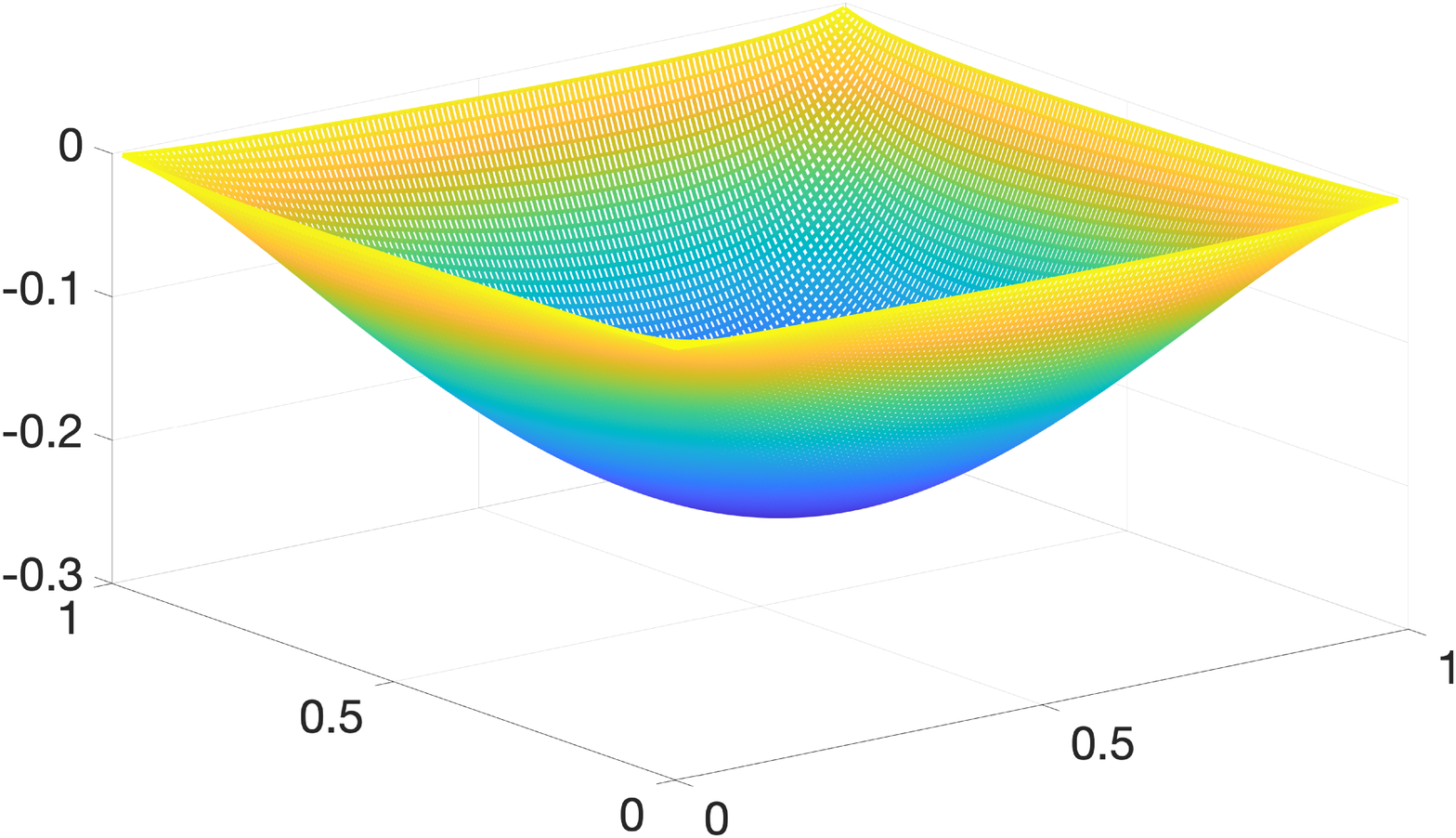}
	\caption[Optimal solutions]{The computed solution, $z^{L,L+1}$, of the minimal surface (left) and the MOREBV (right) problems taking $n=5$.  \label{fig:minimal_morebv_solutions}}
\end{figure}

\section{Conclusions And Perspectives} \label{sec:conclusions}

We have described and analyzed a multilevel  strategy  (MR/OPT)  which can be used to reduce the
cost of solving large scale 
optimization problems. MR/OPT uses  Harten's MRF to define 
a hierarchy of auxiliary problems that provide a
sequence of {\em sub-optimal solutions}, one at each resolution level, so that
at the
last step  the full (large-scale) optimization problem is solved, but with an 
initial guess that is much closer to the desired solution than the original
initial guess (which is assumed to be provided by the end-user). 

The solutions of the auxiliary problems depend on the prediction operators of the chosen MRF. The results in this paper
clearly demonstrate that the accuracy of the prediction scheme plays a
significant role in the overall performance of the MR/OPT strategy,
when the
initial guess and the solution can be interpreted as point-value
discretizations of sufficiently smooth functions.

Harten's MRF provides a
flexible and versatile multilevel 
structure that allows to consider both the {optimization
	tool}  and the objective function \cref{eq:problema} as black
boxes, which is an important asset  in engineering applications. If
an optimization problem such a \cref{eq:problema} can be solved with a
specific optimization tool, ${\cal D}$, MR/OPT will reduce the cost of
finding its solution. For convex optimization problems, we have
presented several theoretical results on the properties and the
performance of the MR/OPT strategy, which have been numerically
validated in a series of examples from the specialized literature. 
More analytic work is necessary in order to determine the properties
of the MR/OPT strategy in more general situations,  as well as the influence of
the underlying minimization code or the specification of constraints.

The strategy can be improved in various ways, that we intend to
investigate in future papers. In particular, it might  be
advantageous to consider a level-dependent tolerance
criterion for the optimizer $\cD$: large tolerances in coarse levels
and small ones in fine levels. Following this idea, to consider
different optimizers might also be profitable: e.g. global searchers in the
first levels and, later, fast/local optimizers. Another line of
research concerns the use of   nonlinear reconstruction techniques
instead of  the linear ones used in this paper,  since they might  be more
efficient in problems where the solution presents discontinuities or
it is important to preserve  specific features of the solution, such as monotonicity or convexity.

\bibliographystyle{plain}

\appendix

\section{On the Property S for subdivision schemes}
For completeness, in this section we provide a proof, without entering into too many details, which is well-known in subdivision schemes theory.
\begin{lemma} \label{lem:subdivision_ineq}
	Let ${\cal P}=(P_k^{k+1})_{k=0}^\infty$ be  the sequence of
	linear, interpolatory,  prediction operators defined in
	\Cref{sec:1D-int}. 
	Then there exists $C>0$ such that
	\[ \|\varepsilon^\ell\|_\infty \leq \|P_\ell^k \varepsilon^\ell \|_\infty \leq C \|\varepsilon^\ell\|_\infty, \qquad \forall 0\leq\ell\leq k.\]
\end{lemma}
\begin{proof}
	Any interpolatory prediction operator fulfills
	\[ \|\varepsilon^\ell\|_\infty = \sup_{0\leq i \leq J_\ell} |\varepsilon_i^\ell |  = \sup_{0\leq i \leq J_\ell} |(P_\ell^{\ell+1} \varepsilon^\ell)_{2i} | \leq \sup_{0\leq i \leq J_{\ell+1}} |(P_\ell^{\ell+1} \varepsilon^\ell)_{i} | = \|P_\ell^{\ell+1} \varepsilon^\ell \|_\infty.\]
	Applying this inductively, we obtain  the first inequality 
	\[ \|\varepsilon^\ell\|_\infty  \leq \|P_\ell^{\ell+1} \varepsilon^\ell \|_\infty \leq \| P_{k-1}^k \cdots  P_{\ell}^{\ell+1} \varepsilon^\ell \|_\infty = \|P_\ell^k \varepsilon^\ell\|_\infty.\]
	
	The second inequality is a consequence of the fact that the
	subdivision process associated to the recursive application of
	the predictions operators is  convergent (see
	e.g. \cite{Dyn92_copy, ADH98a_copy, DGH03_copy}), which means that  for each
	vector $\varepsilon^\ell$ there exists a continuous function,
	which we denote by $P_\ell^\infty
	\varepsilon^\ell\in\cC([0,1],\R)$, such that 
	\[ \lim_{\ell <m\to\infty} \sup_{0\leq i \leq J_m} |(P_\ell^m
	\varepsilon^\ell)_i - (P_\ell^\infty \varepsilon^\ell)(h_m i)|
	= 0, \quad \]
	with  $h_m = J_m^{-1}$ being the  grid spacing at $m$-level.
	For interpolatory prediction operators	$ (P_\ell^k \varepsilon^\ell)_i = (P_\ell^\infty
	\varepsilon^\ell)(h_k i)$, $ 0\leq i \leq J_k$, $0 \leq k$, by  the
	interpolation condition,  hence 
	\[ \|P_\ell^k \varepsilon^\ell \|_\infty  = 
	\sup_{0\leq i \leq J_k} |(P_\ell^k \varepsilon^\ell)_{i} |  = 
	\sup_{0\leq i \leq J_k} |(P_\ell^\infty \varepsilon^\ell)(h_k i) | \leq
	\sup_{t \in [0,1]} |(P_\ell^\infty \varepsilon^\ell)(t) |.
	\]

	Since $\varepsilon^\ell =\sum_{i=0}^{J_\ell} \varepsilon_i
	\delta_i^\ell$, with  $\delta_i^\ell$ the canonical $i$-th vector in
	$\R^{N_\ell}$, and the prediction operators are linear, the result will
	follow from the properties of the limit functions 
	$\phi_i^\ell=
	P_\ell^\infty \delta_i^\ell$. It can be proven (see
	\cite{ADH98a_copy,DGH03_copy}) that these limit functions have compact
	support (narrower than $2n h_\ell$, with $n$ as in \Cref{sec:1D-int})
	and satisfy  a refinability property  which ensures that there exists $\tilde C>0$ such
	that $||\phi^\ell_i||_\infty \leq \tilde C$, $\forall i, 0 \leq i \leq
	J_\ell$, $\forall \ell \geq 0$.  Given $\varepsilon^\ell \in
	\R^{N_{\ell}}$,
	\[ P_\ell^\infty \varepsilon^\ell(t) = 
	\sum_{i=0}^{J_\ell}\varepsilon^\ell_i \phi_i^\ell(t), \quad t\in[0,1] 
	\quad  \rightarrow \quad  |(P_\ell^\infty \varepsilon^\ell)(t)| \leq
	\|\varepsilon^\ell\|_\infty \sum_{i=0}^{J_\ell} |\phi_i^\ell(t)|
	\]
	the result follows from observing that for any $t \in[0,1]$,  $\phi^\ell_i(t) \neq 0$ only for 
	at most $2n$ such functions  (see \cite{ADH98a_copy,DGH03_copy}). Hence
	for some  $C>0$, which does not depend on $\ell$, the second inequality is obtained.
\end{proof}

\section{Limit function basis of the prediction operators}

The prediction operators described in \cref{sec:1D-int} are based on $n$-th degree polynomial interpolation, with $n=1,3,5$. Since they are linear operators, the refined data can be written in terms of a basis. In \cref{fig:limit3,fig:limit5}, some limit function basis for $n=3,5$, respectively, are shown.

\begin{figure}[!h]
	\centering
	\begin{tabular}{cc} 
		\includegraphics[width=0.45\textwidth,clip]{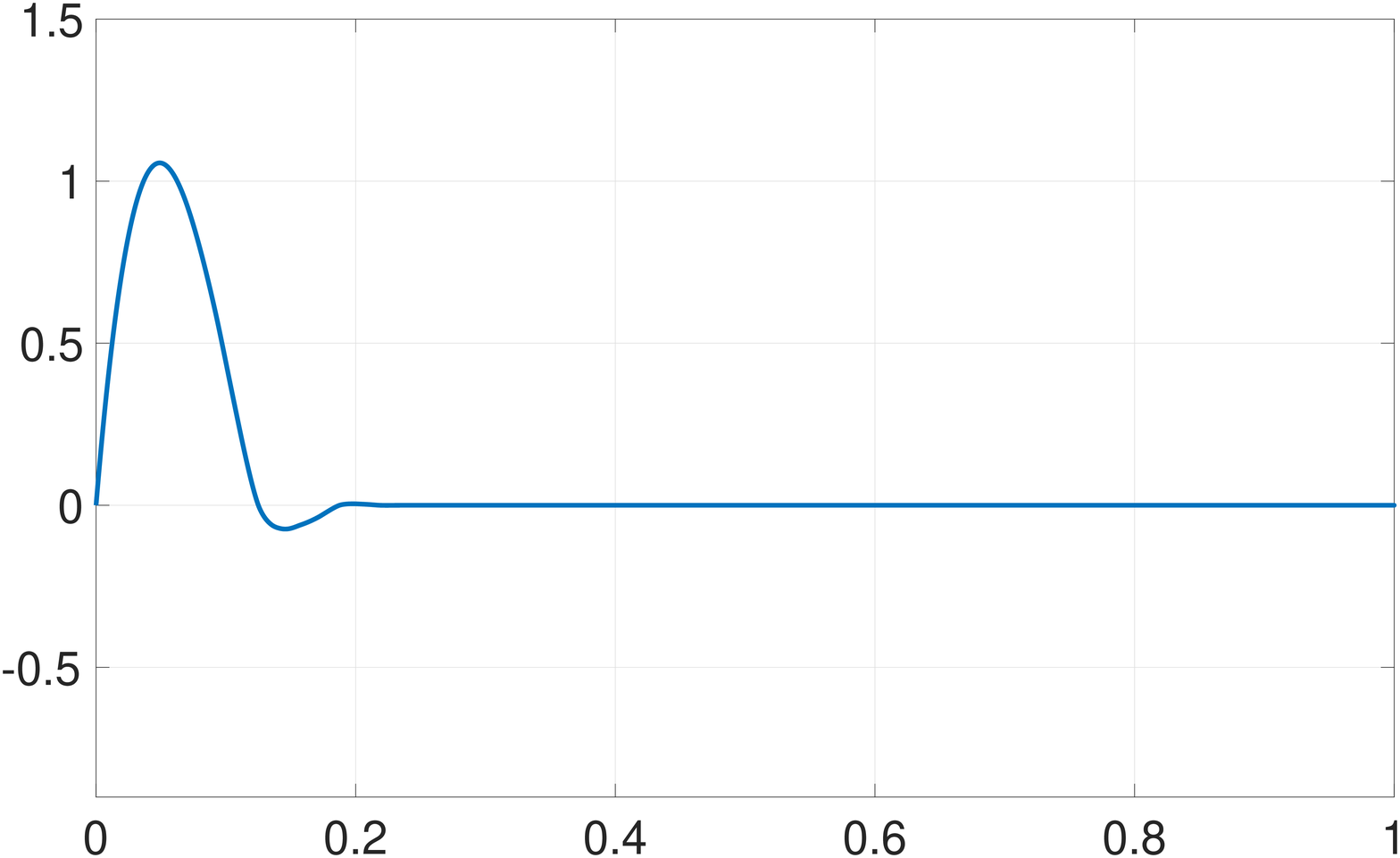} &	\includegraphics[width=0.45\textwidth,clip]{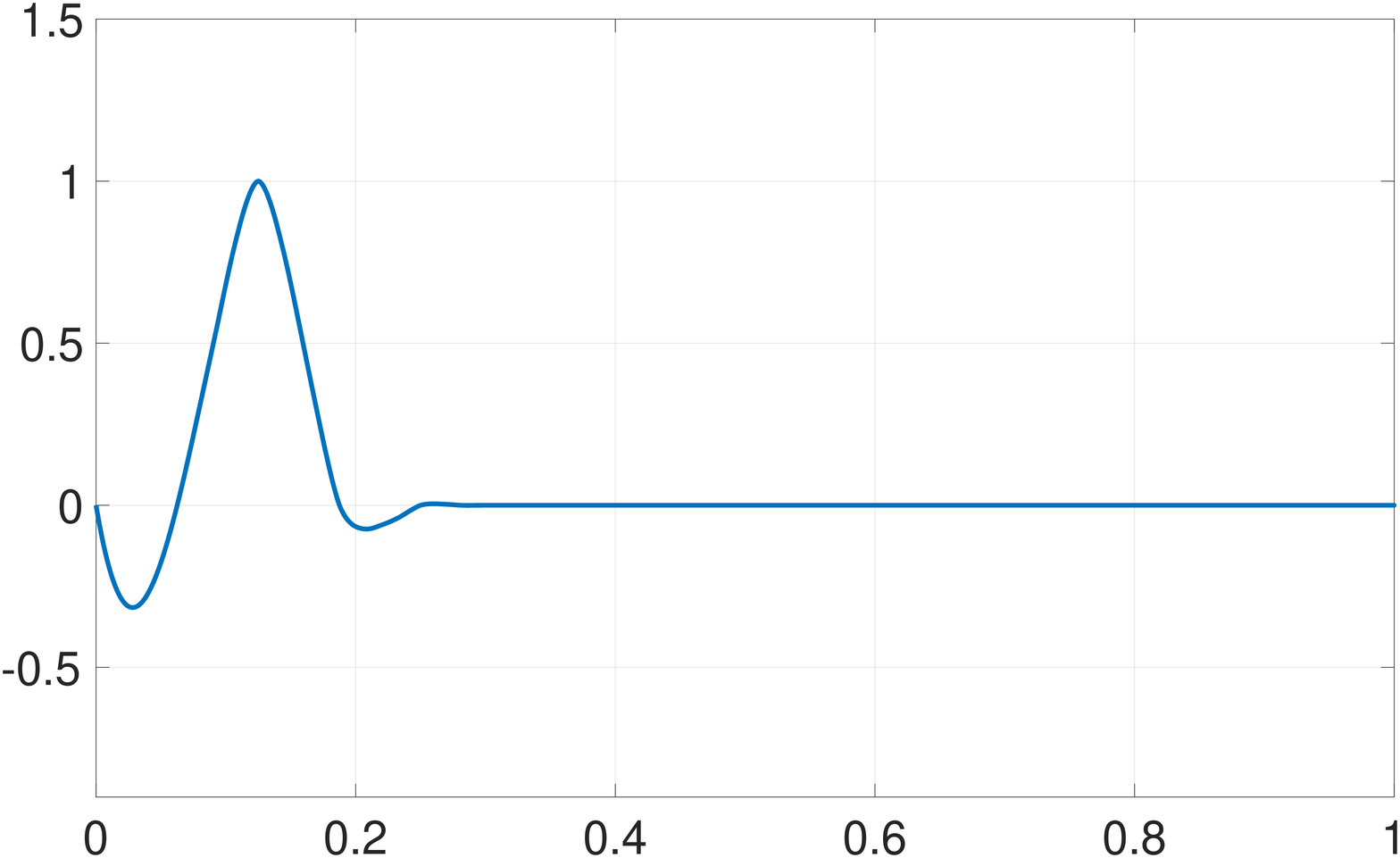} \\
		\includegraphics[width=0.45\textwidth,clip]{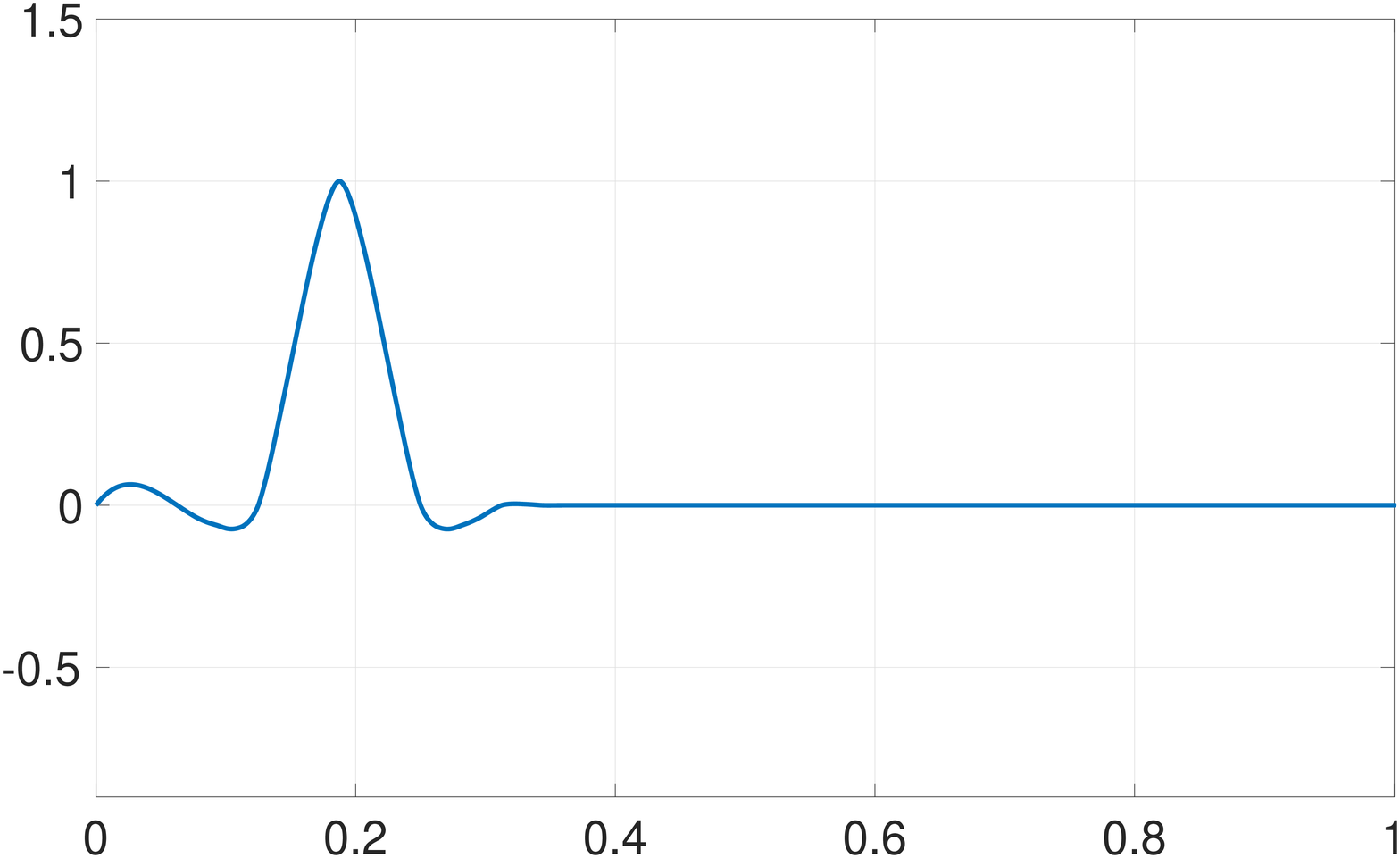} &	\includegraphics[width=0.45\textwidth,clip]{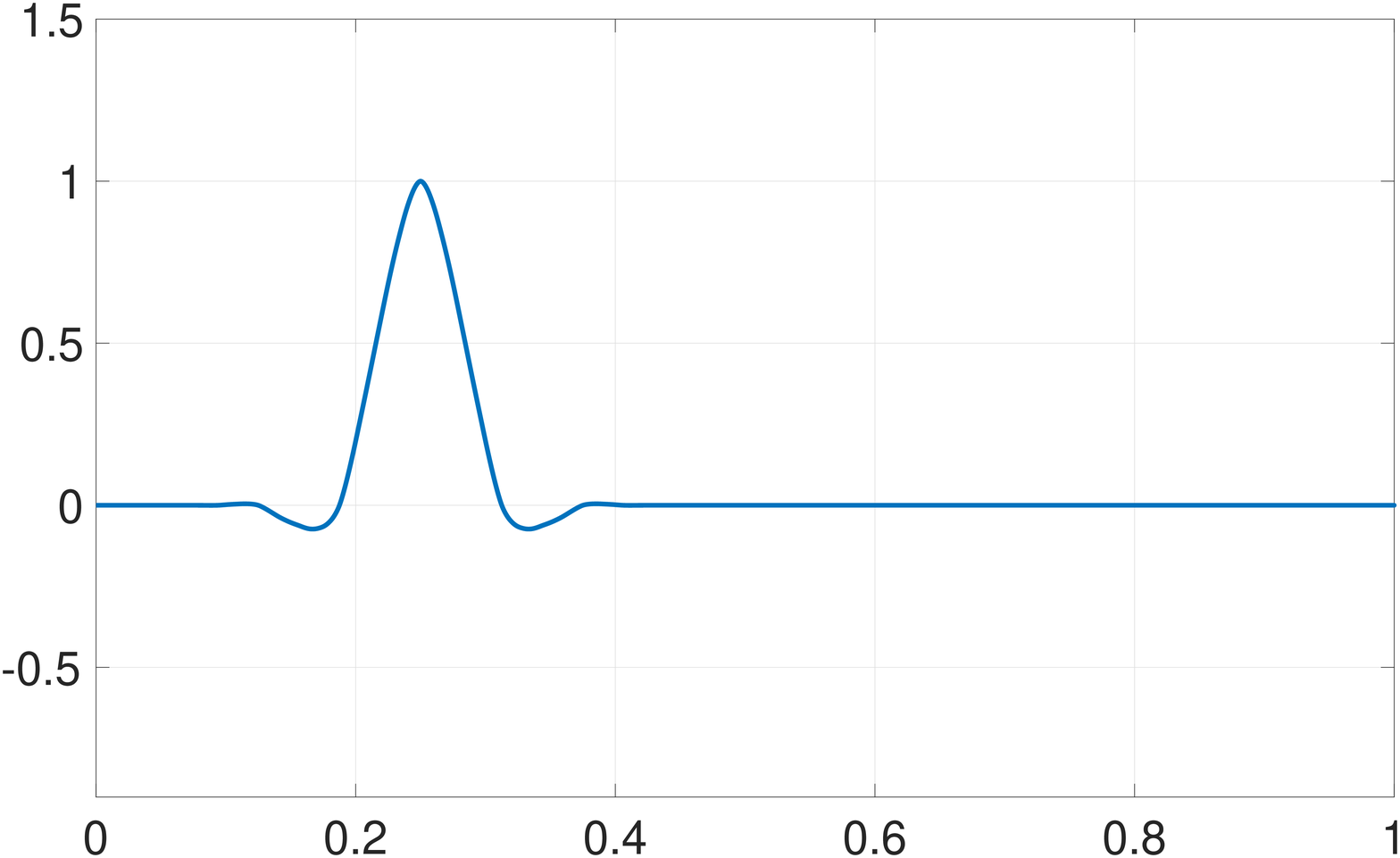} \\
		\includegraphics[width=0.45\textwidth,clip]{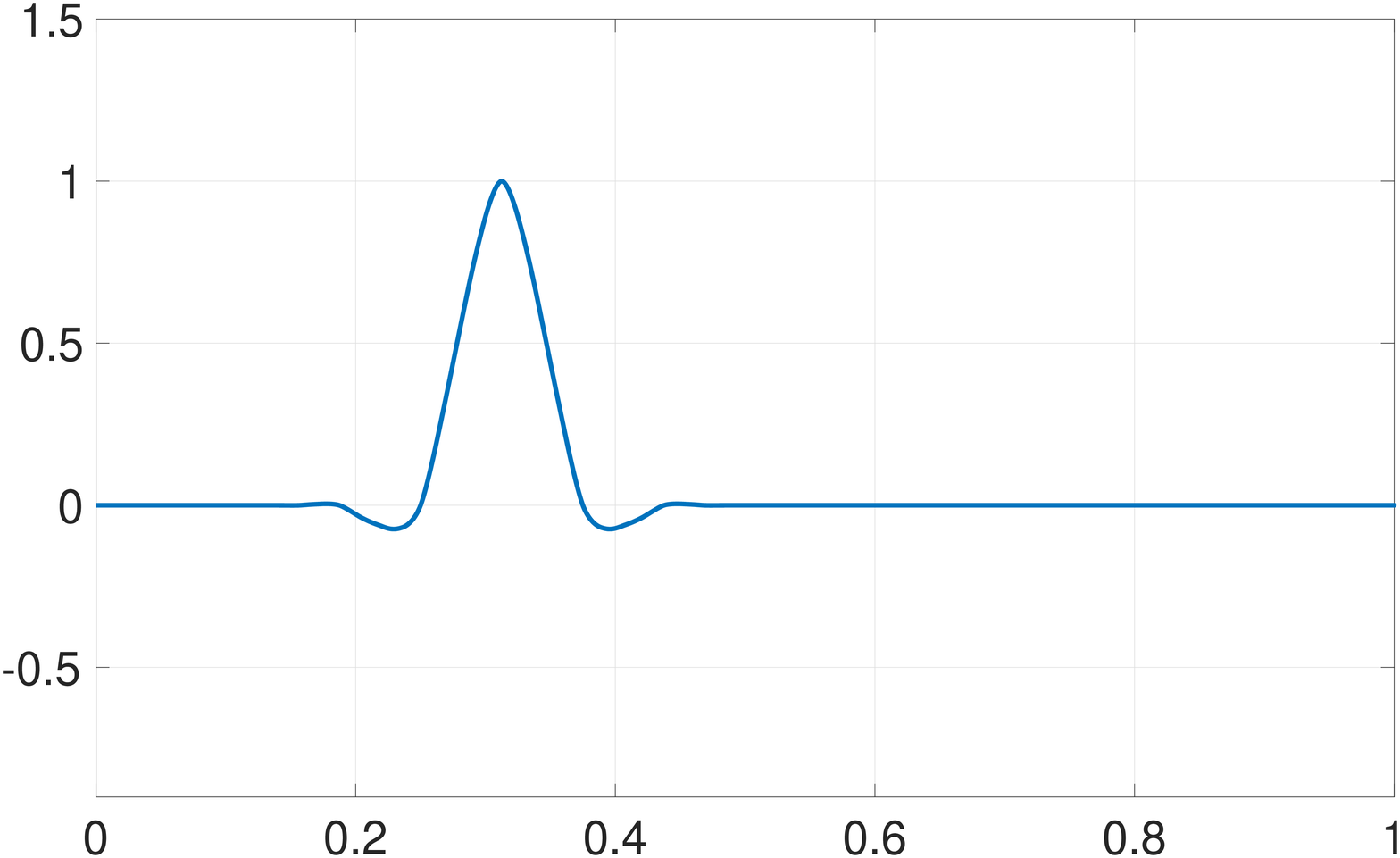} &	\includegraphics[width=0.45\textwidth,clip]{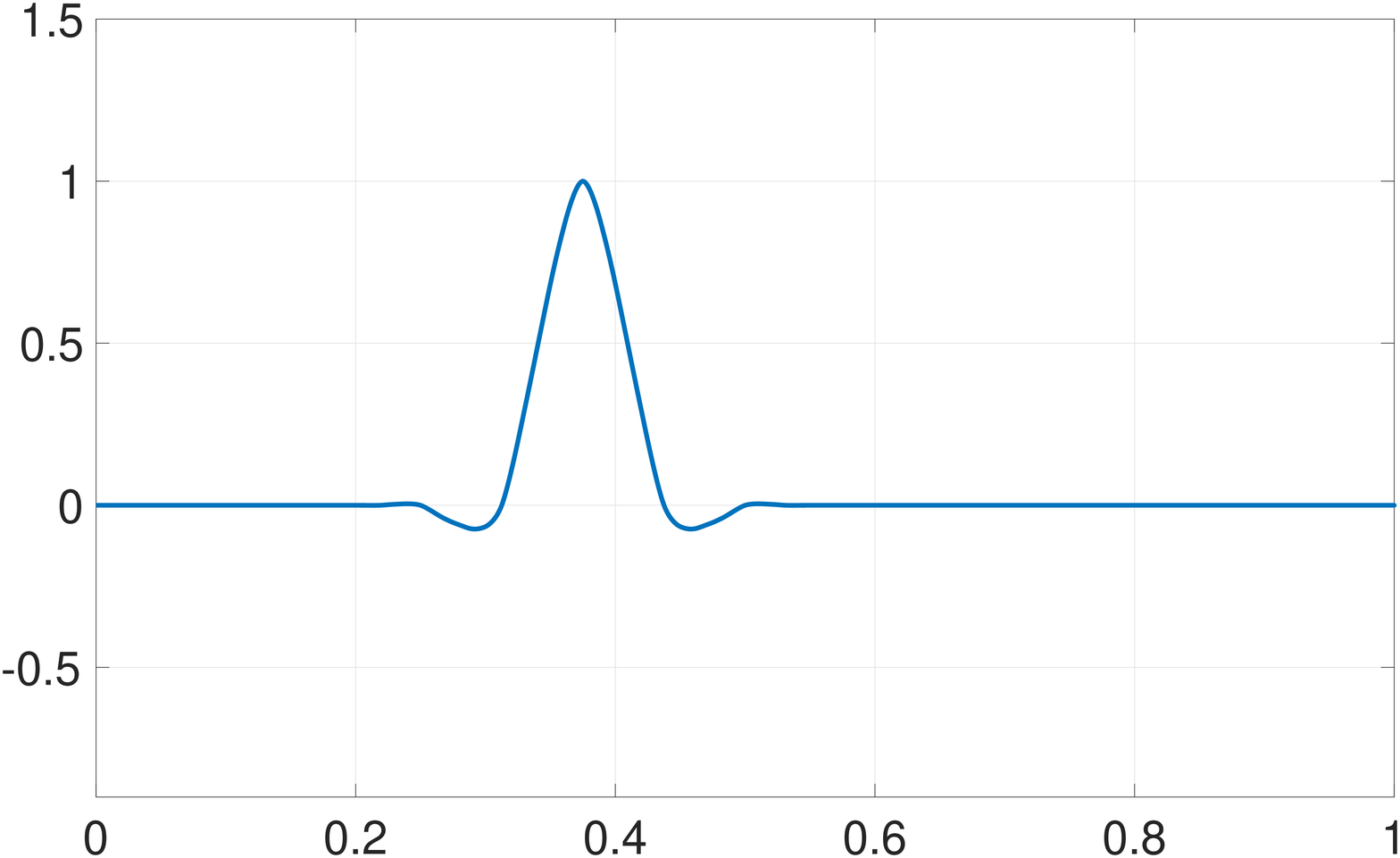}
	\end{tabular}
	\caption[Basis functions for $n=3$]{Some of the basis functions for the space of generated functions by the prediction schemes in \cref{sec:1D-int} for $n=3$. The starting level has 17 equal-spaced nodes in [0,1], and the prediction operators are applied 10 times, so the final level has grid points.} \label{fig:limit3}
\end{figure}

\begin{figure}[!h]
	\centering
	\begin{tabular}{cc} 
		\includegraphics[width=0.45\textwidth,clip]{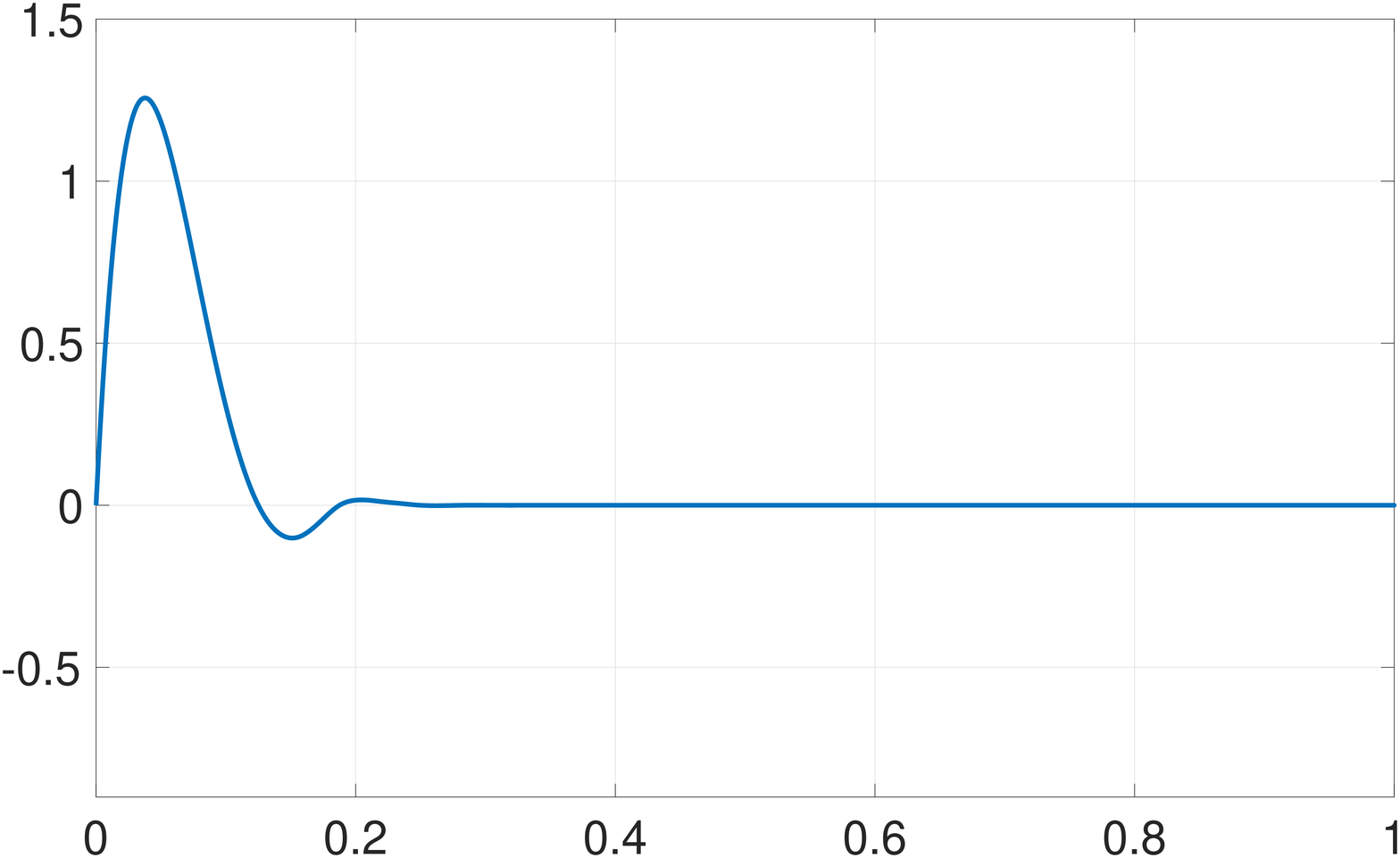} &	\includegraphics[width=0.45\textwidth,clip]{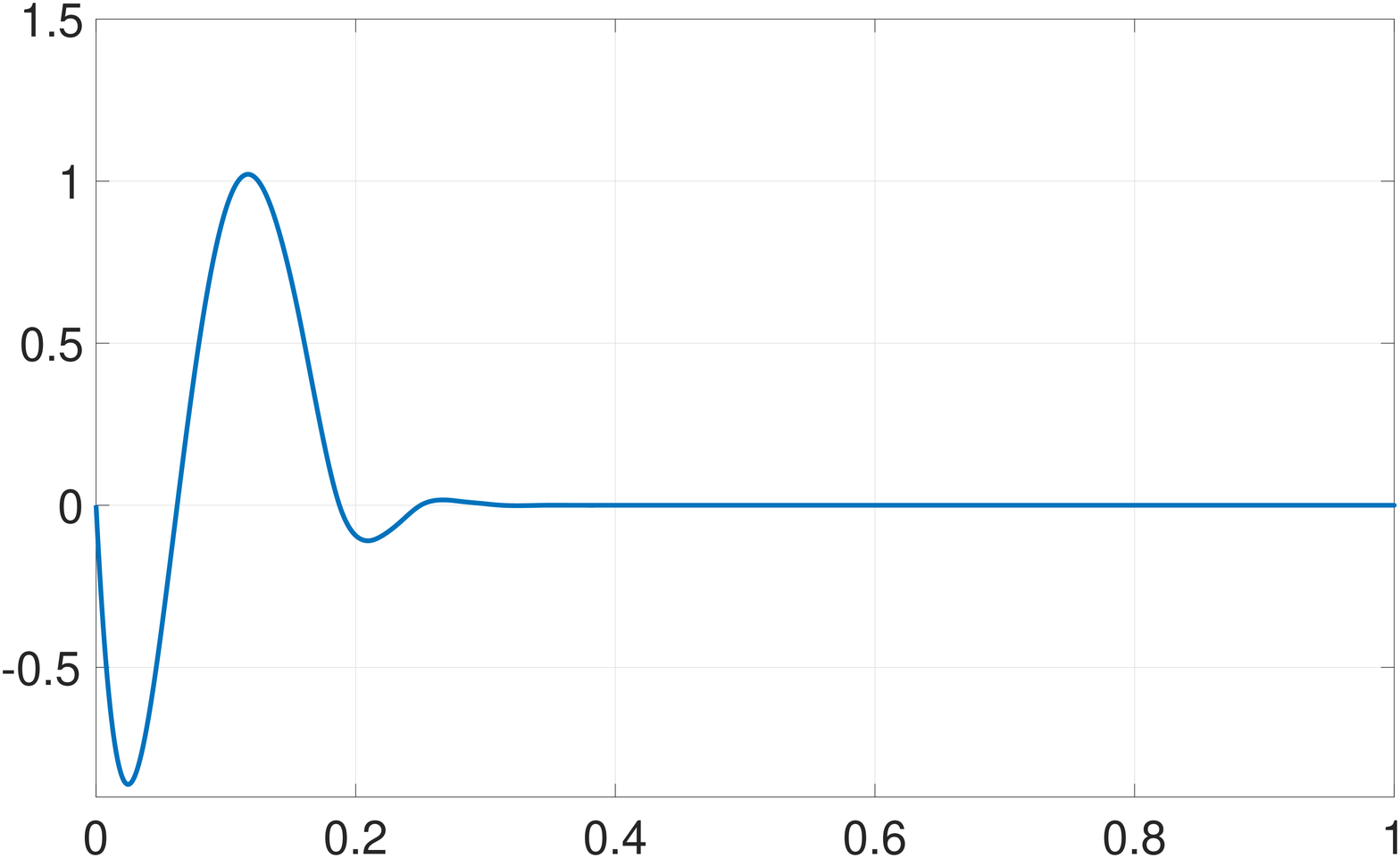} \\
		\includegraphics[width=0.45\textwidth,clip]{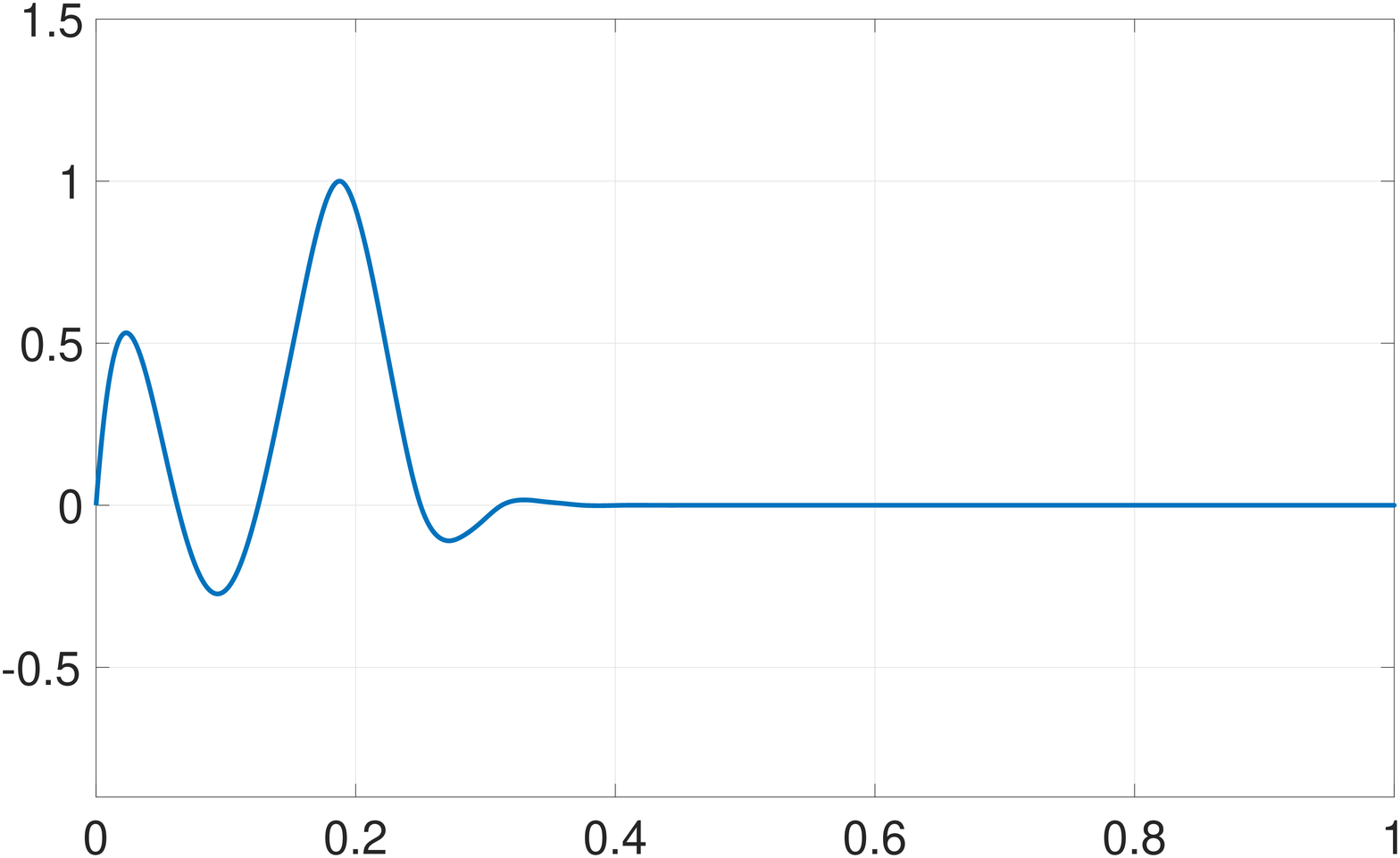} &	\includegraphics[width=0.45\textwidth,clip]{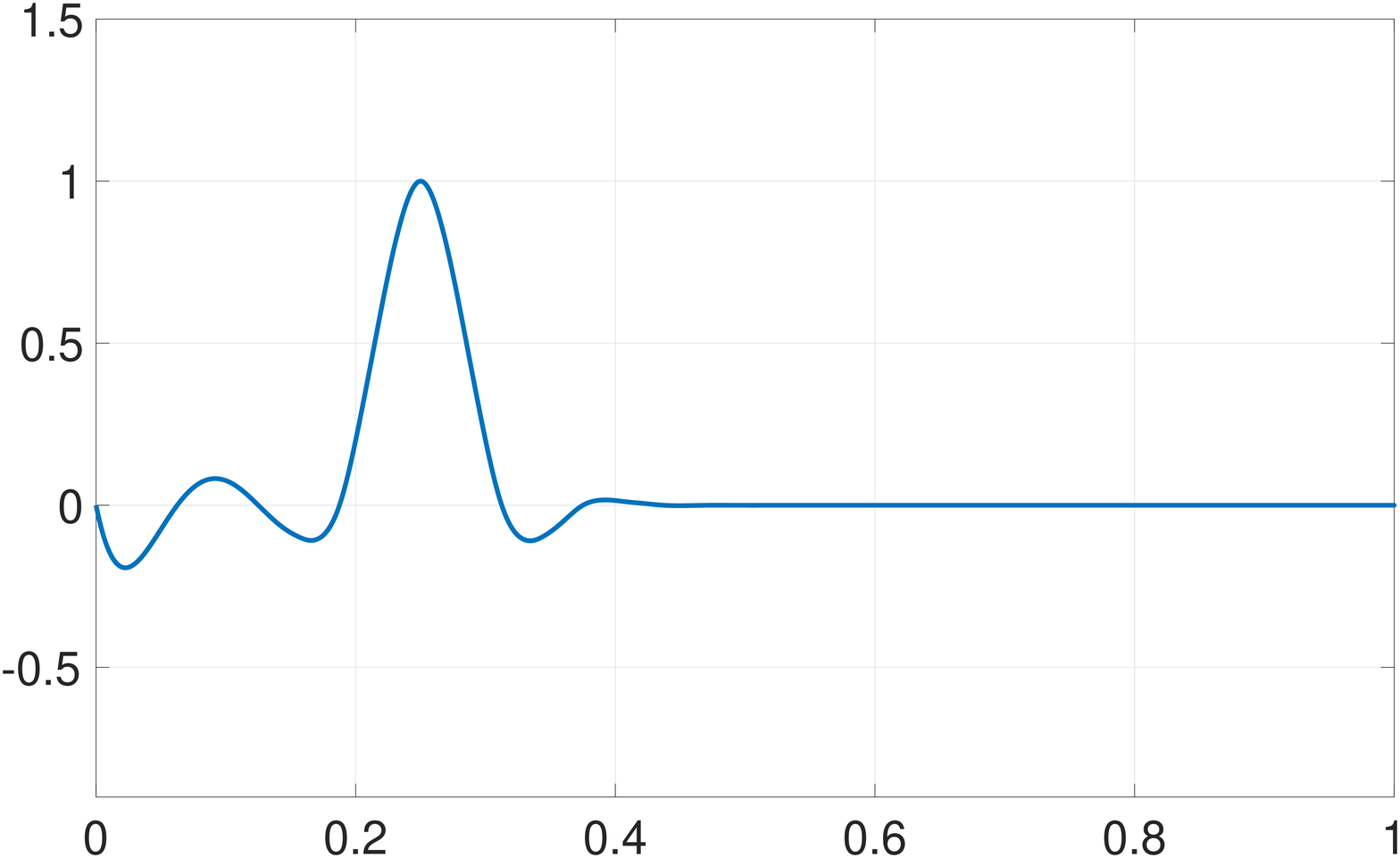} \\
		\includegraphics[width=0.45\textwidth,clip]{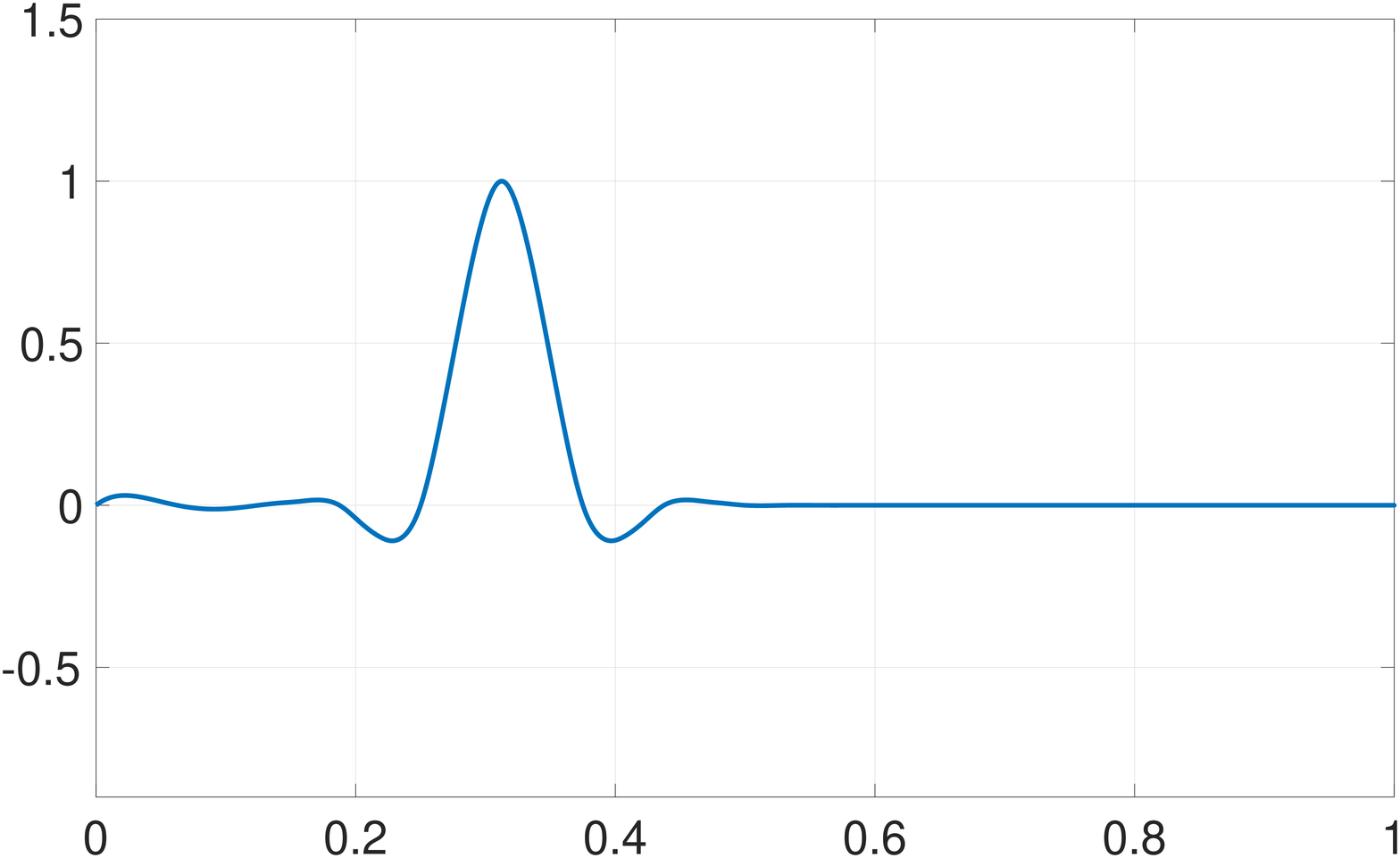} &	\includegraphics[width=0.45\textwidth,clip]{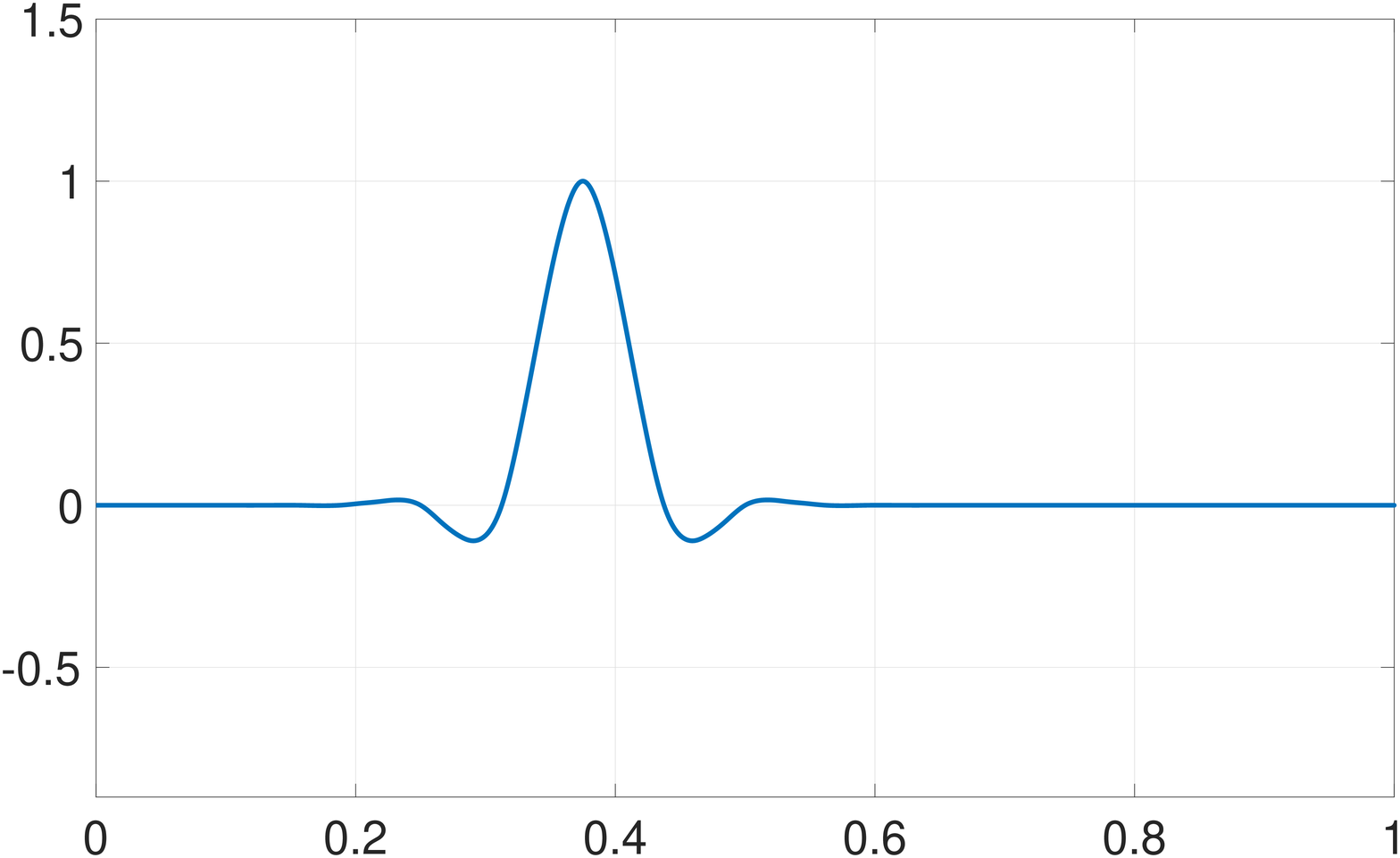}
	\end{tabular}
	\caption[Basis functions for $n=5$]{Some of the basis functions for the space of generated functions by the prediction schemes in \cref{sec:1D-int} for $n=5$. The starting level has 17 equal-spaced nodes in [0,1], and the prediction operators are applied 10 times, so the final level has 65537 grid points.} \label{fig:limit5}
\end{figure} 

\section{Smoothness of the MINS and MOREBV solutions}

In this section we discuss about the regularity of the solutions to the problems `MINS' \cref{eq:minimal} and `MOREBV' \cref{eq:morebv}.

According the boundary conditions, the derivatives $\frac{\partial^3}{\partial x^3}$ $\frac{\partial^3}{\partial y^3}$ of the solution $u$ to the continuous problem must be zero in the four corners $(0,0), (1,0), (0,1), (1,1)$. In the pictures in \cref{fig:C3} there is a sudden change in the value of the derivatives in the corners. On the one direction ($x$ or $y$, depending on the picture), the derivative is zero, while in the other direction, the absolute value grows when approaching to the corner. If there exists some continuous function which point evaluations are the discrete solution $z_{\min}$, then its third derivatives seem to be not continuous in the four corners. Hence, the solutions are not $\cC^3$ globally. For comparison, we show the graphics of $\frac{\partial^2}{\partial x^2}$ $\frac{\partial^2}{\partial y^2}$ in Figure \cref{fig:C2}, which seem to be continuous, but not differentiable.

\begin{figure}[!h]
	\centering
	\begin{tabular}{cc} 
		\includegraphics[width=0.45\textwidth,clip]{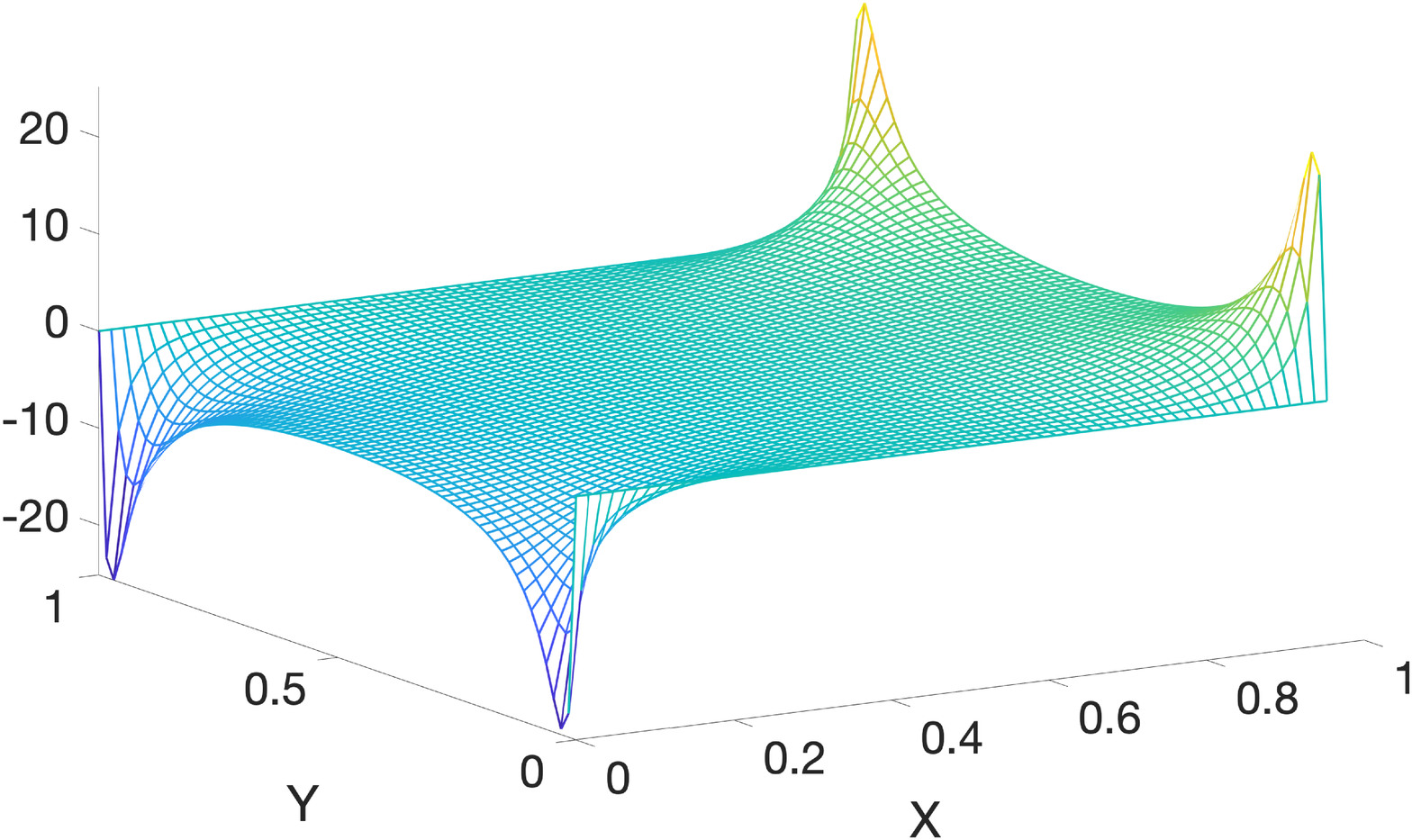} &
		\includegraphics[width=0.45\textwidth,clip]{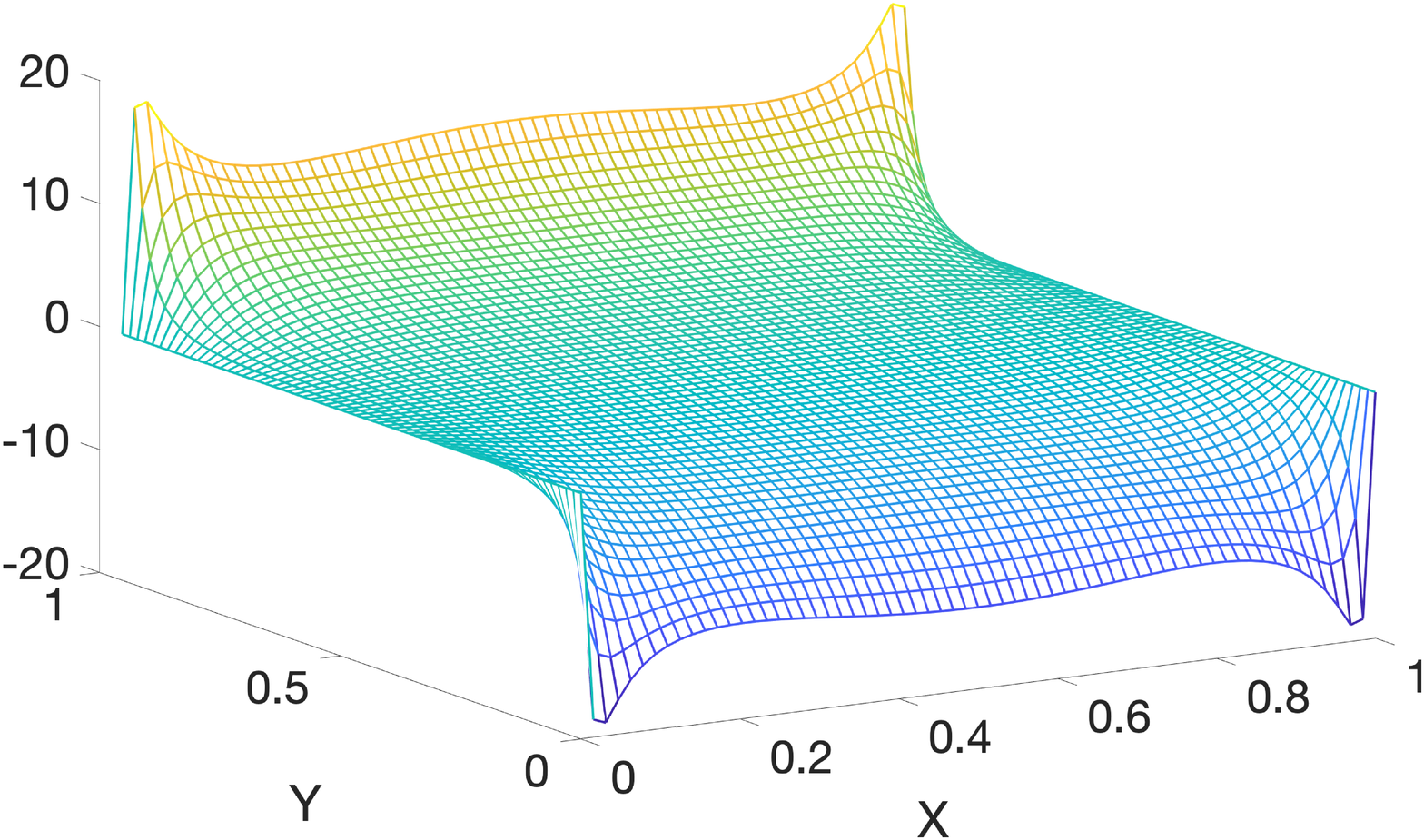} \\
		\includegraphics[width=0.45\textwidth,clip]{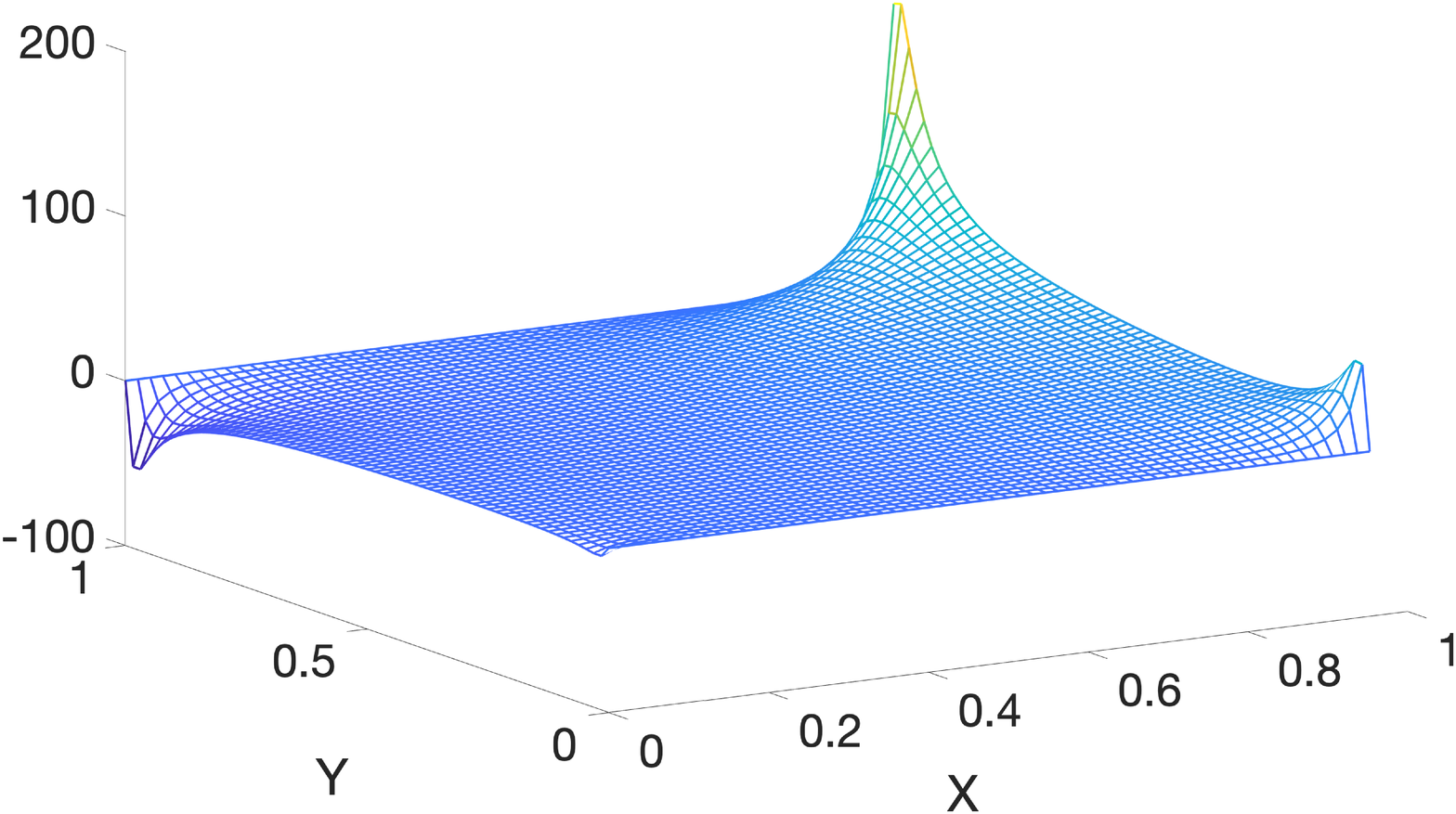} &
		\includegraphics[width=0.45\textwidth,clip]{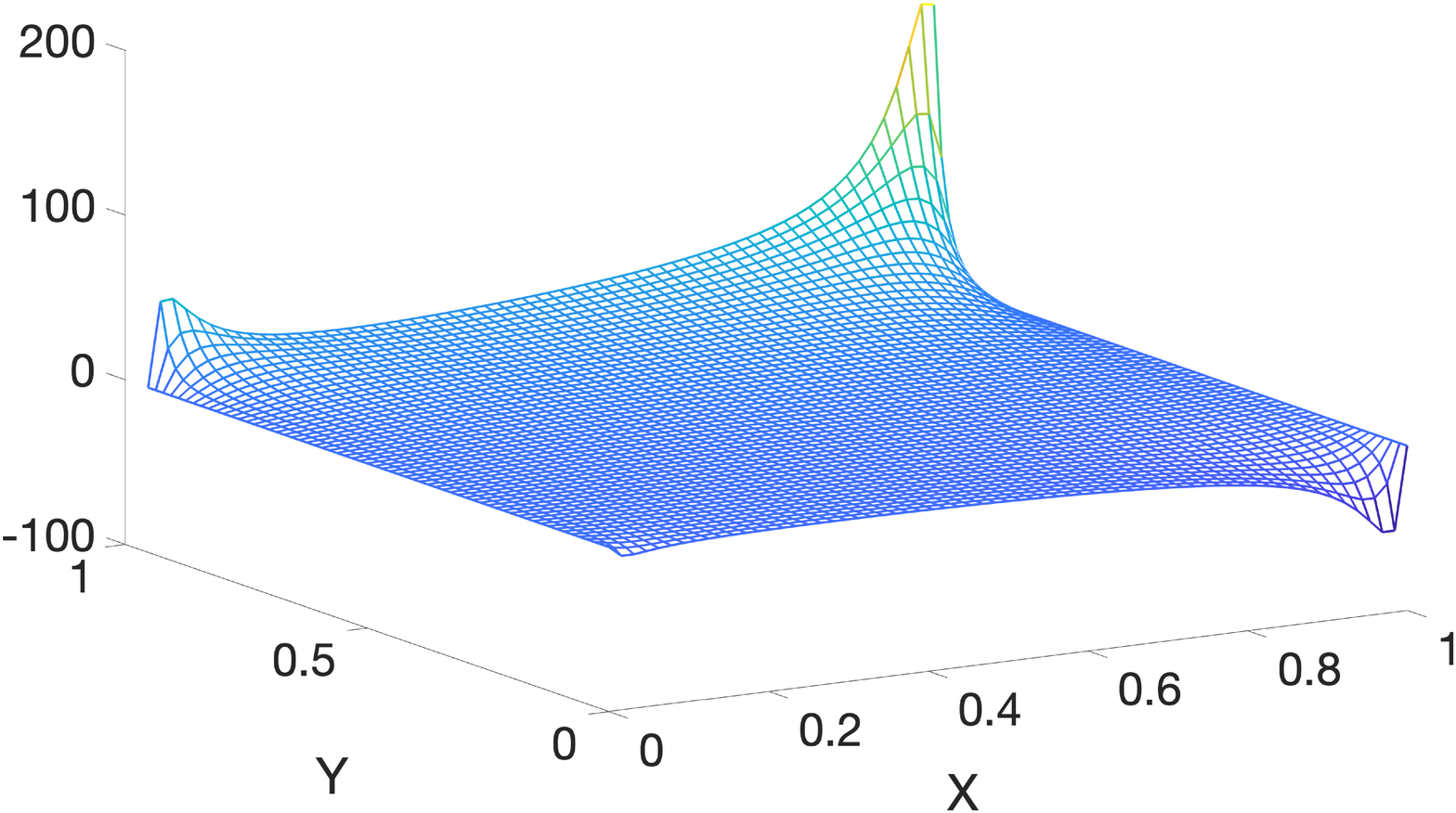}
	\end{tabular}
	\caption[Third order differences]{The derivatives $\frac{\partial^3}{\partial x^3}$ $\frac{\partial^3}{\partial y^3}$ (left and right columns, respectively) in the interval $[0,1]\times[0,1]$ of the solutions to the problems `MINS' \cref{eq:minimal} and `MOREBV' \cref{eq:morebv} (top and bottom rows).
		The solutions were computed on a $65\times 65$ uniform grid up to machine precision, solving the minimization problem with the direct use of \texttt{fminunc}.
	} \label{fig:C3}
\end{figure}

\begin{figure}[!h]
	\centering
	\begin{tabular}{cc} 
		\includegraphics[width=0.45\textwidth,clip]{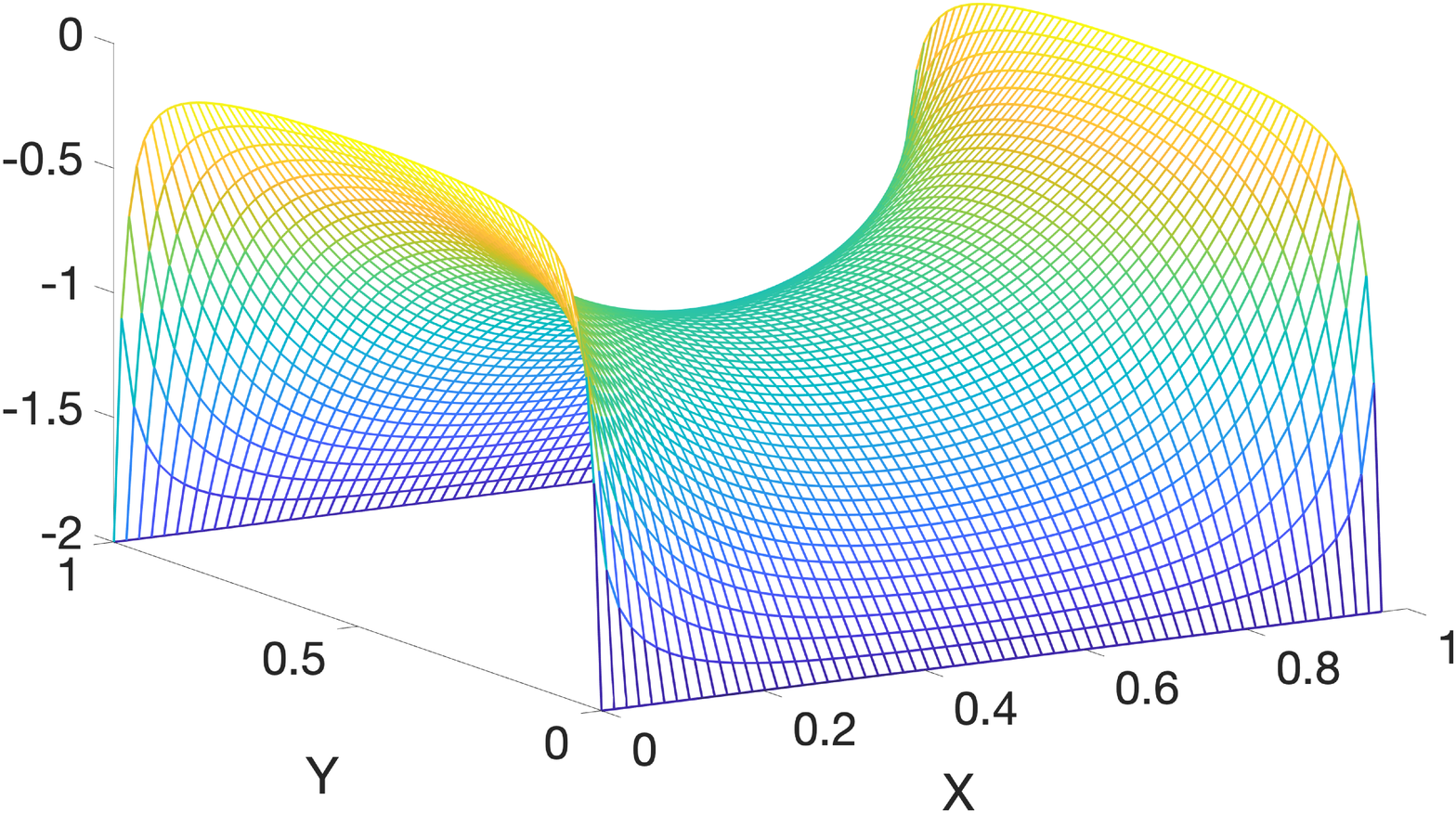} &
		\includegraphics[width=0.45\textwidth,clip]{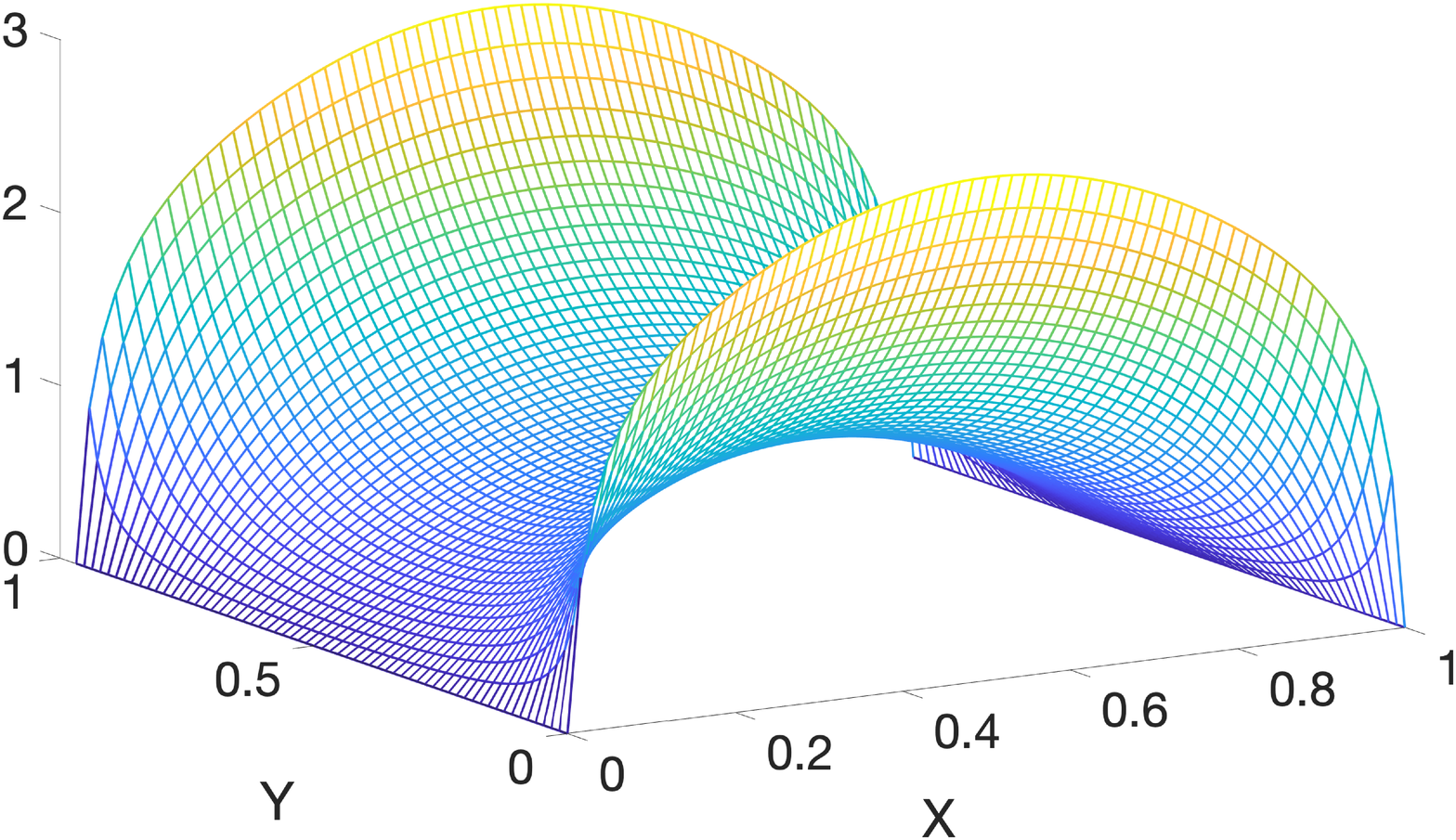} \\
		\includegraphics[width=0.45\textwidth,clip]{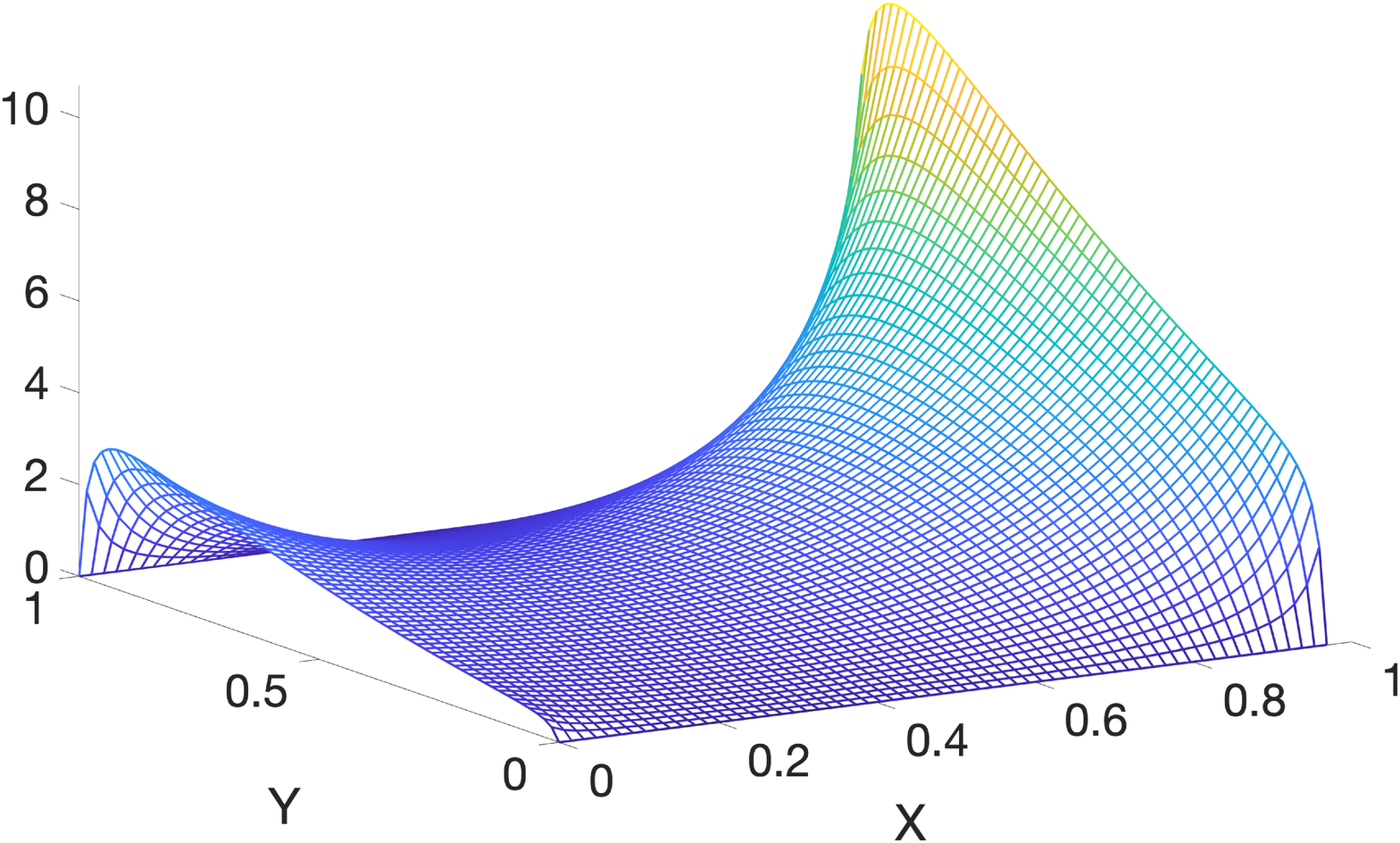} &
		\includegraphics[width=0.45\textwidth,clip]{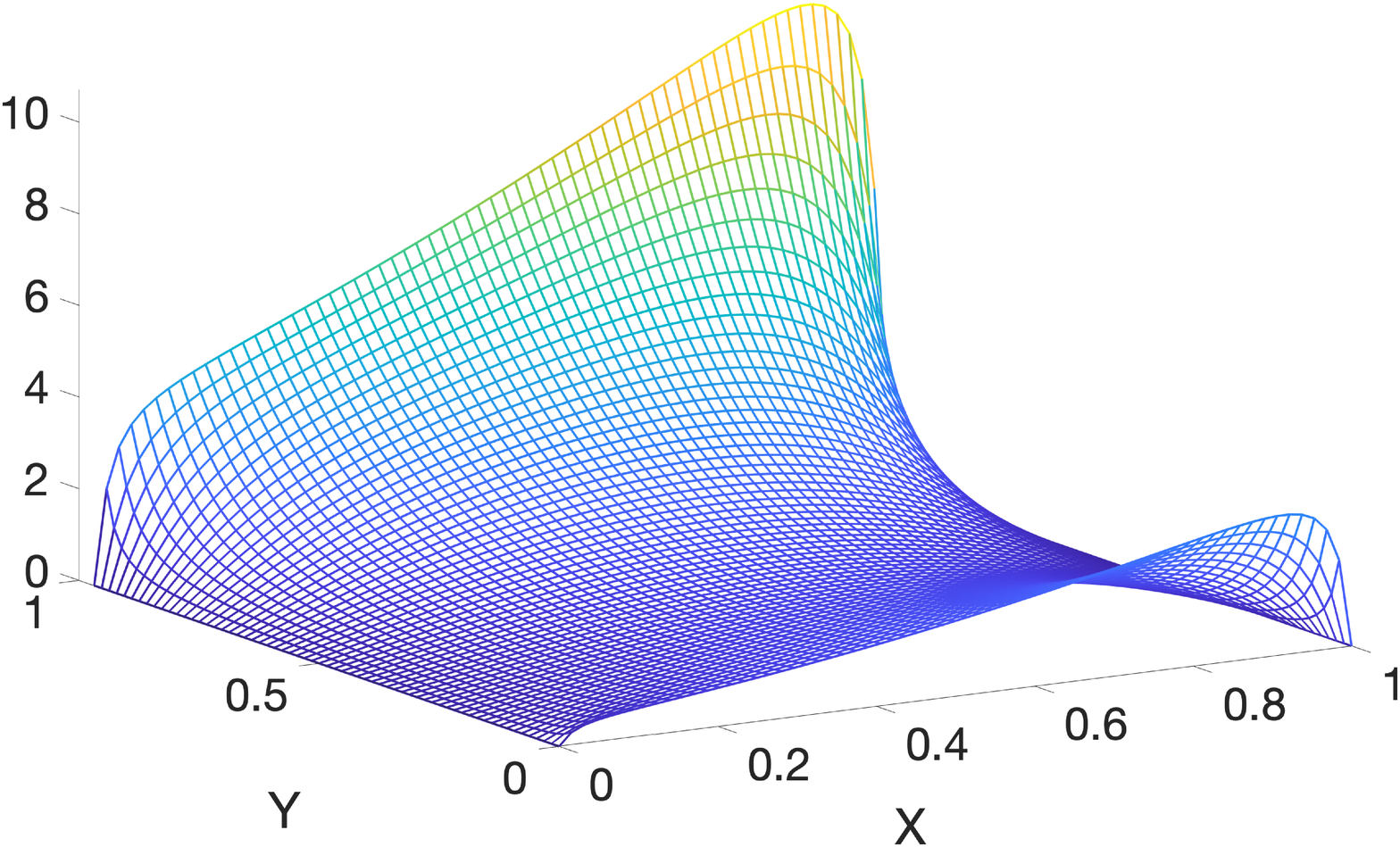}
	\end{tabular}
	\caption[Second order differences]{The derivatives $\frac{\partial^2}{\partial x^2}$ $\frac{\partial^2}{\partial y^2}$ (left and right columns, respectively) in the interval $[0,1]\times[0,1]$ of the solutions to the problems `MINS' \cref{eq:minimal} and `MOREBV' \cref{eq:morebv} (top and bottom rows).
		The solutions were computed on a $65\times 65$ uniform grid up to machine precision, solving the minimization problem with the direct use of \texttt{fminunc}.
	} \label{fig:C2}
\end{figure}

\section{Sub-optimal solutions of the MOREBV problem}

In this section we show the sub-optimal solutions found by MR/OPT when solving the MOREBV problem as in \cref{sec:nonquadratic}.

Observe in \cref{fig:suboptimalsn1} that, for the linear prediction operators in \cref{sec:1D-int} with $n=1$, how slowly the sub-optimal solutions approaches the solution (pay attention to the z-axis). The low regularity of the data generated by this prediction operators may explain this behavior, since the continuous optimization problem requires the computation of a Laplacian.

However, we can see in \cref{fig:suboptimalsn3,fig:suboptimalsn5} that for $n=3,5$ the sub-optimal solution converges quickly to $z_{\min}$.

\begin{figure}[!h]
	\centering
	\begin{tabular}{cc} 
		$z^{L,1}$ & $z^{L,2}$\\
		\includegraphics[width=0.45\textwidth,clip]{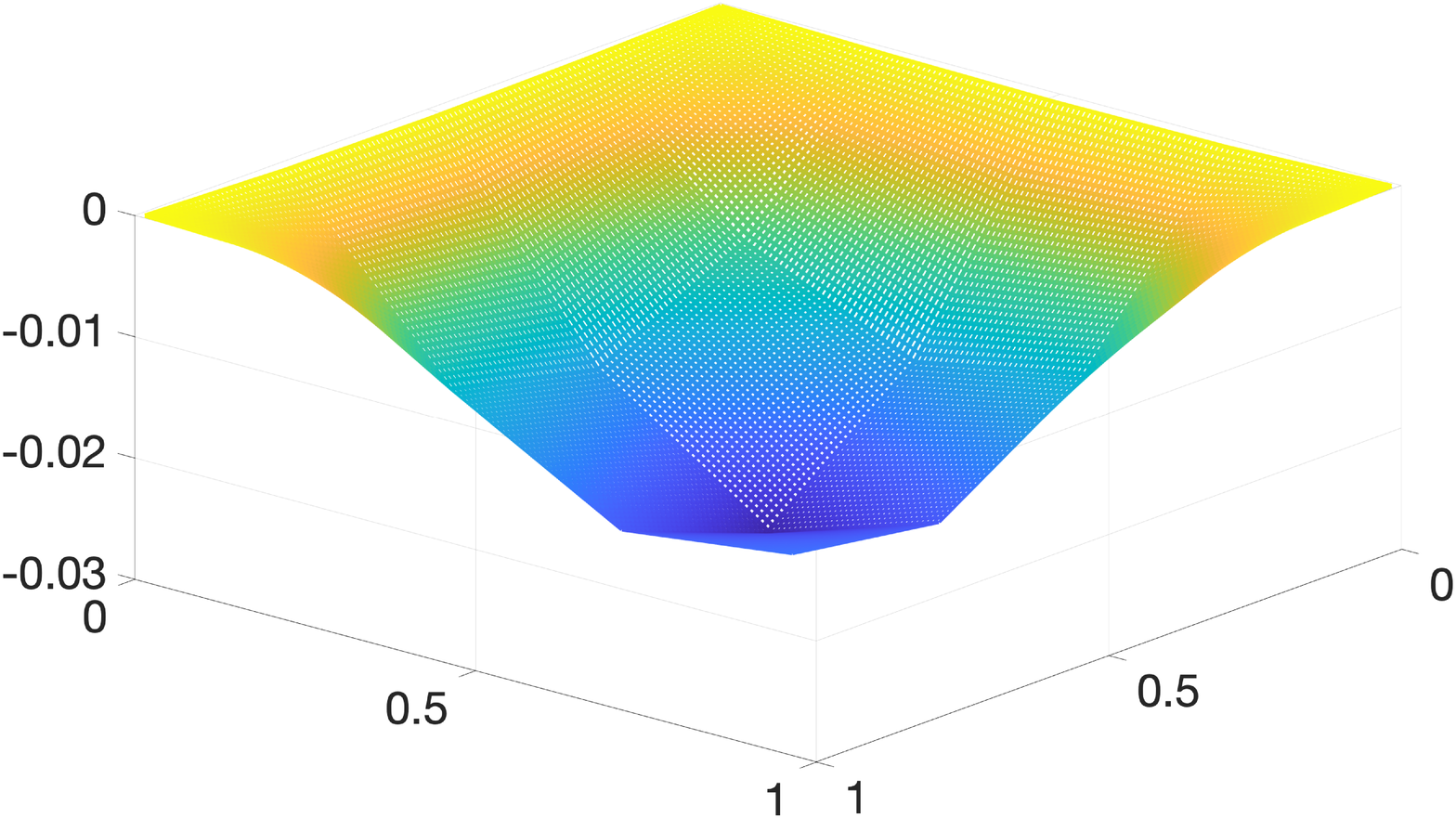} &
		\includegraphics[width=0.45\textwidth,clip]{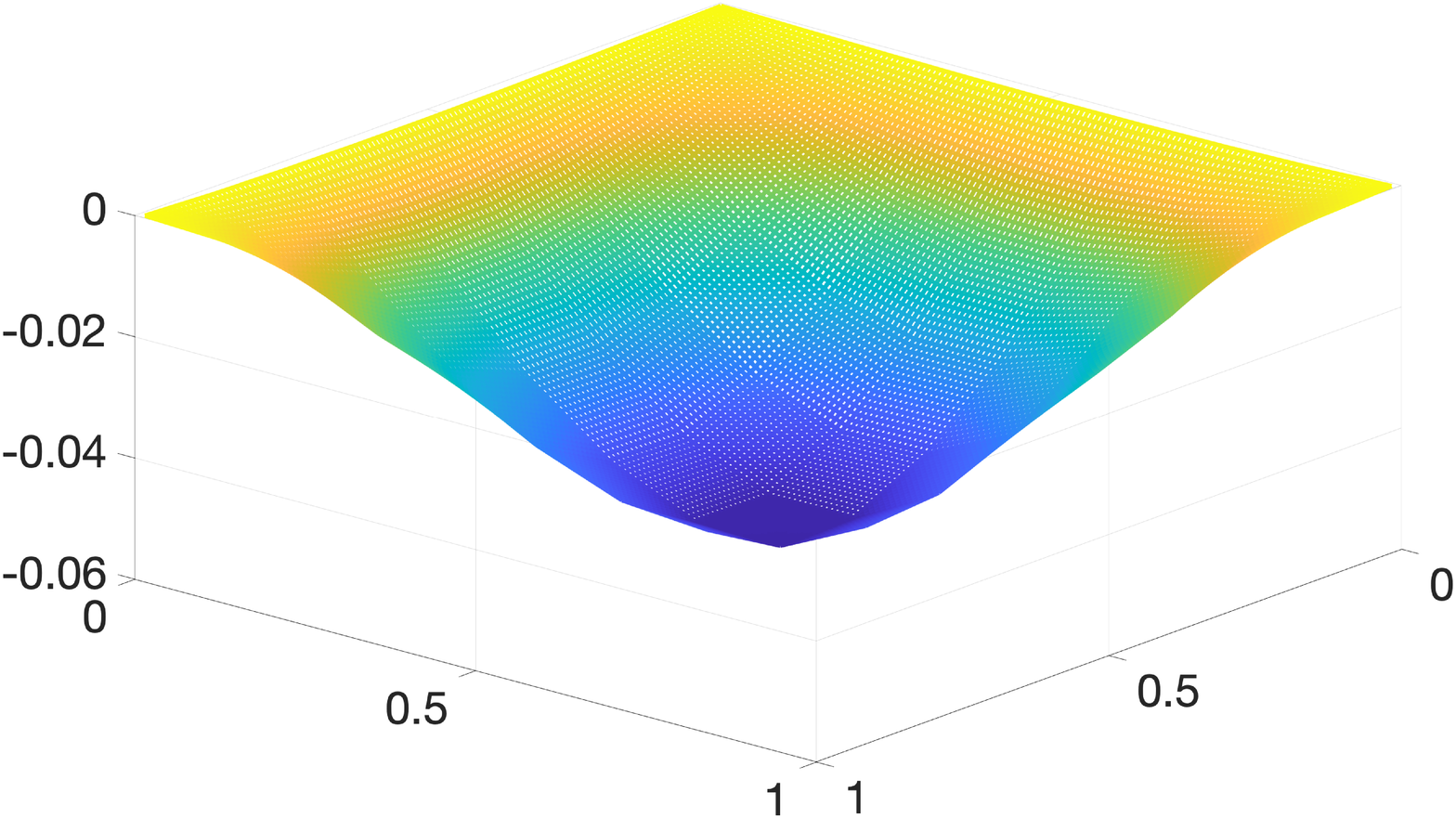} \\[10pt]
		$z^{L,3}$ & $z^{L,4}$\\
		\includegraphics[width=0.45\textwidth,clip]{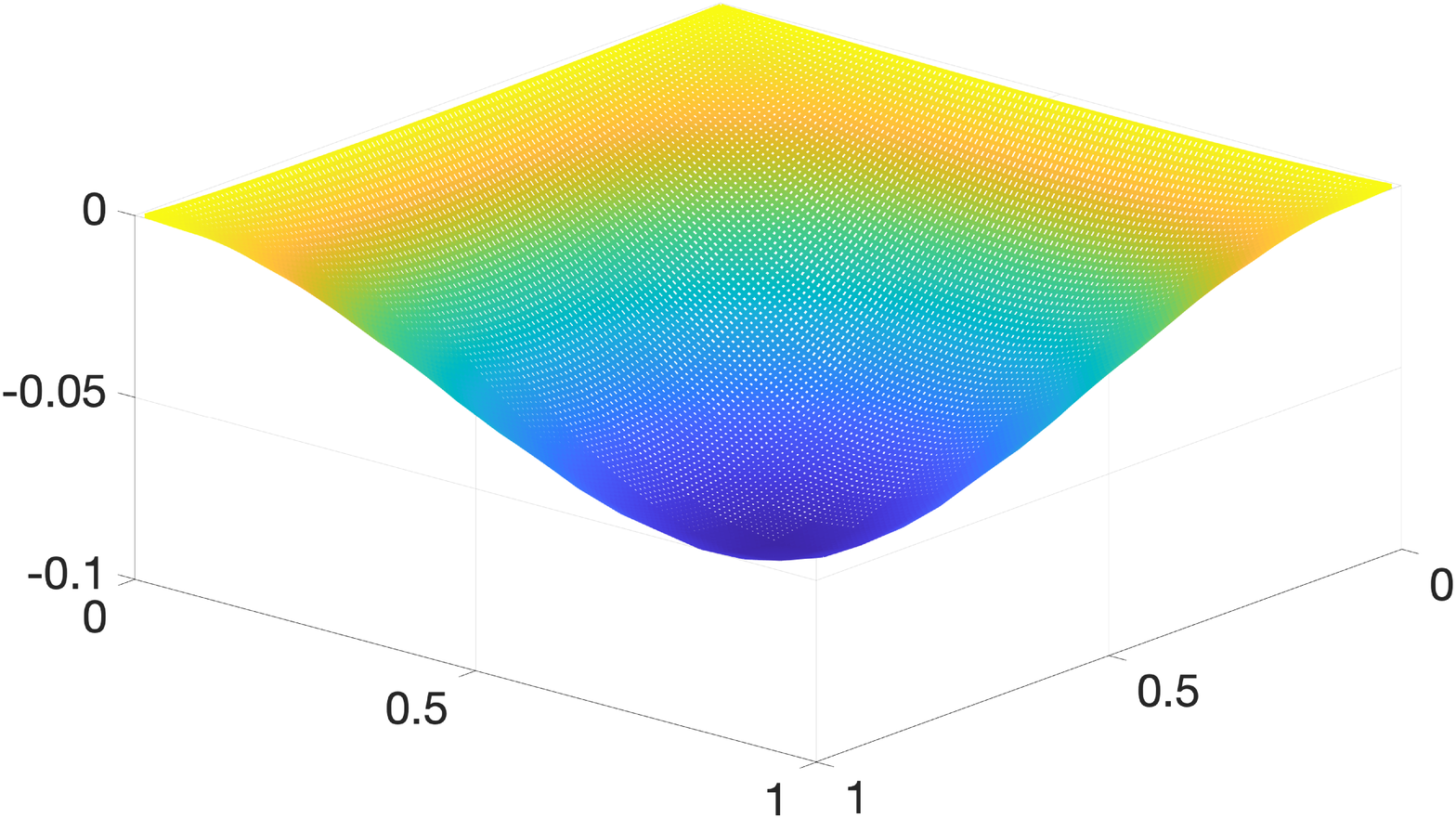} &
		\includegraphics[width=0.45\textwidth,clip]{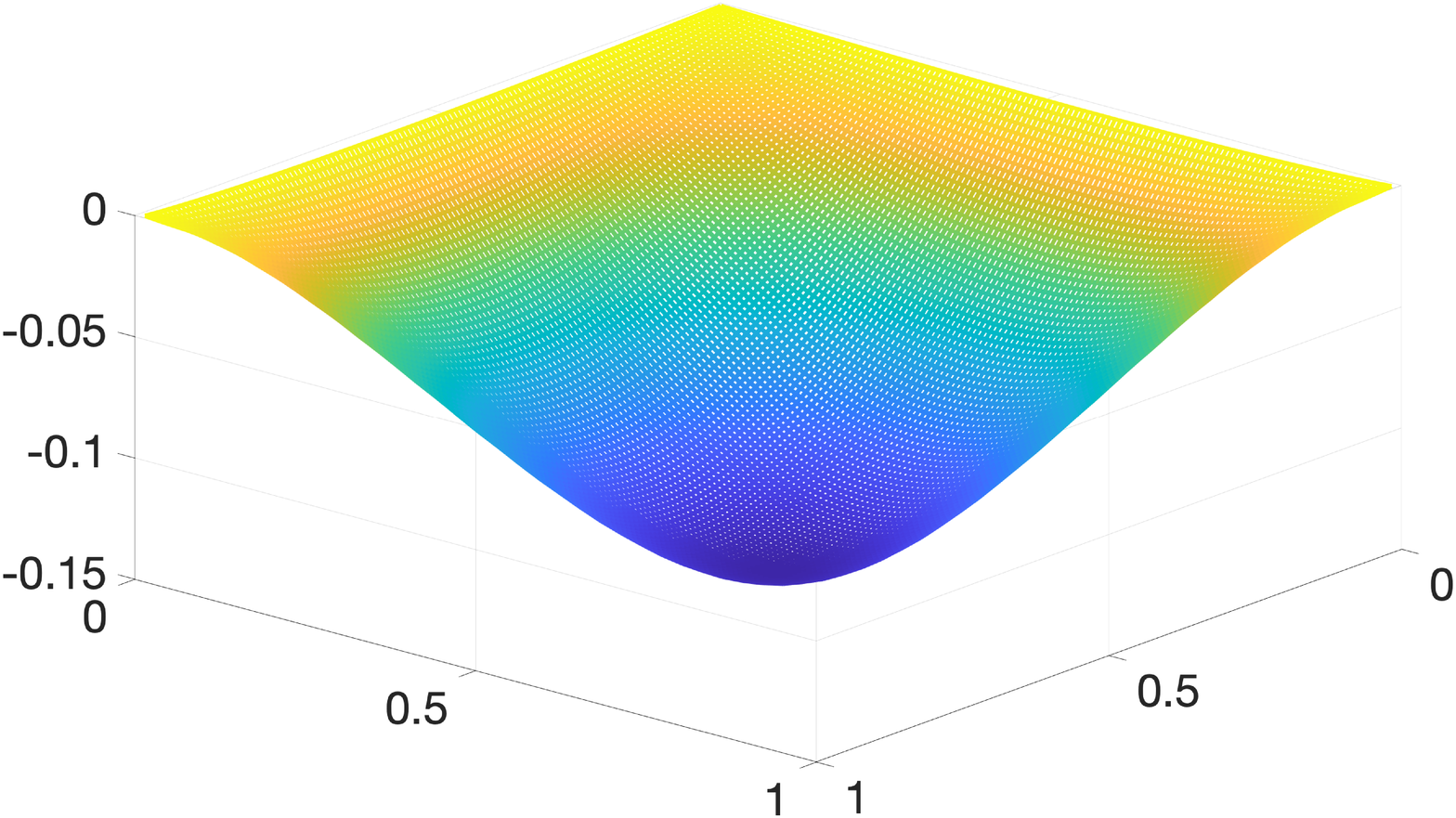} \\[10pt]
		$z^{L,5}$ & $z^{L,6}$\\
		\includegraphics[width=0.45\textwidth,clip]{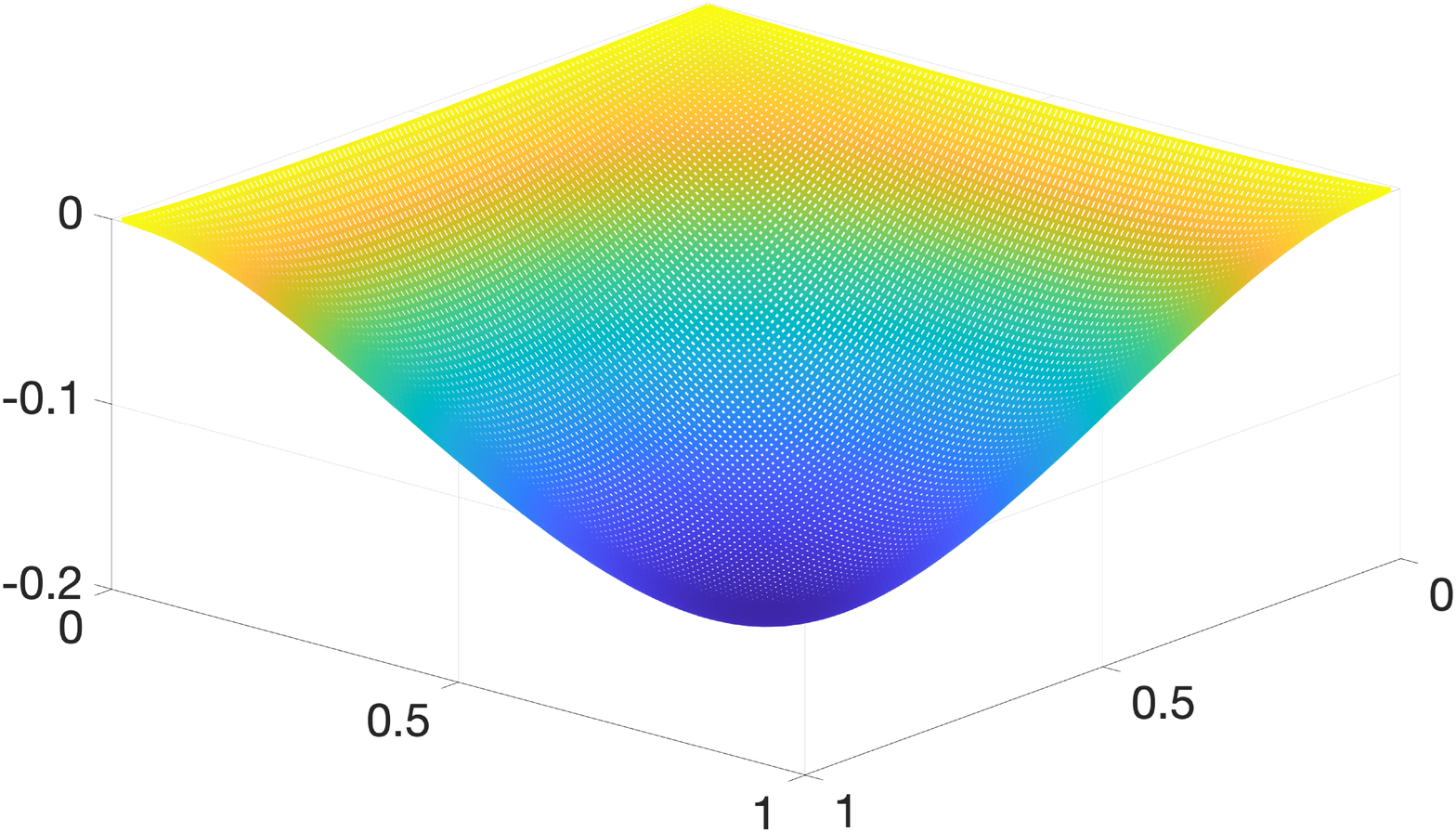} &
		\includegraphics[width=0.45\textwidth,clip]{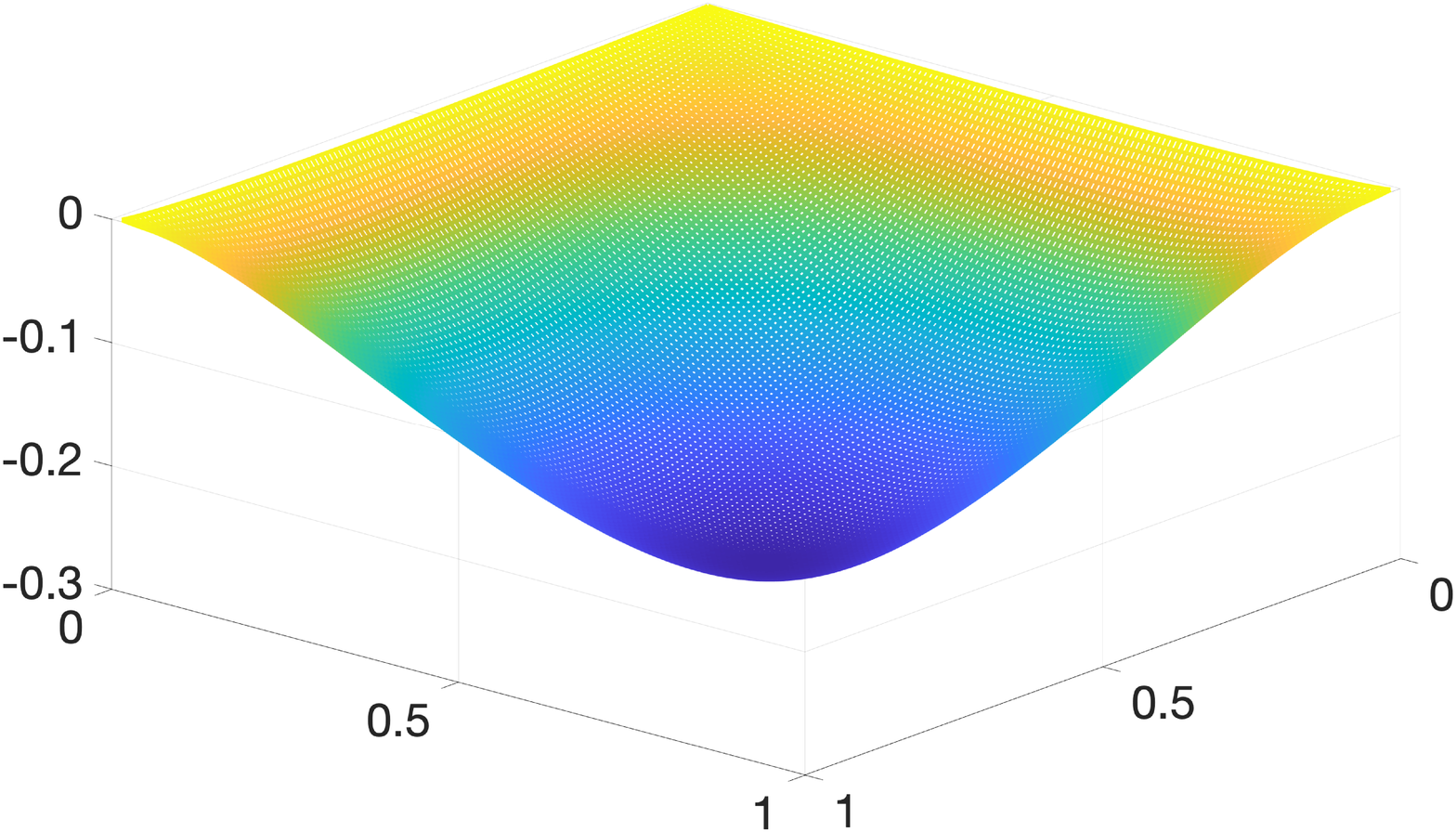}
	\end{tabular}
	\caption[Sub-optimals for MOREBV and $n=1$]{The suboptimal solutions to the problem `MOREBV' \cref{eq:morebv} for $n=1$.
	} \label{fig:suboptimalsn1}
\end{figure}

\begin{figure}[!h]
	\centering
	\begin{tabular}{cc} 
		$z^{L,1}$ & $z^{L,2}$\\
		\includegraphics[width=0.45\textwidth,clip]{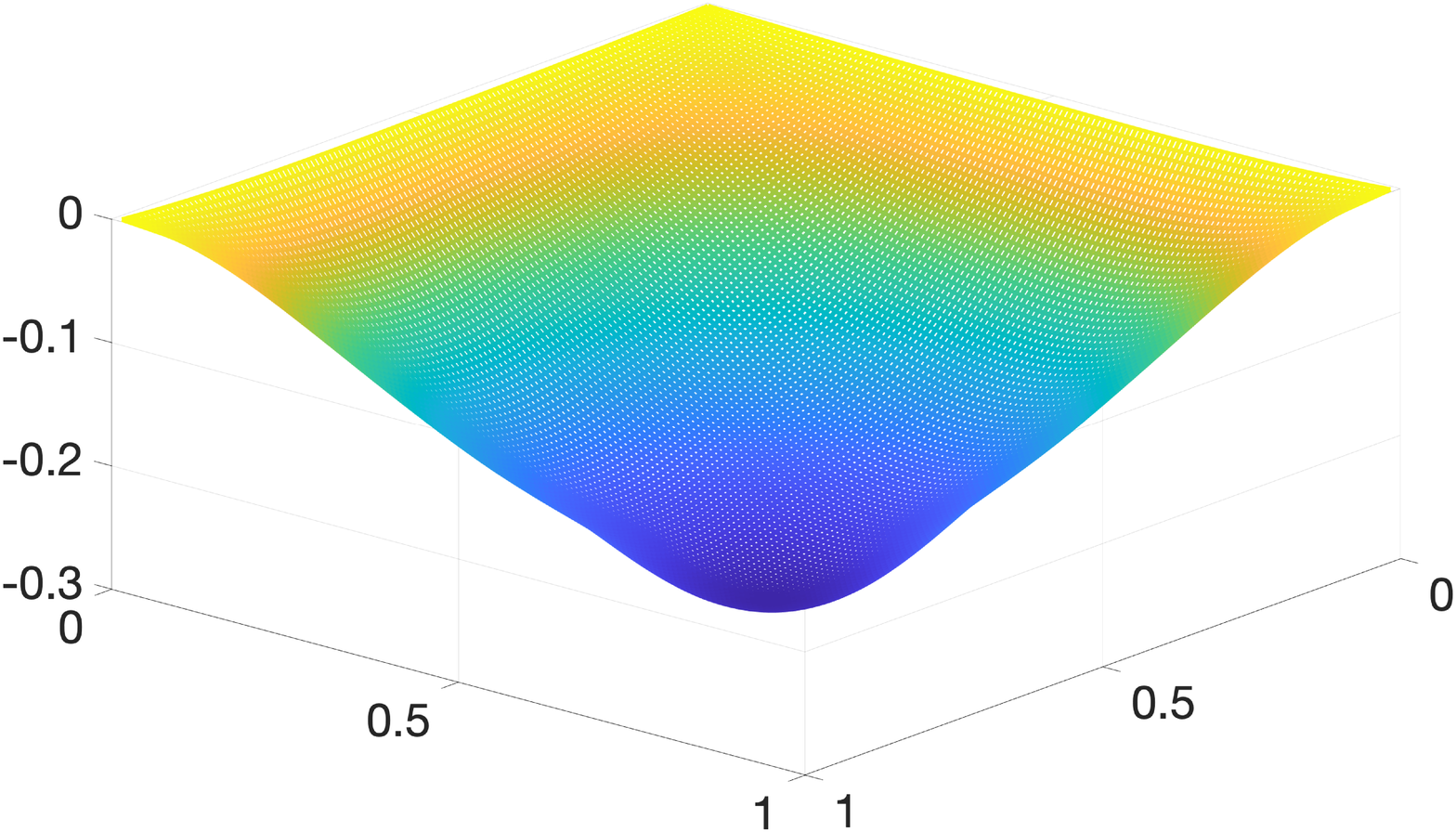} &
		\includegraphics[width=0.45\textwidth,clip]{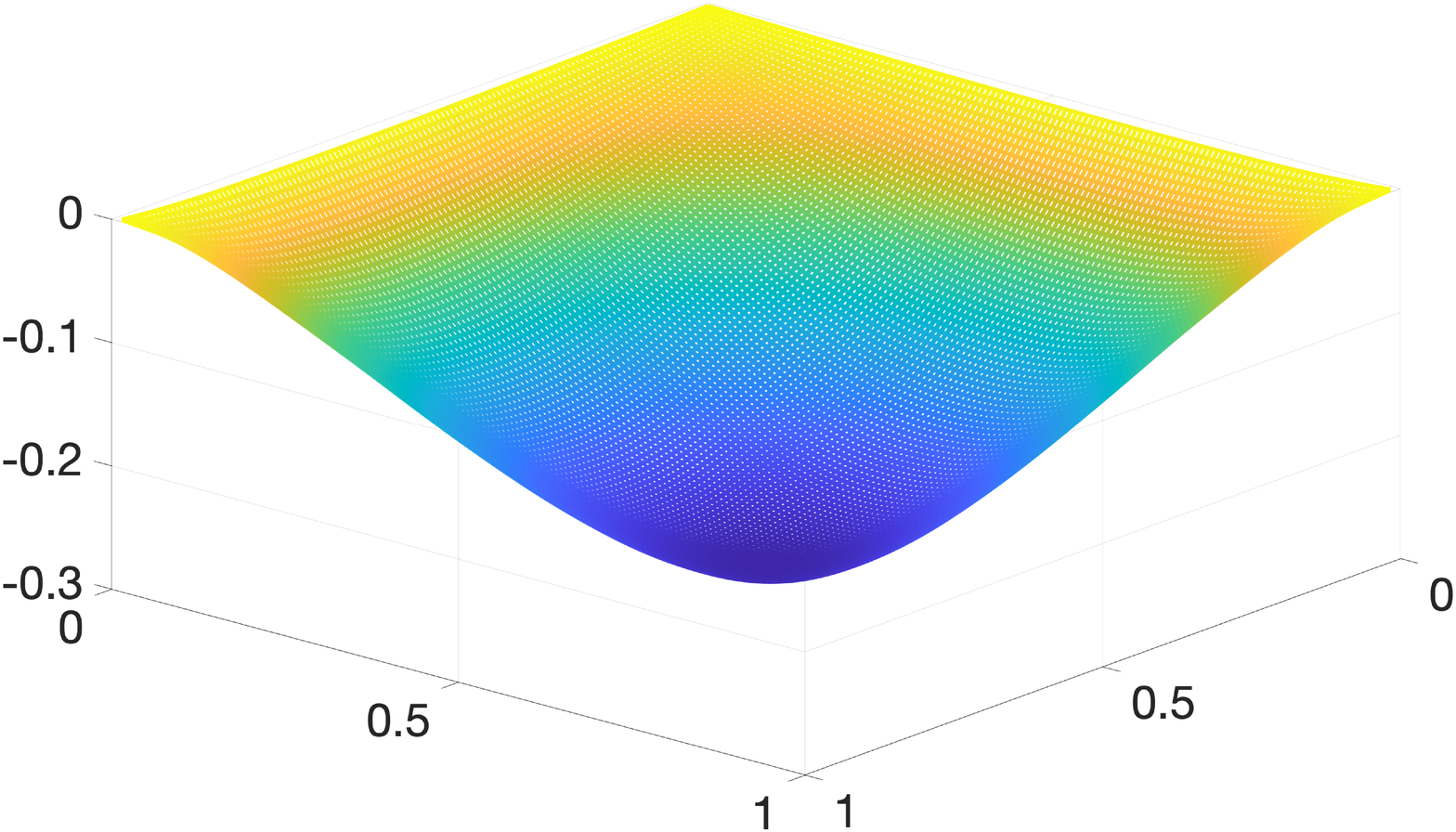} \\[10pt]
		$z^{L,3}$ & $z^{L,4}$\\
		\includegraphics[width=0.45\textwidth,clip]{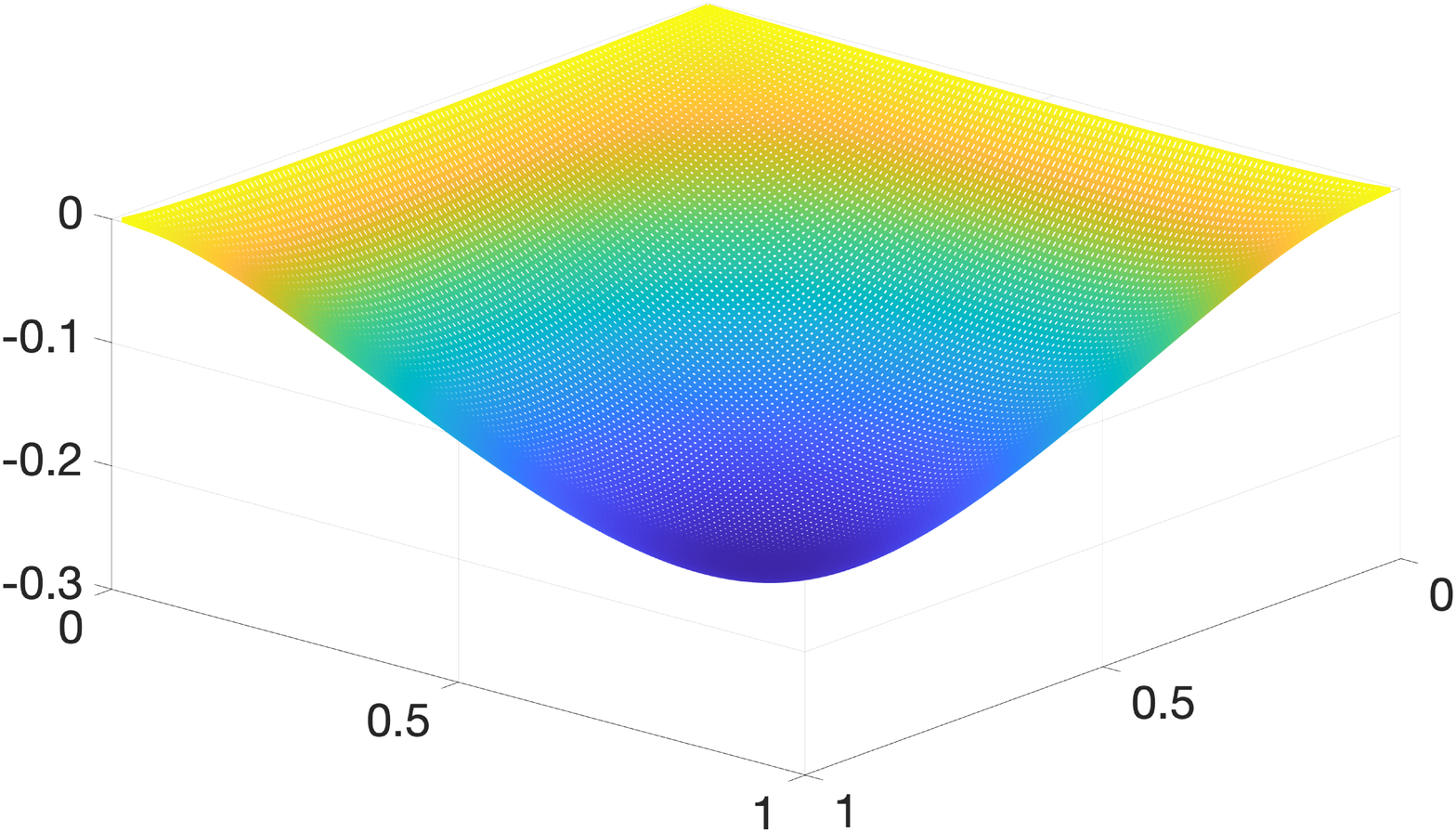} &
		\includegraphics[width=0.45\textwidth,clip]{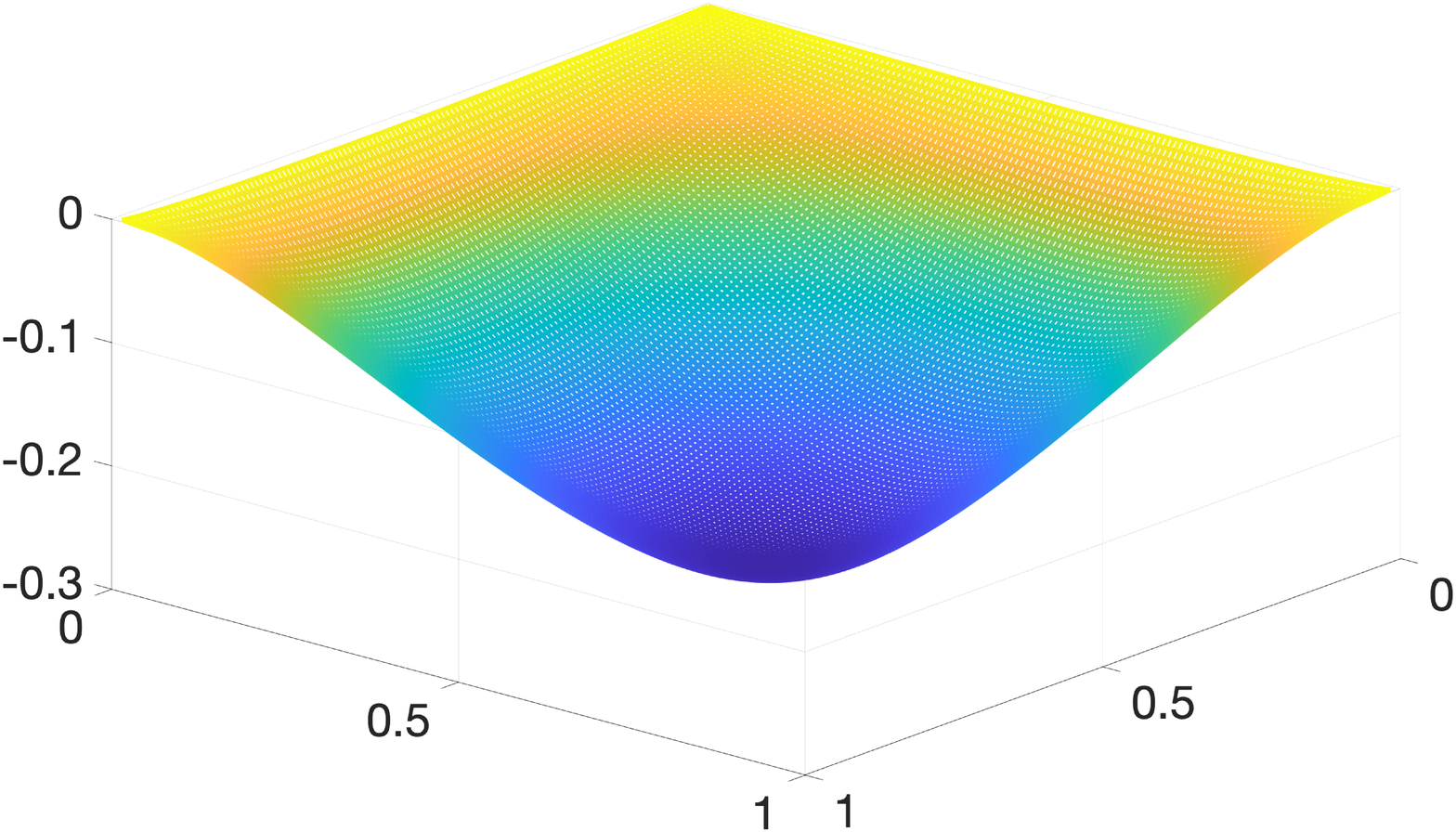}
	\end{tabular}
	\caption[Sub-optimals for MOREBV and $n=3$]{Some suboptimal solutions to the problem `MOREBV' \cref{eq:morebv} for $n=3$.
	} \label{fig:suboptimalsn3}
\end{figure}

\begin{figure}[!h]
	\centering
	\begin{tabular}{cc} 
		$z^{L,1}$ & $z^{L,2}$\\
		\includegraphics[width=0.45\textwidth,clip]{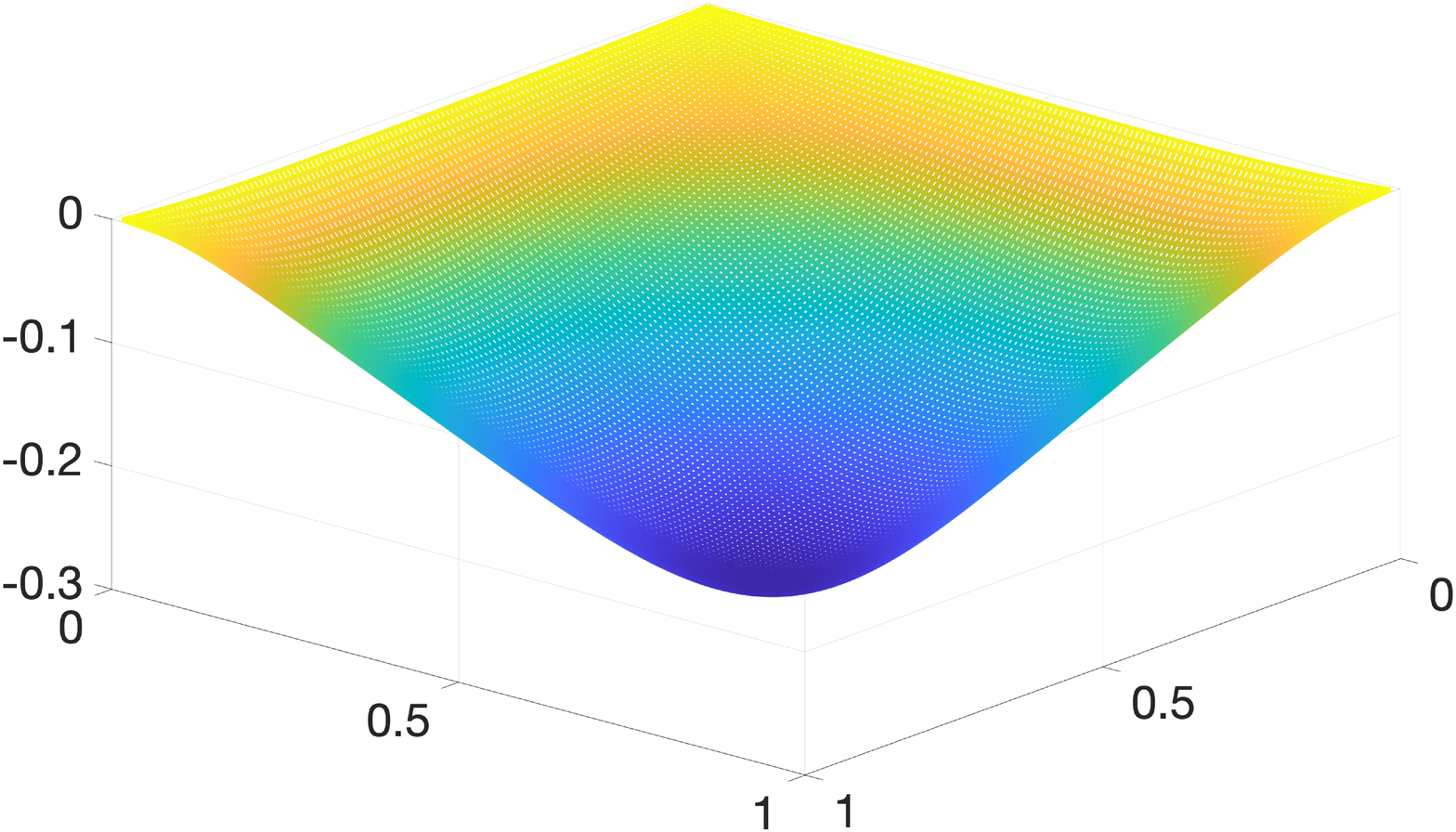} &
		\includegraphics[width=0.45\textwidth,clip]{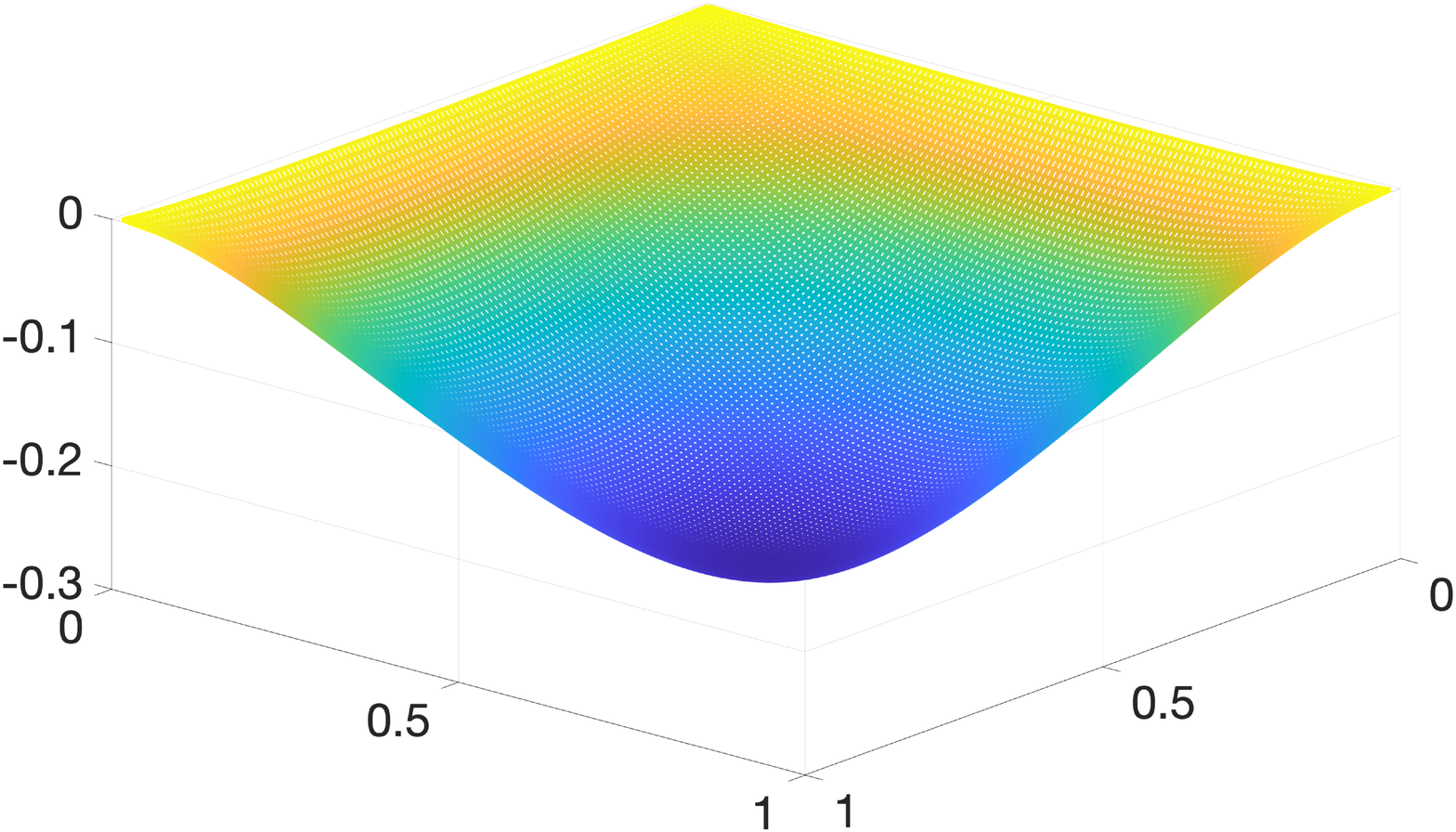} \\[10pt]
		$z^{L,3}$ & $z^{L,4}$\\
		\includegraphics[width=0.45\textwidth,clip]{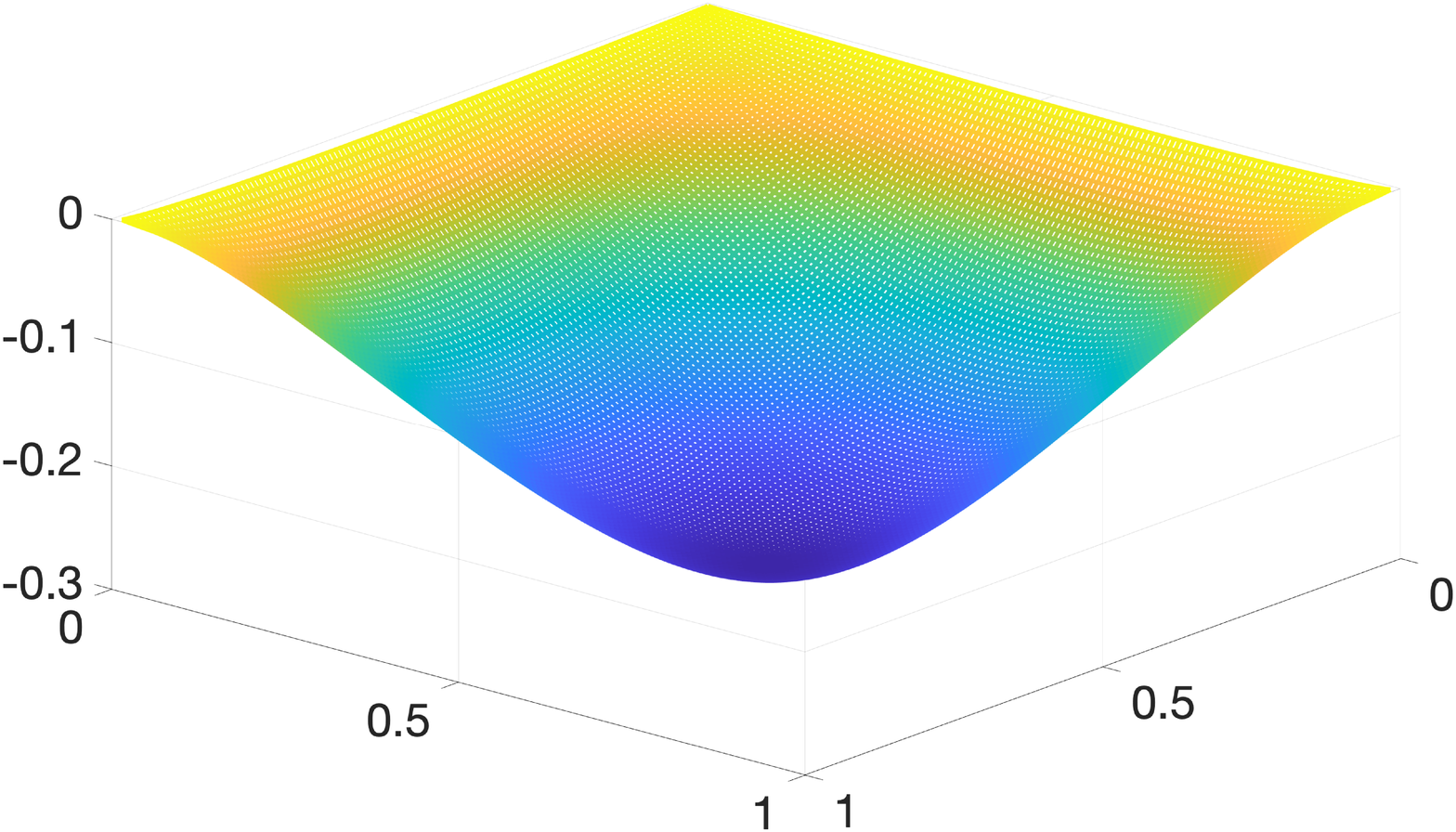} &
		\includegraphics[width=0.45\textwidth,clip]{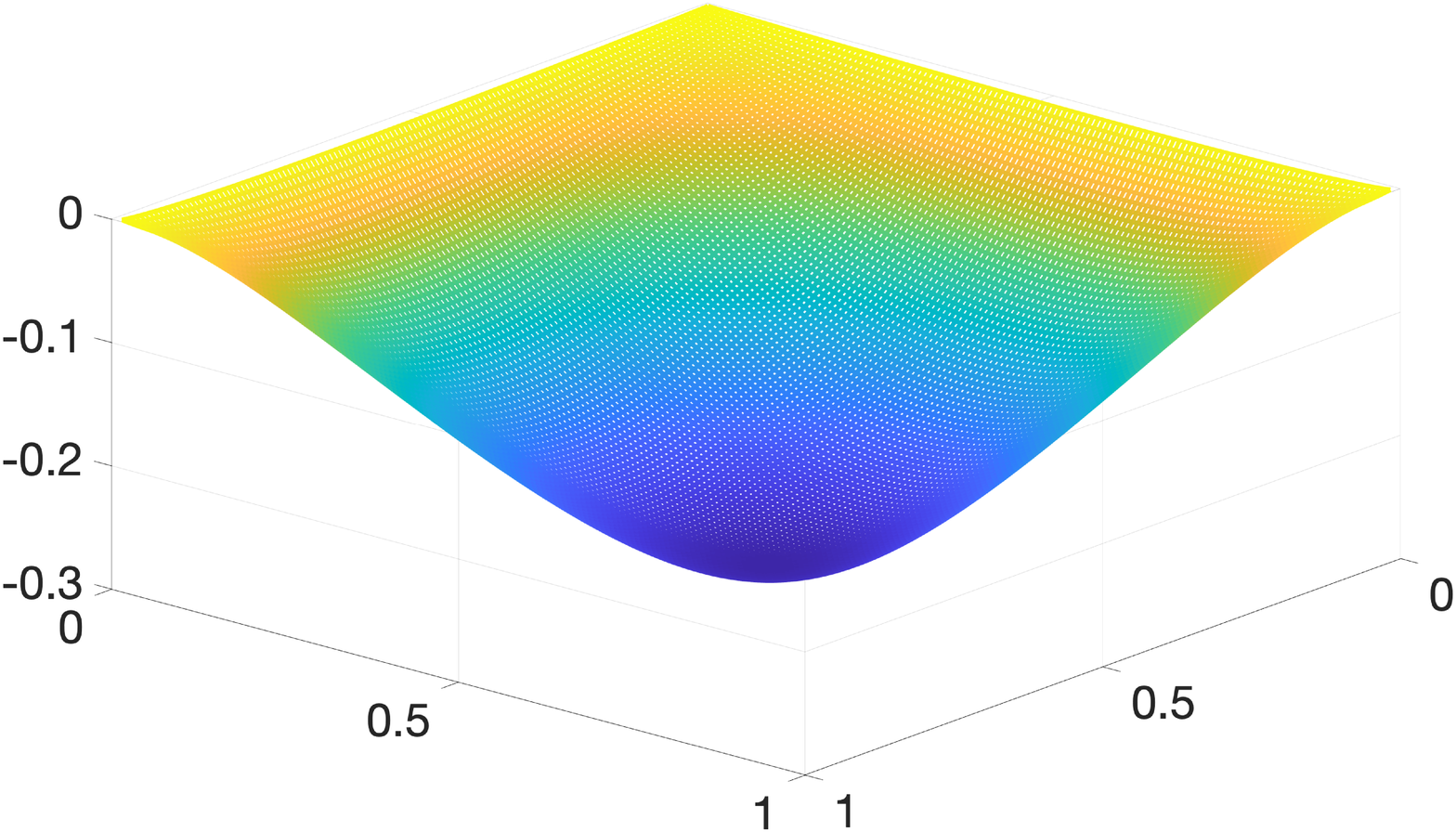}
	\end{tabular}
	\caption[Sub-optimals for MOREBV and $n=5$]{Some suboptimal solutions to the problem `MOREBV' \cref{eq:morebv} for $n=5$.
	} \label{fig:suboptimalsn5}
\end{figure}

\end{document}